\documentclass[a4paper,twoside]{amsart}

\usepackage{a4wide}
\usepackage{amsthm}
\usepackage{amsmath}
\usepackage{amsfonts}
\usepackage{amssymb}
\usepackage[all]{xy}
\usepackage{mathrsfs}
\usepackage{stmaryrd}
\usepackage{oldgerm}
\usepackage[english]{babel}
\usepackage[normalem]{ulem}
\usepackage{esvect}
\usepackage[unicode]{hyperref}
\usepackage{bookmark}
\usepackage[utf8]{inputenc}
\usepackage{breakurl}
\usepackage{xcolor}
\usepackage{comment}

\newtheorem{theo}[subsection]{Theorem}
\newtheorem{prop}[subsection]{Proposition}
\newtheorem{coro}[subsection]{Corollary}
\newtheorem{lemm}[subsection]{Lemma}
\theoremstyle{definition}
\newtheorem*{defi}{Definition}
\newtheorem*{nota}{Notation}
\newtheorem{rema}[subsection]{Remark}

\def\boxit#1#2{\setbox1=\hbox{\kern#1{#2}\kern#1}%
\dimen1=\ht1 \advance\dimen1 by #1
\dimen2=\dp1 \advance\dimen2 by #1
\setbox1=\hbox{\vrule height\dimen1 depth\dimen2\box1\vrule}%
\setbox1=\vbox{\hrule\box1\hrule}%
\advance\dimen1 by .4pt \ht1=\dimen1
\advance\dimen2 by .4pt \dp1=\dimen2 \box1\relax}

\DeclareMathOperator{\Mod}{\mathbf{Mod}} 
\DeclareMathOperator{\Rep}{\mathbf{Rep}} 

\DeclareMathOperator{\Id}{\mathsf{Id}} 
\DeclareMathOperator{\Hom}{\mathsf{Hom}}
\DeclareMathOperator{\End}{\mathsf{End}}

\DeclareMathOperator{\Ext}{\mathsf{Ext}}
\DeclareMathOperator{\Tor}{\mathsf{Tor}}
\DeclareMathOperator{\Ker}{\mathsf{Ker}}
\DeclareMathOperator{\im}{\mathsf{Im}}

\DeclareMathOperator{\Spd}{\mathsf{Spd}}

\DeclareMathOperator{\et}{\acute{e}t}

\DeclareMathOperator{\cris}{cris}

\DeclareMathOperator{\st}{st} 

\DeclareMathOperator{\dR}{dR}
\DeclareMathOperator{\HT}{HT}

\renewcommand{\H}{\mathsf{H}}

\DeclareMathOperator{\dd}{d} 
\DeclareMathOperator{\rg}{\mathsf{rk}} 
\DeclareMathOperator{\Sym}{\mathsf{Sym}} 

\DeclareMathOperator{\Frac}{\mathsf{Frac}} 
\DeclareMathOperator{\Gal}{\mathsf{Gal}} 
\DeclareMathOperator{\Tr}{\mathsf{Tr}} 
\DeclareMathOperator{\tr}{\mathsf{Tr}} 

\DeclareMathOperator{\Fil}{\mathsf{Fil}}
\DeclareMathOperator{\gr}{\mathsf{gr}}

\DeclareMathOperator{\Mat}{\mathsf{M}}
\DeclareMathOperator{\GL}{\mathsf{GL}}

\DeclareMathOperator{\I}{I} 

\DeclareMathOperator{\FF}{\mathbf{F}} 
\DeclareMathOperator{\NN}{\mathbf{N}} 
\DeclareMathOperator{\ZZ}{\mathbf{Z}} 
\DeclareMathOperator{\QQ}{\mathbf{Q}} 
\DeclareMathOperator{\RR}{\mathbf{R}} 

\renewcommand{\AA}{\mathbf{A}}
\DeclareMathOperator{\EE}{\mathbf{E}}
\DeclareMathOperator{\BB}{\mathbf{B}}

\DeclareMathOperator{\W}{\mathsf{W}} 

\DeclareMathOperator{\isomto}{\overset{\sim}{\to}}

\DeclareMathOperator{\A}{\mathsf{A}}
\DeclareMathOperator{\B}{\mathsf{B}}

\DeclareMathOperator{\D}{\mathsf{D}}

\renewcommand{\L}{\mathsf{L}}
\renewcommand{\l}{\mathsf{l}}

\DeclareMathOperator{\free}{\mathsf{f}}
\DeclareMathOperator{\pfree}{\mathsf{pf}}
\DeclareMathOperator{\pr}{\mathsf{pr}}

\DeclareMathOperator{\Sen}{\mathsf{Sen}}
\DeclareMathOperator{\dif}{\mathsf{dif}}
\DeclareMathOperator{\reg}{\mathsf{reg}}
\DeclareMathOperator{\ur}{\mathsf{ur}}

\newcommand{\cf}{\textit{cf }}
\newcommand{\ie}{\textit{i.e. }}

\newcommand{\abs}[1]{\lvert#1\rvert}

\newcommand{\eq}[1][r]
 {\ar@<-3pt>@{-}[#1]
  \ar@<-1pt>@{}[#1]|<{}="gauche"
  \ar@<+0pt>@{}[#1]|-{}="milieu"
  \ar@<+1pt>@{}[#1]|>{}="droite"
  \ar@/^2pt/@{-}"gauche";"milieu"
  \ar@/_2pt/@{-}"milieu";"droite"}

\newcommand{\incl}[1][r]
  {\ar@<-0.2pc>@{^(-}[#1] \ar@<+0.2pc>@{-}[#1]}

\newcommand{\immouv}[1][r]
   {\ar@{}[#1] |*[o][F]{\hbox{%
         \vrule width 1.5mm height 0pt depth 0pt%
         \vrule width 0pt height .75mm depth .75mm%
         }}
     \ar@{^{(}->}[#1]}

\def \ooverline #1#2#3%
{\mkern #1mu \overline{\mkern -#1mu #3 \mkern -#2mu }\mkern #2mu }

\def \Kbar {\ooverline40K{\mkern 1mu}{}}

\def \uunderline #1#2#3%
{\mkern #1mu \underline{\mkern -#1mu #3 \mkern -#2mu }\mkern #2mu }

\title[Multivariable de Rham representations]{Multivariable de Rham representations, Sen theory and $p$-adic differential equations}%
\author[O. Brinon \& B. Chiarellotto \& N. Mazzari]{O. Brinon \& B. Chiarellotto \& N. Mazzari}%
\address{IMB, Universit\'e de Bordeaux, 351, cours de la Lib\'eration, 33405 Talence, France}%
\address{Dipartimento di Matematica "Tullio Levi-Civita", Università degli Studi di Padova, Via Trieste, 63 - 35121 Padova}%
\email{olivier.brinon@math.u-bordeaux.fr}
\email{chiarbru@math.unipd.it}
\email{mazzari@math.unipd.it}

\begin{document}

\maketitle

\begin{abstract}
Let $K$ be a complete valued field extension of $\QQ_p$ with perfect residue field. We consider $p$-adic representations of a finite product $G_{K,\Delta}=G_K^\Delta$ of the absolute Galois group $G_K$ of $K$. This product appears as the fundamental group of a product of diamonds. We develop the corresponding $p$-adic Hodge theory by constructing analogues of the classical period rings $\B_{\dR}$ and $\B_{\HT}$, and multivariable Sen theory. In particular, we associate to any $p$-adic representation $V$ of $G_{K,\Delta}$ an integrable $p$-adic differential system in several variables $\D_{\dif}(V)$. We prove that this system is trivial if and only if the representation $V$ is de Rham. Finally, we relate this differential system  to the multivariable overconvergent $(\varphi,\Gamma)$-module of $V$ constructed by Pal and Z\'abr\'adi in \cite{PZ}, along classical Berger's construction \cite{BE}.
\end{abstract}

\begin{center}
\today
\end{center}

\setcounter{tocdepth}{1}
\tableofcontents

\section{Introduction}

Let $K$ be a finite extension of $\QQ_p$, $G_K$ its absolute Galois group, and $\Delta$ a finite set. After the work of Scholze and Weinstein \cite{W}, \cite{SW}, the finite product $G_{K,\Delta}:=G_K^\Delta$ can be understood as the fundamental group of a diamond: the product $\times_\Delta\Spd\QQ_p=\Spd\QQ_p\times\cdots\times\Spd\QQ_p $ (this diamond is \emph{not} associated to a perfectoid space). It is then natural to consider $p$-adic representations for this fundamental group, viewed as coefficients for the diamond  $\times_\Delta\Spd\QQ_p$. Hereafter, we work in a slightly more general context: we assume that $K$ is a complete discretely valued extension of $\QQ_p$ with perfect residue field.

Of course the study of such representations can be considered in the {\it classical} framework of $p$-adic Hodge theory as developed after the work of Fontaine (\cf \cite{Font04}, \cite{Sen}, \cite{Bri2006}), \ie in terms of period rings and $(\varphi,\Gamma)$-modules. This second approach has been pursued in recent
works by Z\'abr\'adi, Pal, Kedlaya and Carter (\cite{Z1}, \cite{Z2}, \cite{PZ} and \cite{CKZ}) in terms of multivariable (multivariate) $(\varphi,\Gamma)$-modules associated to $p$-adic representations of $G_{K, \Delta}$.

In this article, we develop a multivariable Sen theory in this framework, and construct multivariable
(multivariate) $p$-adic period rings $\B_{\dR,\Delta}$ and $\B_{\HT,\Delta}$. To any $p$-adic representation $V$ of $G_{K,\Delta}$, we associate an integrable differential system $\D_{\dif}(V)$ in several ($=\#\Delta$) variables. We prove that this system is trivial (\ie has a full set of solutions) if and only if the $G_{K, \Delta}$-representation $V$ is de Rham. Moreover, we relate the differential module $\D_{\dif}(V)$ with overconvergent $(\varphi,\Gamma)$-module arising from Pal-Z\'abr\'adi theory (\cf \cite{PZ}).

\medskip

Before giving a precise description of the content of this article, we make some remarks and thoughts for future developments. First of all, we note that the theory we develop here does not fit in the framework of relative $p$-adic Hodge theory as studied by Andreatta, Brinon (\cite{And06}, \cite{AB}, \cite{AB2}, \cite{B1}). In fact the geometric base for our objects is merely a finite discrete space (\cf remark \ref{remaRR}). Secondly, this article should be seen as a first step towards the introduction of the multivariable periods rings $\B_{\cris}$ and $\B_{\st}$, and eventually, a step in the direction of full analogue of Berger's results via the theory of $p$-adic differential systems in several variables (as was foreseen in \cite{PZ}).

\medskip

We now describe more precisely the content of this work. In the second section we fix some notations and recall useful results. In particular we recall classical Sen theory of $C$-representations, where $C$ is the completion of an algebraic closure of $K$, and introduce the completion $C_\Delta$ of the tensor product of the $\#\Delta$-fold tensor product of $C$ over the maximal unramified subextension of $K$. In the third section we study free $C_\Delta$-representations of finite rank of $G_{K,\Delta}$: more precisely, we develop an analogue of Sen theory in this context. Classically (\ie when $\#\Delta=1$) Sen theory is an equivalence of categories between $C$-representations of $G_K$ and $K_\infty$-representations of $\Gamma_K=\Gal(K_\infty/K)$, where $K_\infty$ is the cyclotomic extension of $K$. Our result in the multivariable context is theorem \ref{theoequivSenCDelta}: there is an equivalence of categories between that of free $C_\Delta$-representations of finite rank of $G_{K,\Delta}$ and that of free
$K_{\Delta,\infty}$-representation of finite rank of $\Gamma_{K,\Delta}$, where $K_{\Delta,\infty}$ is the (non completed) $\#\Delta$-fold tensor product of $K_\infty$ over the maximal unramified subextension of $K$. To the latter we can associate generalized Sen operators (describing the infinitesimal action of $\Gamma_{K,\Delta}$) and develop a Hodge-Tate theory (\cf corollaries \ref{coroSenop2} and \ref{coroSenop3}). In the fourth section we introduce the period rings $\B_{\dR,\Delta}$ and  $\B_{\HT,\Delta}$ and the corresponding de Rham and Hodge-Tate representations: in particular we show that there are functors $\D_{\dR}$ and $\D_{\HT}$ having the expected properties (propositions \ref{propalphainj}, \ref{propdRHT} and \ref{propdRHT2}), in particular that being de Rham implies to be Hodge-Tate. In the fifth section we prove the multivariable analogue of the work of Fontaine in \cite{Font04}: namely Sen theory for $\B_{\dR,\Delta}$-representations. To do this we follow \cite{AB2}: the central result (theorem \ref{theoSenBdR}) is that the category of free $\B_{\dR,\Delta}^+$-representations of finite rank of $G_{K,\Delta}$ is equivalent to that of free $\l_{\dR,\Delta}^+=K_{\Delta,\infty}[\![t_\alpha]\!]_{\alpha\in\Delta}$-representations of finite rank of $\Gamma_{K,\Delta}$ (where $t_\alpha$ is a $p$-adic $2i\pi$ corresponding to the action of the factor of index $\alpha$ in $G_{K,\Delta}$). By inverting $\prod\limits_{\alpha\in\Delta}t_\alpha$, we deduce the analogue for $\B_{\dR, \Delta}$ (theorem \ref{theoSenBdR2}). The upshot is that we can associate a free module $\D_{\dif}(V)$ with a regular, integrable connection in $\#\Delta$ variables with coefficients in $\l_{\dR,\Delta}=\l^+_{\dR, \Delta}\big[\frac{1}{t_\alpha}\big]_{\alpha\in\Delta}$ to any $p$-adic representation $V$ of $G_{K, \Delta}$. This is the analogue to that introduced by Fontaine in \cite{Font04}, and used by Berger in \cite{BE}. In particular, we show that a $p$-adic representation $V$ of $G_{K,\Delta}$ is de Rham if and only if the associated module with connection $\D_{\dif}(V)$ is trivial (proposition \ref{propdifdR}), and relate our construction to that of overconvergent $(\varphi,\Gamma)$-modules developed by Pal-Z\'abr\'adi (\cf \cite{PZ}) and Carter-Kedlaya-Z\'abr\'adi (\cf \cite{CKZ}) by an analogue of \cite[Corollaire 5.8]{BE} (\cf theorem \ref{theocompphiGammaDdif}).

\begin{rema}
There is little doubt that a general Tate-Sen formalism (such as that of \cite{AB}) does exist in the multivariable case, and that could be applied to families of multivariable representations (as for \cite[\S3]{BC} and \cite[Proposition 5.2.1]{BC}). This said, we proceed here with Tate-Sen descent by hand (this already contains most of the necessary ideas).
\end{rema}

\medskip

B. Chiarellotto and N. Mazzari are supported by the grant MIUR-PRIN 2017 “Geometric, algebraic and analytic methods in arithmetic”. O. Brinon thanks the Department of Mathematics of the University of Padua for organizing his visits during a pandemic lull.

\section{Notations}

Let $K$ be a complete discrete valuation field of characteristic $0$, with perfect residue field $k$ of characteristic $p>0$. Fix an algebraic closure $\Kbar$ of $K$ and let $G_K=\Gal(\Kbar/K)$. Denote by $v$ the valuation on $K$ normalized by $v(p)=1$. It extends uniquely to a valuation of $\Kbar$ : let $C$ be the completion of the latter. If $F$ is a subextension of $C/K$, we will denote by $\mathcal{O}_F$ (resp. $\mathfrak{m}_F$) its ring of integers (resp. its maximal ideal). By continuity, the action of $G_K$ extends to $C$. Fix $\varepsilon=(\varepsilon^{(n)})_{n\in\NN}$ a compatible system of primitive $p^n$-th roots of the unity (\ie such that $\varepsilon^{(0)}=1$, $\varepsilon^{(1)}\neq1$ and $(\varepsilon^{(n+1)})^p=\varepsilon^{(n)}$ for all $n\in\NN$). For each $n\in\NN$, put $K_n=K(\varepsilon^{(n)})$ and let $K_\infty=\bigcup\limits_{n=0}^\infty K_n$ be the cyclotomic extension, and $L=\widehat{K_\infty}$ its completion with respect to $v$. Put $H_K=\Gal(\Kbar/K_\infty)$ and $\Gamma_K=\Gal(K_\infty/K)$. The cyclotomic character $\chi\colon\Gamma_K\to\ZZ_p^\times$ is characterized by $\gamma\big(\varepsilon^{(n)}\big)=\big(\varepsilon^{(n)}\big)^{\chi(\gamma)}$ for all $\gamma\in\Gamma_K$: it induces an continuous isomorphism between $\Gamma_K$ and an open subgroup of $\ZZ_p^\times$. We still denote $\chi$ the composite $G_K\to\Gamma_K\xrightarrow{\chi}\ZZ_p^\times$. In what follows, cohomology will always refer to \emph{continuous} cohomology.

\medskip

Using ramification estimates, Tate proved in \cite{Tate} that $\H^i(H_K,C)=\begin{cases}L &\text{ if }i=0\\ 0 &\text{ if }i>0\end{cases}$ and constructed the so-called Tate's normalized traces $(R_n\colon L\to K_n)_{n\geq n_K}$ (for some integer $n_K\in\NN$), that he used to show that $\dim_K\H^i(\Gamma_K,C^{H_K})=\begin{cases}1 &\text{ if }i\in\{0,1\}\\ 0 &\text{ if }i>1\end{cases}$, so that $\dim_K\H^i(G_K,C)=\begin{cases}1 &\text{ if }i\in\{0,1\}\\ 0 &\text{ if }i>1\end{cases}$.

Recall that Tate's normalized trace map $R_n\colon L\to K_n$ induces the map $x\mapsto\frac{1}{p^{m-n}}\tr_{K_m/K_n}(x)$ on $K_m$ for all $m\geq n_K$.

\begin{prop}\label{propRn}
(\cf \cite[\S3]{Tate}, \cite[\S3.1 \& Proposition 4.1.1]{BC} and \cite[Proposition 14.1.6]{Hawaii}) These maps have the following properties:
\begin{itemize}
\item[(i)] $R_n$ is a $K_n$-linear projector onto $K_n$: put $X_n=\Ker(\Id-R_n)\subset L$;
\item[(ii)] $R_n$ commutes to the action of $\Gamma_K$;
\item[(iii)] $(\forall c_2\in\RR_{>0})\,(\forall x\in L)\,v_p(R_n (x))\geq v_p(x)-c_2$ (in particular $R_n$ is continuous);
\item[(iv)] $(\forall x\in L)\,\lim\limits_{n\to\infty}R_n(x)=x$;
\item[(v)] for all $c_3>\frac{1}{p-1}$, there exists $n_K^\prime\geq n_K$ such that for all $n\geq n_K^\prime$ and $\gamma\in\Gamma_K$ such that $v_p(1-\chi(\gamma))\leq n_K^\prime$, then $\gamma-1$ is invertible on $X_n$, and for all $x\in X_n$, we have $v_p((\gamma-1)^{-1}(x))\geq v_p(x)-c_3$ (in particular, $\gamma-1$ induces an homeomorphism from $X_n$ to itself).
\end{itemize}
\end{prop}
In what follows, we will use the previous properties with $c_2=c_3=1$. This is certainly not optimal, but for technical reasons, it is much more convenient to make computations and work with $\mathcal{O}_{C_\Delta}/(p)$ (\cf \textit{infra}) rather than with quotients by elements of non integral valuation. 

\medskip

Let $ d\in\NN_{>0}$. Based on the work of Tate, Sen showed in \cite{Sen} that the set $\H^1(H_K,\GL_d(C))$ is trivial, so that the inflation map
$$\H^1(\Gamma_K,\GL_d(L))\to\H^1(G_K,\GL_d(C))$$
is bijective. He also proved that the natural map
$$\H^1(\Gamma_K,\GL_d(K_\infty))\to\H^1(\Gamma_K,\GL_d(L))$$
is bijective. This means that if $W$ is a $C$-representation of $G_K$ (\cf \S \ref{SenC}), there exists a $K_\infty$-representation $W_\infty$ of $\Gamma_K$ such that $W\simeq C\otimes_{K_\infty}W_\infty$ as $C[G_K]$-modules.

\medskip

\textcolor{red}{As mentioned in the introduction, the first aim of this note is to generalize these results to the case where $G_K$ is replaced by a finite power of $G_K$ (along the lines of \cite{PZ}).}

\medskip

Let $\Delta$ be a finite set, and put $\delta=\#\Delta$. Put $G_{K,\Delta}=\prod\limits_{\alpha\in\Delta}G_K$, $H_{K,\Delta}=\prod\limits_{\alpha\in\Delta}H_K$, and $\Gamma_{K,\Delta}=\prod\limits_{\alpha\in\Delta}\Gamma_K$. We have an exact sequence
$$1\to H_{K,\Delta}\to G_{K,\Delta}\xrightarrow{\chi_\Delta}\Gamma_{K,\Delta}\to1.$$
The morphism $\chi_\Delta=\prod\limits_{\alpha\in\Delta}\chi$ identifies $\Gamma_{K,\Delta}$ with an open subgroup of $(\ZZ_p^\times)^\Delta$. Il $\alpha\in\Delta$, we denote by $G_{K,\alpha}$ the image of the group homomorphim $\iota_\alpha\colon G_K\to G_{K,\Delta}$ that maps $g$ to the element whose component of index $\alpha$ is $g$ and the others are $\Id_{\Kbar}$. The groups $H_{K,\alpha}$ and $\Gamma_{K,\alpha}$ are defined similarly.

\medskip

\begin{nota}
(1) If $\underline{n}=(n_\alpha)_{\alpha\in\Delta}\in\ZZ^\Delta$, we define a character $\chi_\Delta^{\underline{n}}\colon G_{K,\Delta}\to\ZZ_p^\times$ by
$$\chi_\Delta^{\underline{n}}\big((g_\alpha)_{\alpha\in\Delta}\big)=\prod\limits_{\alpha\in\Delta}\chi(g_\alpha)^{n_\alpha}.$$
This provides a $\ZZ_p$-representation $\ZZ_p(\underline{n})$ of $G_{K,\Delta}$. More generally, if $M$ is a $\ZZ_p$-module endowed with an action of $G_{K,\Delta}$, we put $M(\underline{n})=M\otimes_{\ZZ_p}\ZZ_p(\underline{n})$, endowed with the diagonal action of $G_{K,\Delta}$.

\noindent
(2) If $F$ is a closed subextension of $C/F_0$ (where $F_0:=\W(k)\big[\frac{1}{p}\big]$), let $\mathcal{O}_{F_\Delta}$ be the $p$-adic completion of the tensor product $\mathcal{O}_F\otimes_{\W(k)}\cdots\otimes_{\W(k)}\mathcal{O}_F$ (where the copies of $\mathcal{O}_F$ are indexed by $\Delta$), and $F_\Delta=\mathcal{O}_{F_\Delta}\big[\frac{1}{p}\big]$. Observe that when $F/F_0$ is finite, $\mathcal{O}_F$ is a free $\W(k)$-module of finite rank, so that the tensor product $\mathcal{O}_F\otimes_{\W(k)}\cdots\otimes_{\W(k)}\mathcal{O}_F$ is $p$-adically separated and complete, so that $F_\Delta$ is nothing but the tensor product $F^{\otimes\Delta}=F\otimes_{F_0}\cdots\otimes_{F_0}F$ (where the copies of $F$ are indexed by $\Delta$).
\end{nota}

\begin{rema}
(1) The ring $\mathcal{O}_{F_\Delta}$ depends on $k$. This dependence is not indicated in the notation so as not to make it heavier.

\noindent
(2) When it is not mentioned, tensor products are taken over $\W(k)$.
\end{rema}

The group $G_{K_\Delta}$ naturally acts on $\mathcal{O}_{C_\Delta}$ and $C_\Delta$. Note also that $\mathcal{O}_{L_\Delta}$ is naturally a $\mathcal{O}_{K_\infty}^{\otimes\Delta}$-algebra.


\section{Multivariable classical Sen theory}

\subsection{The cohomology of \texorpdfstring{$C_\Delta$}{C\unichar{"005F}\unichar{"0394}}}\label{sectcohoCDelta}

\begin{lemm}\label{lemmcohoHC}
Let $r\in\NN_{>0}$. Then the cokernel of $\mathcal{O}_{L_\Delta}/(p^r)\to\H^0(H_{K,\Delta},\mathcal{O}_{C_\Delta}/(p^r))$ is killed by $\mathfrak{m}_{K_\infty}^{\otimes\Delta}$. If $i>0$, the group $\H^i(H_{K,\Delta},\mathcal{O}_{C_\Delta}/(p^r))$ is killed by $\mathfrak{m}_{K_\infty}^{\otimes\Delta}$.
\end{lemm}

\begin{proof}
If $\Delta^\prime\subset\Delta$, we denote by $\mathcal{O}_{C_{\Delta,\Delta^\prime}}$ the $p$-adic completion of $\mathcal{O}_{C_{\Delta^\prime}}\otimes_{\W(k)}\mathcal{O}_{L_{\Delta\setminus\Delta^\prime}}$. In particular, we have $\mathcal{O}_{C_{\Delta,\varnothing}}=\mathcal{O}_{L_\Delta}$ and $\mathcal{O}_{C_{\Delta,\Delta}}=\mathcal{O}_{C_\Delta}$.

\medskip

We proceed componentwise: let $\Delta^\prime\subset\Delta$ and $\alpha\in\Delta\setminus\Delta^\prime$, and consider the action of $H_{K,\alpha}$ on $\mathcal{O}_{C_{\Delta,\Delta^\prime\cup\{\alpha\}}}/(p^r)$. The topology on the latter is discrete: we have
$$\H^i(H_{K,\alpha},\mathcal{O}_{C_{\Delta,\Delta^\prime\cup\{\alpha\}}}/(p^r))=\varinjlim\limits_F\H^i(H_{K,\alpha},\mathcal{O}_{F,\alpha}\otimes_{\W(k)}\mathcal{O}_{C_{\Delta\setminus\{\alpha\},\Delta^\prime}}/(p^r))$$
where $F$ runs among the finite Galois subextensions of $\Kbar/K_\infty$. Recall that by \cite[\S3.2, Proposition 9]{Tate}, we have $\mathfrak{m}_{K_\infty}\subset\Tr_{F/K_\infty}(\mathcal{O}_F)$ for every such $F$. This implies that $\H^i(H_{K,\alpha},\mathcal{O}_{F,\alpha}\otimes_{\W(k)}\mathcal{O}_{C_{\Delta\setminus\{\alpha\},\Delta^\prime}}/(p^r))$ is killed by $\mathfrak{m}_{K_\infty,\alpha}$ for all $F$ (\cf \cite[Lemma 3.1]{Olsson}): so does $\H^i(H_{K,\alpha},\mathcal{O}_{C_{\Delta,\Delta^\prime\cup\{\alpha\}}}/(p^r))$ for $i>0$. If $x\in\H^0(H_{K,\alpha},\mathcal{O}_{F,\alpha}\otimes_{\W(k)}\mathcal{O}_{C_{\Delta\setminus\{\alpha\},\Delta^\prime}}/(p^r))$ and $\eta\in\mathfrak{m}_{K_\infty,\alpha}$, let $y\in\mathcal{O}_F$ such that $\Tr_{F/K_\infty}(y)=\eta$. Then $\eta x=\Tr_{F/K_\infty,\alpha}(xy)\in\mathcal{O}_{C_{\Delta,\Delta^\prime}}/(p^r)$, where $\Tr_{F/K_\infty,\alpha}\colon\mathcal{O}_{F,\alpha}\otimes_{\W(k)}\mathcal{O}_{C_{\Delta\setminus\{\alpha\},\Delta^\prime}}/(p^r))\to\mathcal{O}_{C_{\Delta,\Delta^\prime}}/(p^r)$ denotes the tensor product of $\Tr_{F/K_\infty}$ on the factor of index $\alpha$ with the identity on the other factors. 

\medskip

The lemma follows by applying the Hochschild-Serre spectral sequence finitely many times.
\end{proof}

\begin{theo}\label{theocohoHC}
We have $\H^i(H_{K,\Delta},C_\Delta)=\begin{cases}L_\Delta&\text{ if }i=0\\ 0&\text{ if }i>0\end{cases}$.
\end{theo}

\begin{proof}
By \cite[Proposition 2.7.4]{NSW}, there is a commutative diagram with exact rows
$$\xymatrix@C=15pt{0\ar[r] & \varprojlim\limits_r\mathcal{O}_{L_\Delta}/(p^r)\ar[d]\ar[r] & \prod\limits_{r=1}^\infty\mathcal{O}_{L_\Delta}/(p^r)\ar[d]\ar[r] & \prod\limits_{r=1}^\infty\mathcal{O}_{L_\Delta}/(p^r)\ar[d]\ar[r] & \varprojlim\limits_r{}^{(1)}\mathcal{O}_{L_\Delta}/(p^r)\ar[d]\ar[r] & 0\\
0\ar[r] & \varprojlim\limits_r(\mathcal{O}_{C_\Delta}/(p^r))^{H_{K,\Delta}}\ar[r] & \prod\limits_{r=1}^\infty(\mathcal{O}_{C_\Delta}/(p^r))^{H_{K,\Delta}}\ar[r] & \prod\limits_{r=1}^\infty(\mathcal{O}_{C_\Delta}/(p^r))^{H_{K,\Delta}}\ar[r] & \varprojlim\limits_r{}^{(1)}(\mathcal{O}_{C_\Delta}/(p^r))^{H_{K,\Delta}}\ar[r] & 0}$$
The first three vertical maps are injective and the cokernels of those in the middle are killed by $\mathfrak{m}_{K_\infty}^{\otimes\Delta}$. This implies that the cokernel of the first vertical map is killed by $\mathfrak{m}_{K_\infty}^{\otimes\Delta}$. This implies that the cokernel of composite map
$$\mathcal{O}_{L_\Delta}\hookrightarrow\mathcal{O}_{C_\Delta}^{H_{K,\Delta}}\to\varprojlim\limits_r(\mathcal{O}_{C_\Delta}/(p^r))^{H_{K,\Delta}}$$
is killed by $\mathfrak{m}_{K_\infty}^{\otimes\Delta}$, showing (by injectivity of the second map) that the cokernel of
$$\mathcal{O}_{L_\Delta}\to\H^0(H_{K,\Delta},\mathcal{O}_{C_\Delta})$$
is killed by $\mathfrak{m}_{K_\infty}^{\otimes\Delta}$. On the other hand, the inverse system $\{\mathcal{O}_{L_\Delta}/(p^r)\}_r$ has the Mittag-Leffler property: we have $\varprojlim\limits_r{}^{(1)}\mathcal{O}_{L_\Delta}/(p^r)=0$. This implies that $\varprojlim\limits_r{}^{(1)}(\mathcal{O}_{C_\Delta}/(p^r))^{H_{K,\Delta}}$ is killed by $\mathfrak{m}_{K_\infty}^{\otimes\Delta}$.

If $i>0$, we have an exact sequence
$$0\to\varprojlim\limits_r{}^{(1)}\H^{i-1}(H_{K,\Delta},\mathcal{O}_{C_\Delta}/(p^r))\to\H^i(H_{K,\Delta},\mathcal{O}_{C_\Delta})\to\varprojlim\limits_r\H^i(H_{K,\Delta},\mathcal{O}_{C_\Delta}/(p^r))\to0$$
(\cf \cite[Theorem 2.7.5]{NSW}). By lemma \ref{lemmcohoHC} (and what precedes when $i=1$) the modules $\varprojlim\limits_r{}^{(1)}\H^{i-1}(H_{K,\Delta},\mathcal{O}_{C_\Delta}/(p^r))$ and $\varprojlim\limits_r\H^i(H_{K,\Delta},\mathcal{O}_{C_\Delta}/(p^r))$ are killed by $\mathfrak{m}_{K_\infty}^{\otimes\Delta}$, implying that $\H^i(H_\Delta,\mathcal{O}_{C_\Delta})$ is killed by $\big(\mathfrak{m}_{K_\infty}^{\otimes\Delta}\big)^2=\mathfrak{m}_{K_\infty}^{\otimes\Delta}$.

The proposition follows by inverting $p$.
\end{proof}

\begin{nota}
If $\Delta^\prime\subset\Delta$ and $n\in\NN$, we denote by $\mathcal{O}_{L_{\Delta,\Delta^\prime,n}}$ the $p$-adic completion of $\mathcal{O}_{L_{\Delta\setminus\Delta^\prime}}\otimes_{\W(k)}\mathcal{O}_{K_{n,\Delta^\prime}}$. In particular, we have $\mathcal{O}_{L_{\Delta,\varnothing,n}}=\mathcal{O}_{K_n,\Delta}$ and $\mathcal{O}_{L_{\Delta,\Delta,n}}=\mathcal{O}_{L_\Delta}$. Put $L_{\Delta,\Delta^\prime,n}=\mathcal{O}_{L_{\Delta,\Delta^\prime,n}}\big[\frac{1}{p}\big]$. Let $n\geq n_K$: the map $R_n$ induces a continuous map $\mathcal{O}_L\to\frac{1}{p}\mathcal{O}_{K_n}$. If $\alpha\in\Delta$, denote by $R_{n,\alpha}\colon L_\Delta\to L_{\Delta,\Delta\setminus\{\alpha\},n}$ be the map induced by the tensor product of Tate's normalized trace $R_n\colon L\to K_n$ (\cf \cite[\S3]{Tate}) on the factor of index $\alpha$ with the identity on the other factors (this makes sense by the continuity of $R_n$). This defines a continuous $L_{\Delta,\{\alpha\},n}$-linear projector that maps $\mathcal{O}_{L_{\Delta,\Delta^\prime,n}}$ into $\frac{1}{p}\mathcal{O}_{L_{\Delta,\Delta^\prime\cup\{\alpha\},n}}$ (\cf proposition \ref{propRn}). We also put $\mathcal{O}_{L_{\Delta,\Delta^\prime,\infty}}=\bigcup\limits_{n=0}^\infty\mathcal{O}_{L_{\Delta,\Delta^\prime,n}}$ and $L_{\Delta,\Delta^\prime,\infty}=\bigcup\limits_{n=0}^\infty L_{\Delta,\Delta^\prime,n}$.
\end{nota}

\begin{theo}\label{theocohoGC}
We have $\H^i(G_{K,\Delta},C_\Delta)\simeq\H^i(\Gamma_{K,\Delta},L_\Delta)=\bigwedge^i\Big(\bigoplus\limits_{\alpha\in\Delta}K_\Delta\log(\chi_\alpha)\Big)$.
\end{theo}

\begin{proof}
The first isomorphism follows from the inflation-restriction exact sequence. For all $n\in\NN$, the map $K_\Delta\to K_{n,\Delta}$ is finite \'etale, it is enough to prove the statement replacing $K$ by $K_n$ with $n\geq n_K$ (\cf above). Then the Horschild-Serre spectral sequence reduces the proof of the second isomorphism to the equalities
$$\H^i(\Gamma_{K_{n,\alpha}},L_{\Delta,\Delta^\prime,n})=\begin{cases}
L_{\Delta,\Delta^\prime\cup\{\alpha\},n}&\text{ if }i=0\\
L_{\Delta,\Delta^\prime\cup\{\alpha\},n}\log(\chi_\alpha)&\text{ if }i=1\\
0&\text{ if }i>1\end{cases}$$
for all $\alpha\in\Delta$ and $\Delta^\prime\subset\Delta\setminus\{\alpha\}$. Working modulo $p^r$ as we did in the proof of lemma \ref{lemmcohoHC} and using the maps $R_{m,\alpha}$ for $m\geq n$, we deduce that the cokernels of the maps
\begin{align*}
\mathcal{O}_{L_{\Delta,\Delta^\prime,n}}/(p^r) &\to \H^0(\Gamma_{K_{n,\alpha}},L_{\Delta,\Delta^\prime,n}/(p^r))\\ 
\mathcal{O}_{L_{\Delta,\Delta^\prime,n}}/(p^r)\log(\chi_\alpha) &\to \H^1(\Gamma_{K_{n,\alpha}},L_{\Delta,\Delta^\prime,n}/(p^r))
\end{align*}
are killed by $p$, as do the elements of $\H^i(\Gamma_{K_{n,\alpha}},L_{\Delta,\Delta^\prime,n}/(p^r))$ if $i>1$. Then we can pass to the limit arguing as in the proof of proposition \ref{theocohoHC} to conclude that the cokernels of the maps
\begin{align*}
\mathcal{O}_{L_{\Delta,\Delta^\prime,n}} &\to \H^0(\Gamma_{K_{n,\alpha}},L_{\Delta,\Delta^\prime,n})\\ 
\mathcal{O}_{L_{\Delta,\Delta^\prime,n}}\log(\chi_\alpha) &\to \H^1(\Gamma_{K_{n,\alpha}},L_{\Delta,\Delta^\prime,n})
\end{align*}
and the modules $\H^i(\Gamma_{K_{n,\alpha}},L_{\Delta,\Delta^\prime,n})$ with $i>1$ are killed by $p$. The theorem follows by inverting $p$.
\end{proof}

\begin{theo}\label{theocohoGCn}
If $n\in\NN$ and $\underline{n}\in\ZZ^\Delta\setminus\{\underline{0}\}$, we have $\H^i(G_{K,\Delta},C_\Delta(\underline{n}))=0$ for all $i\in\NN$.
\end{theo}

\begin{proof}
Again, using the inflation-restriction exact sequence and the Horschild-Serre spectral sequence, we are reduced to show that if $\alpha\in\Delta$ is such that $n_\alpha\neq0$, then the cohomology groups $\H^i(\Gamma_{K,\alpha},L_\Delta(\underline{n}))$ all vanish. Let $\gamma\in\Gamma_{K,\alpha}$ be such that the closure $\overline{\langle\gamma\rangle}$ has finite index in $\Gamma_{K,\alpha}$: it is enough to show that $\H^i(\overline{\langle\gamma_\alpha\rangle},L_\Delta(\underline{n}))=0$ for all $i\in\NN$, where $\gamma_\alpha\in\Gamma_{K,\Delta}$ is the element whose components are the identity except that of index $\alpha$, which is $\gamma$. As these cohomology groups are those of the complex $L_\Delta(\underline{n})\xrightarrow{\gamma_\alpha-1}L_\Delta(\underline{n})$ (concentrated in degrees $0$ and $1$), this is equivalent to showing that $\gamma_\alpha-1$ is bijective on $L_\Delta(\underline{n})$, \ie that it is injective with cokernel killed by $p^r$ for some $r\in\NN$ on $\mathcal{O}_{L_\Delta}(\underline{n})$. This follows from the fact that $\gamma-1$ is injective with cokernel killed by $p^r$ for some $r\in\NN$ on $\mathcal{O}_L(n_\alpha)$, since it is bijective with continuous inverse on $L$ (\cf proposition \ref{propRn} (v)).
\end{proof}


\subsection{Sen theory for \texorpdfstring{$C_\Delta$}{C\unichar{"005F}\unichar{"0394}}-representations}\label{SenC}

We fix terminology and notation that will be used hereafter.

\begin{defi}\label{defiBrepres}
Let $G$ be a topological group and $B$ a topological ring endowed with a continuous action of $G$. A \emph{$B$-representation} of $G$ is a topological module of finite type $W$ endowed with a continuous and semi-linear action of $G$, \ie such that $g(w_1+bw_2)=g(w_1)+g(b)g(w_2)$ for all $b\in B$, $w_1,w_2\in W$ and $g\in G$. We say that $W$ is free (resp. projective) of rank $d$ when the underlying $B$-module is. We denote by $\Rep_B(G)$ (resp. $\Rep_B^{\free}(G)$, resp. $\Rep_B^{\pr}(G)$) the category of $B$-representations with $G$-equivariant maps (resp. the full subcategory of free, resp. projective $B$-representations of finite rank).
\end{defi}

\begin{rema}
If $W$ is a free $B$-representation of rank $d$ of $G$, and $\mathfrak{B}$ is a basis of $W$ over $B$, we can denote by $U_g\in\Mat_d(B)$ the matrix of $g$ acting on $W$ in the basis $\mathfrak{B}$. Then $U_g\in\GL_d(B)$ for all $g\in G$, and the map $g\mapsto U_g$ is a continuous $1$-cocycle $G\to\GL_d(B)$. Conversely, the data of such a cocycle endows $B^d$ with a $B$-representation structure. Moreover, changing the basis precisely amounts to replace the cocycle by a cohomologous one. This means that isomorphism classes of free $B$-representations of rank $d$ are in bijection with the continuous cohomology set $\H^1(G,\GL_d(B))$.
\end{rema}

\medskip

Fix $d\in\NN_{>0}$. Let $H_0\leq H_K$ be an open normal subgroup, and put $H_\Delta=\prod\limits_{\alpha\in\Delta}H_0$. If $\Delta^\prime\subset\Delta$, let $H_{\Delta,\Delta^\prime}$ be the subgroup of $H_{K,\Delta}$ generated by the subgroups $\iota_\alpha(H_0)$ for $\alpha\in\Delta^\prime$: we have $H_{\Delta,\varnothing}=\{1\}$ and $H_{\Delta,\Delta}=H_\Delta$. Note that although the groups $H_\Delta$ and $H_{\Delta,\Delta^\prime}$ depend on $H_0$, we do not indicate this dependency in order not to make the notations too heavy. 

\begin{lemm}\label{lemmSenH}
(\cf \cite[Lemma 1]{Sen} and \cite[Lemme 2.1]{AB}) Let $U\colon H_\Delta\to\GL_d(C_\Delta)$ be a continuous cocycle. Let $\Delta^\prime\subset\Delta$, $\alpha\in\Delta\setminus\Delta^\prime$ and $m\in\NN_{\geq2}$ be such that $U_h=\I_d$ for all $h\in H_{\Delta,\Delta^\prime}$ and $U_h\in\I_d+p^m\Mat_d\big(\mathcal{O}_{C_\Delta}^{H_{\Delta^\prime}}\big)$ for all $h\in H_\Delta$. Then there exists $B_0\in\I_d+p^{m-1}\Mat_d\big(\mathcal{O}_{C_\Delta}^{H_{\Delta^\prime}}\big)$ such that $B_0^{-1}U_hh(B_0)\in\I_d+p^{m+1}\Mat_d\big(\mathcal{O}_{C_\Delta}^{H_{\Delta,\Delta^\prime}}\big)$ for all $h\in\iota_\alpha(H_0)$.
\end{lemm}

\begin{proof}
By continuity of $U$, there exists an open normal subgroup $H_1\leq H_0$ such that $U_h\in\I_d+p^{m+2}\Mat_d\big(\mathcal{O}_{C_\Delta}^{H_{\Delta,\Delta^\prime}}\big)$ for all $h\in\iota_\alpha(H_1)$. Let $T$ be a complete set of representatives of $H_0/H_1\simeq\iota_\alpha(H_0)/\iota_\alpha(H_1)$: if $h\in\iota_\alpha(H_0)$ and $\tau\in T$, there exists unique $h^\prime\in\iota_\alpha(H_1)$ and $\tau^\prime\in T$ such that $h\tau=\tau^\prime h^\prime$: we have
$$U_{h\tau}=U_{\tau^\prime h^\prime}=U_{\tau^\prime}\tau^\prime(U_{h^\prime})\in U_{\tau^\prime}+p^{m+2}\Mat_d\big(\mathcal{O}_{C_\Delta}^{H_{\Delta,\Delta^\prime}}\big).$$
By \cite[\S3.2, Proposition 9]{Tate}, we can choose $c\in\mathcal{O}_C^{H_1}$ such that $\Tr_{H_0/H_1}(c):=\sum\limits_{\tau\in T}\tau(c)=p$. Put
$$B_0=\tfrac{1}{p}\sum\limits_{\tau\in T}\tau(c_\alpha)U_\tau\in\Mat_d\big(C_\Delta^{H_{\Delta,\Delta^\prime}}\big)$$
(where $c_\alpha\in\mathcal{O}_{C_\Delta}$ is the image of the tensor $1\otimes\cdots\otimes1\otimes c\otimes1\otimes\cdots\otimes1$, with $c$ on the factor of index $\alpha$). As $U_\tau\in\I_d+p^m\Mat_d\big(\mathcal{O}_{C_\Delta}^{H_{\Delta,\Delta^\prime}}\big)$ for all $\tau\in T$, we have $B_0\in\I_d+p^{m-1}\Mat_d\big(\mathcal{O}_{C_\Delta}^{H_{\Delta,\Delta^\prime}}\big)$. This implies in particular that $B_0$ is invertible in $\I_d+p^{m-1}\Mat_d\big(\mathcal{O}_{C_\Delta}^{H_{\Delta,\Delta^\prime}}\big)$, with $B_0^{-1}=\sum\limits_{i=0}^\infty(\I_d-B_0)^i$. Moreover, if $h\in\iota_\alpha(H_0)$, we have
$$h(B_0)=\tfrac{1}{p}\sum\limits_{\tau\in T}h\tau(c_\alpha)h(U_\tau)=\tfrac{1}{p}\sum\limits_{\tau\in T}h\tau(c_\alpha)U_h^{-1}U_{h\tau}=\tfrac{1}{p}U_h^{-1}\sum\limits_{\tau\in T}\tau^\prime h^\prime(c_\alpha)U_{\tau^\prime h^\prime}=\tfrac{1}{p}U_h^{-1}\sum\limits_{\tau\in T}\tau^\prime(c_\alpha)U_{\tau^\prime h^\prime}.$$
As $U_{\tau^\prime h^\prime}\in U_{\tau^\prime}+p^{m+2}\Mat_d(\mathcal{O}_{C_\Delta}^{H_{\Delta^\prime}})$, we have
$$U_hh(B_0)\in\underbrace{\tfrac{1}{p}\sum\limits_{\tau\in T}\tau^\prime(c_\alpha)U_{\tau^\prime}}_{=B_0}+p^{m+1}\Mat_d\big(\mathcal{O}_{C_\Delta}^{H_{\Delta,\Delta^\prime}}\big)$$
hence $B_0^{-1}U_hh(B_0)\in\I_d+p^{m+1}\Mat_d\big(\mathcal{O}_{C_\Delta}^{H_{\Delta,\Delta^\prime}}\big)$.
\end{proof}

\begin{lemm}\label{lemmSenH2}
Under the assumptions of lemma \ref{lemmSenH}, there exists $B_\alpha\in\I_d+p^{m-1}\Mat_d\big(\mathcal{O}_{C_\Delta}^{H_{\Delta,\Delta^\prime}}\big)$ such that the cocycle $U\colon H_\Delta\to\GL_d(C_\Delta)$ defined by $U_h^\prime:=B_\alpha^{-1}U_hh(B_\alpha)$ is such that $U_h^\prime=\I_d$ for all $h\in H_{\Delta,\Delta^\prime\cup\{\alpha\}}$ and $U_h^\prime\in\I_d+p^{m-1}\Mat_d\big(\mathcal{O}_{C_\Delta}^{H_{\Delta,\Delta^\prime\cup\{\alpha\}}}\big)$ for all $h\in H_\Delta$.
\end{lemm}

\begin{proof}
Lemma \ref{lemmSenH} produces inductively a sequence $(B_n)_{n\geq0}$ in $\I_d+p^{m-1}\Mat_d\big(\mathcal{O}_{C_\Delta}^{H_{\Delta,\Delta^\prime}}\big)$ and a sequence of cocycles $\big(U_n\colon H\to\GL_d\big(C_\Delta^{H_{\Delta,\Delta^\prime}}\big)\big)_{n\geq0}$ such that $U_0=U$, $B_n\in\I_d+p^{m+n-1}\Mat_d\big(\mathcal{O}_{C_\Delta}^{H_{\Delta,\Delta^\prime}}\big)$, $U_{n+1,h}=B_n^{-1}U_{n,h}h(B_n)$ for all $h\in H_\Delta$ and $U_{n,h}\in\I_d+p^{n+m}\Mat_d\big(\mathcal{O}_{C_\Delta}^{H_{\Delta,\Delta^\prime}}\big)$ for all $h\in\iota_\alpha(H_0)$ and $n\in\NN$. The infinite product $B_\alpha=B_0B_1\cdots$ converges in $\I_d+p^{m-1}\Mat_d\big(\mathcal{O}_{C_\Delta}^{H_{\Delta,\Delta^\prime}}\big)$: for $h\in H_\Delta$, put $U_h^\prime=B_\alpha^{-1}U_hh(B_\alpha)$. By construction, we have $U_h^\prime=\I_d$ for all $h\in\iota_\alpha(H_0)$.

\noindent
If $h\in H_{\Delta,\Delta^\prime}$, we have $h(B_\alpha)=B_\alpha$ and $U_h=\I_d$ hence $U_h^\prime=\I_d$. This shows that $U_h^\prime=\I_d$ for all $h\in H_{\Delta,\Delta^\prime\cup\{\alpha\}}$.

\noindent
If $\beta\in\Delta\setminus(\Delta^\prime\cup\{\alpha\})$, $h\in\iota_\alpha(H_0)$ and $\eta\in\iota_\beta(H_0)$, we have $h\eta=\eta h$, so $U_h^\prime h(U_\eta^\prime)=U_\eta^\prime\eta(U_h^\prime)$: as $U_h^\prime=\I_d$, we get $h(U_\eta^\prime)=U_\eta^\prime$. This implies that $U_\eta^\prime\in\GL_d\big(C_\Delta^{H_{\Delta,\Delta\cup\{\alpha\}}}\big)$.

\noindent
As $B_\alpha\in\I_d+p^{m-1}\Mat_d\big(\mathcal{O}_{C_\Delta}^{H_{\Delta,\Delta^\prime}}\big)$ and $U_h\in\I_d+p^m\Mat_d\big(\mathcal{O}_{C_\Delta}^{H_{\Delta,\Delta^\prime}}\big)$, we have
$$U_h^\prime=B_\alpha^{-1}U_hh(B_\alpha)\in\I_d+p^{m-1}\Mat_d\big(\mathcal{O}_{C_\Delta}^{H_{\Delta,\Delta^\prime}}\big)$$
for all $h\in H$.
\end{proof}

\begin{prop}\label{propSenH}
(\cf \cite[Proposition 4]{Sen}, \cite[Proposition 2.2]{AB}) Let $U\colon H_\Delta\to\GL_d(C_\Delta)$ be a continuous cocycle such that $U_h\equiv\I_d\mod p^r\Mat_d(\mathcal{O}_{C_\Delta})$ for all $h\in H_\Delta$, where $r\in\NN_{\geq\delta+1}$. Then there exists $B\in\I_d+p^{r-\delta}\Mat_d(\mathcal{O}_{C_\Delta})$ such that $B^{-1}U_hh(B)=\I_d$ for all $h\in H_\Delta$. In particular $U$ has a trivial image in $\H^1(H_\Delta,\GL_d(\mathcal{O}_{C_\Delta}))$.
\end{prop}

\begin{proof}
Using lemma \ref{lemmSenH2}, this follows inductively from the case $\Delta^\prime=\varnothing$ and $m=r$, working componentwise.
\end{proof}

\begin{coro}\label{coroSenH}
(\cf \cite[Corollaire 2.3]{AB}) The maps
\begin{align*}
\varinjlim\limits_{\substack{H\lhd H_{K,\Delta}\\ H\text{ \emph{open}}}}\H^1(H_{K,\Delta}/H,\GL_d(C_\Delta^H)) &\to \H^1(H_{K,\Delta},\GL_d(C_\Delta))\\
\varinjlim\limits_{\substack{H\lhd H_{K,\Delta}\\ H\text{ \emph{open}}}}\H^1(G_{K,\Delta}/H,\GL_d(C_\Delta^H)) &\to \H^1(G_{K,\Delta},\GL_d(C_\Delta))
\end{align*}
induced by inflation maps are bijective.
\end{coro}

\begin{proof}
The second statement follows from the first one. Let $U\colon H_{K,\Delta}\to\GL_d(C_\Delta)$ be a continuous cocycle. By continuity, there exists an open normal subgroup $H\leq H_{K,\Delta}$ such that $U_h\in\I_d+p^{\delta+1}\Mat_d(\mathcal{O}_{C_\Delta})$ for all $h\in H$. Making $H$ smaller if necessary, we can assume that $H=\prod\limits_{\alpha\in\Delta}H_0$ where $H_0\leq H_K$ is an open normal subgroup (note that such subgroups are cofinal among open normal subgroups of $H_{K,\Delta}$). Proposition \ref{propSenH} shows that the cocycle we started with has a trivial image in $\H^1(H,\GL_d(C_\Delta))$. The inflation-restriction exact sequence of sets:
$$\{1\}\to\H^1\big(H_{K,\Delta}/H,\GL_d(C_\Delta^H)\big)\to\H^1\big(H_{K,\Delta},\GL_d(C_\Delta)\big)\to\H^1\big(H,\GL_d(C_\Delta)\big)$$
(\cf \cite[I, \S5.8]{Serre97}) shows that $U$ is the induction of a unique class in $\H^1\big(H_{K,\Delta}/H,\GL_d(C_\Delta^H)\big)$.
\end{proof}


Let $H_0\leq H_K$ be an open normal subgroup. By \cite{Ax}, the field $C^{H_0}$ is the closure of $\Kbar^{H_0}$. The latter is a finite extension of $\Kbar^{H_K}=K_\infty$. If $x$ is a primitive element for $\Kbar^{H_0}/K_\infty$, there exists $n\in\NN$ such that $x$ is algebraic over $K_n$: if we put $K^\prime=K_n(x)$, then $\Kbar^{H_0}=K^\prime_\infty$, and $C^{H_0}=L^\prime:=\widehat{K^\prime_\infty}$. In particular, if $H=\prod\limits_{\alpha\in\Delta}H_0$, we have $C_\Delta^H=L^\prime_\Delta$ by Theorem \ref{theocohoHC}. In what follows, we denote by $R_{n,\alpha}^\prime$ the generalized normalized Tate traces relative to $K^\prime$ (\cf \cite[\S3]{Tate} and proposition \ref{propRn}).

Fix $\Delta^\prime\subset\Delta$, $\alpha\in\Delta\setminus\Delta^\prime$ and $n\geq n_{K^\prime}^\prime$. The topological decomposition $L^\prime=K_n^\prime\oplus X_n^\prime$ (where $X_n^\prime=\Ker(R_n^\prime)$ gives rise to the topological decomposition
$$L^\prime_{\Delta,\Delta^\prime}=L^\prime_{\Delta,\Delta^\prime\cup\{\alpha\},n}\oplus\Ker(R_{n,\alpha}^\prime).$$
Moreover, if $\gamma\in\Gamma_{K,\alpha}$ is such that $v_p(1-\chi(\gamma))\leq n_{K^\prime}^\prime$, then $\gamma-1$ is invertible on $\Ker(R_{n,\alpha}^\prime)$, and for all $x\in\Ker(R_{n,\alpha}^\prime)\cap\mathcal{O}_{L^\prime_{\Delta,\Delta^\prime}}$, we have $(\gamma-1)^{-1}(x)\in\frac{1}{p}\mathcal{O}_{L^\prime_{\Delta,\Delta^\prime}}$ (\cf proposition \ref{propRn}).

\medskip

\begin{lemm}\label{lemmdecompletion0}
(\cf \cite[Proposition 3]{Sen}, \cite[Lemme 3.2.5]{BC}). Let $\gamma\in\Gamma_{K^\prime,\alpha}$ such that $v_p(1-\chi(\gamma))\leq n$. Assume $B\in\GL_d(L^\prime_{\Delta,\Delta^\prime})$ and $V_1,V_2\in\I_d+p^2\Mat_d\big(\mathcal{O}_{L^\prime_{\Delta,\Delta^\prime\cup\{\alpha\},n}}\big)$ are such that $\gamma(B)=V_1BV_2$. Then $B\in\GL_d(L^\prime_{\Delta,\Delta^\prime\cup\{\alpha\},n})$.
\end{lemm}

\begin{proof}
Put $Z=B-R_{n,\alpha}^\prime(B)\in\Mat_d(L^\prime_{\Delta,\Delta^\prime})$: as $R_{n,\alpha}$ is $L^\prime_{\Delta,\Delta^\prime\cup\{\alpha\},n}$-linear and commutes with the action of $\Gamma_{K^\prime,\alpha}$, we have $\gamma(Z)=V_1ZV_2$, hence
$$\gamma(Z)-Z=V_1ZV_2-Z=(V_1-\I_d)Z+V_1Z(V_2-\I_d)-(V_1-\I_d)Z(V_2-\I_d).$$
If $Z\in p^a\Mat_d\big(\mathcal{O}_{L^\prime_{\Delta,\Delta^\prime}}\big)$, this implies that $(\gamma-1)(Z)\in p^{a+2}\Mat_d\big(\mathcal{O}_{L^\prime_{\Delta,\Delta^\prime}}\big)$. As $Z$ has entries in $\Ker(R_{n,\alpha}^\prime)$, this implies that $Z\in p^{a+1}\Mat_d\big(\mathcal{O}_{L^\prime_{\Delta,\Delta^\prime}}\big)$. This shows that $Z=0$, \ie that $B$ has entries in $L^\prime_{\Delta,\Delta^\prime\cup\{\alpha\},n}$.
\end{proof}

\begin{lemm}\label{lemmdecompletion1}
(\cf \cite[Lemma 3]{Sen}, \cite[Lemme 3.2.3]{BC}). Let $a,b\in\NN$ be such that $b\geq a>2$ and $\gamma\in\Gamma_{K^\prime,\alpha}$ such that $v_p(1-\chi(\gamma))\leq n$. Assume $U=\I_d+p^aU_1+p^bU_2$ with $U_1\in\Mat_d\big(\mathcal{O}_{L^\prime_{\Delta,\Delta^\prime\cup\{\alpha\},n}}\big)$ and $U_2\in\Mat_d\big(\mathcal{O}_{L^\prime_{\Delta,\Delta^\prime}}\big)$. Then there exists $V\in\Mat_d\big(\mathcal{O}_{L^\prime_{\Delta,\Delta^\prime}}\big)$ such that $M^{-1}U\gamma(M)=\I_d+p^aV_1+p^{b+1}V_2$ with $V_1\in\Mat_d\big(\mathcal{O}_{L^\prime_{\Delta,\Delta^\prime\cup\{\alpha\},n}}\big)$ and $V_2\in\Mat_d\big(\mathcal{O}_{L^\prime_{\Delta,\Delta^\prime}}\big)$, where $M=\I_d+p^{b-1}V$.
\end{lemm}

\begin{proof}
We can write $U_2=R^\prime_{n,\alpha}(U_2)+\frac{1}{p}(1-\gamma)(V)$, where $V\in\Mat_d\big(\mathcal{O}_{L^\prime_{\Delta,\Delta^\prime}}\big)$ has entries in $\Ker(R^\prime_{n,\alpha})$. Then we have $M^{-1}=\sum\limits_{j=0}^\infty p^{j(b-1)}V^j\in\I_d-p^{b-1}V+p^{b+1}\Mat_d\big(\mathcal{O}_{L^\prime_{\Delta,\Delta^\prime}}\big)$ (because $2(b-1)\geq b+1$ since $b\geq3$ by hypothesis), so that
\begin{align*}
M^{-1}U\gamma(M) &\in (\I_d-p^{b-1}V)U(\I_d+p^{b-1}\gamma(V))+p^{b+1}\Mat_d\big(\mathcal{O}_{L^\prime_{\Delta,\Delta^\prime}}\big)\\
&\in U-p^{b-1}(VU-U\gamma(V))+p^{b+1}\Mat_d\big(\mathcal{O}_{L^\prime_{\Delta,\Delta^\prime}}\big).
\end{align*}
We have $U\in\I_d+p^2\Mat_d\big(\mathcal{O}_{L^\prime_{\Delta,\Delta^\prime}}\big)$ (because $b\geq a>2$), so that $VU-U\gamma(V)\in(1-\gamma)(V)+p^2\Mat_d\big(\mathcal{O}_{L^\prime_{\Delta,\Delta^\prime}}\big)$. This implies that
\begin{align*}
M^{-1}U\gamma(M) &\in U-p^{b-1}(1-\gamma)(V)+p^{b+1}\Mat_d\big(\mathcal{O}_{L^\prime_{\Delta,\Delta^\prime}}\big)\\
&\in \I_d+p^aU_1+p^b\big(R^\prime_{n,\alpha}(U_2)+\tfrac{1}{p}(1-\gamma)(V)\big)-p^{b-1}(1-\gamma)(V)+p^{b+1}\Mat_d\big(\mathcal{O}_{L^\prime_{\Delta,\Delta^\prime}}\big)\\
&\in \I_d+p^aV_1+p^{b+1}\Mat_d\big(\mathcal{O}_{L^\prime_{\Delta,\Delta^\prime}}\big)
\end{align*}
with $V_1=U_1+p^{b-a}R^\prime_{n,\alpha}(U_2)\in\Mat_d\big(\mathcal{O}_{L^\prime_{\Delta,\Delta^\prime\cup\{\alpha\},n}}\big)$.
\end{proof}

\begin{coro}\label{corodecompletion1}
(\cf \cite[Proposition 6]{Sen}, \cite[Corollaire 3.2.4]{BC}). Let $a\in\NN_{>2}$ and $\gamma\in\Gamma_{K^\prime,\alpha}$ such that $v_p(1-\chi(\gamma))\leq n$. Assume $\displaystyle U\in\I_d+p^a\Mat_d\big(\mathcal{O}_{L^\prime_{\Delta,\Delta^\prime}}\big)$. Then there exists $\displaystyle M\in\I_d+p^{a-1}\Mat_d\big(\mathcal{O}_{L^\prime_{\Delta,\Delta^\prime}}\big)$ such that $ \displaystyle M^{-1}U\gamma(M)\in\I_d+p^a\Mat_d\big(\mathcal{O}_{L^\prime_{\Delta,\Delta^\prime\cup\{\alpha\},n}}\big)$.
\end{coro}

\begin{proof}
Using the previous lemma inductively, we can construct a sequence of matrices $(M_b)_{b\geq a}$ such that $M_b\in\I_d+p^{b-1}\Mat_d\big(\mathcal{O}_{L^\prime_{\Delta,\Delta^\prime}}\big)$ and
$$(M_aM_{a+1}\cdots M_b)^{-1}U\gamma(M_aM_{a+1}\cdots M_b)\in\I_d+p^a\Mat_d\big(\mathcal{O}_{L^\prime_{\Delta,\Delta^\prime\cup\{\alpha\},n}}\big)+p^{b+1}\Mat_d\big(\mathcal{O}_{L^\prime_{\Delta,\Delta^\prime}}\big)$$
for all $b\geq a$. The infinite product $\prod\limits_{b=a}^\infty M_b$ converges in $\I_d+p^{a-1}\Mat_d\big(\mathcal{O}_{L^\prime_{\Delta,\Delta^\prime}}\big)$ and has the required property.
\end{proof}

By definition, we have inclusions
$$K^\prime_{n,\Delta}=L^\prime_{\Delta,\varnothing,n}\subset L^\prime_{\Delta,\Delta^\prime,n}\subset L^\prime_{\Delta,\Delta,n}=L^\prime_\Delta$$
for all $n\in\NN$, hence inclusions
$$K^\prime_{\Delta,\infty}:=\bigcup\limits_{n=0}^\infty K^\prime_{n,\Delta}=L^\prime_{\Delta,\varnothing,\infty}\subset L^\prime_{\Delta,\Delta^\prime,\infty}\subset L^\prime_\Delta$$
for all $\Delta^\prime\subset\Delta$. We put $\mathcal{O}_{K^\prime_{\Delta,\infty}}=\bigcup\limits_{n=0}^\infty\mathcal{O}_{K^\prime_{n,\Delta}}$.

\begin{coro}\label{corodecompletion2}
(\cf \cite[Proposition 3.2.6]{BC}). Let $U\colon\Gamma_{K^\prime,\Delta}\to\I_d+p^2\Mat_d\big(\mathcal{O}_{L^\prime_{\Delta,\Delta^\prime,\infty}}\big)$ be a continuous cocycle. Then there exists $M\in\I_d+p\Mat_d\big(\mathcal{O}_{L^\prime_{\Delta,\Delta^\prime,\infty}}\big)$ such that $M^{-1}U_gg(M)\in\I_d+p^2\Mat_d\big(\mathcal{O}_{K^\prime_{\Delta,\Delta^\prime\cup\{\alpha\},\infty}}\big)$ for all $g\in\Gamma_{K^\prime,\Delta}$.
\end{coro}

\begin{proof}
As $\Gamma_{K^\prime,\Delta}$ is topologically generated by finitely many elements, there exists $n\geq n_{K^\prime}^\prime$ such that $U_g\in\I_d+p^2\Mat_d\big(\mathcal{O}_{L^\prime_{\Delta,\Delta^\prime,n}}\big)$ for all $g\in\Gamma_{K^\prime,\Delta}$. Let $\gamma\in\Gamma_{K^\prime,\alpha}$ be an element of infinite order. Enlarging $n$ if necessary, we may assume that $v_p(1-\chi(\gamma))\leq n$. By corollary \ref{corodecompletion1}, there exists $M\in\I_d+p\Mat_d\big(\mathcal{O}_{L^\prime_{\Delta,\Delta^\prime,n}}\big)$ such that $M^{-1}U_\gamma\gamma_0(M)\in\I_d+p^2\Mat_d\big(\mathcal{O}_{L^\prime_{\Delta,\Delta^\prime\cup\{\alpha\},n}}\big)$.

\noindent
For all $g\in\Gamma_{K^\prime,\Delta}$, put $U^\prime_g=M^{-1}U_gg(M)$: this defines a cocycle $U^\prime\colon\Gamma_{K^\prime,\Delta}\to\I_d+p^2\Mat_d\big(\mathcal{O}_{L^\prime_{\Delta,\Delta^\prime,n}}\big)$ which is cohomologous to $U$, and such that $U^\prime_\gamma\in\I_d+p^2\Mat_d\big(\mathcal{O}_{L^\prime_{\Delta,\Delta^\prime\cup\{\alpha\},n}}\big)$. Let $g\in\Gamma_{K^\prime,\Delta}$. As $\Gamma_{K^\prime,\Delta}$ is commutative, we have $\gamma g=g\gamma$ hence $U^\prime_\gamma\gamma(U^\prime_g)=U^\prime_gg(U^\prime_\gamma)$ \ie $\gamma(U^\prime_g)=U_\gamma^{\prime-1}U^\prime_gg(U^\prime_\gamma)$. Lemma \ref{lemmdecompletion0} applied with $V_1=U_\gamma^{\prime-1}$ and $V_2=g(U^\prime_\gamma)$ (here he use the fact that $L_{\Delta,\Delta^\prime\cup\{\alpha\},n}$ is stable by $g$, which follows form the commutativity of $\Gamma_{K^\prime,\Delta}$) implies that $U^\prime_g$ has coefficients in $L^\prime_{\Delta,\Delta^\prime\cup\{\alpha\},n}$, so that $U^\prime$ has values in $\I_d+p^2\Mat_d\big(\mathcal{O}_{L^\prime_{\Delta,\Delta^\prime\cup\{\alpha\},n}}\big)$.
\end{proof}

\begin{coro}\label{corodecompletion3}
Let $U\colon\Gamma_{K^\prime,\Delta}\to\I_d+p^2\Mat_d\big(\mathcal{O}_{L^\prime_\Delta}\big)$ be a continuous cocycle. Then there exists $B\in\I_d+p\Mat_d\big(\mathcal{O}_{L^\prime_\Delta}\big)$ such that $B^{-1}U_gg(B)\in\I_d+p^2\Mat_d\big(\mathcal{O}_{K^\prime_{\Delta,\infty}}\big)$ for all $g\in\Gamma_{K^\prime,\Delta}$.
\end{coro}

\begin{proof}
This follows by using corollary \ref{corodecompletion2} finitely many times.
\end{proof}

\begin{prop}\label{propdecompletion}
(\cf \cite[Proposition 3.2.6]{BC}). Let $U\colon G_{K,\Delta}\to\GL_g(C_\Delta)$ be a continuous cocycle. There exist a finite subextension $K^\prime$ of $\Kbar/K$, a matrix $B\in\GL_d(C_\Delta)$ such that $B^{-1}U_gg(B)=\I_d$ for all $g\in H_{K^\prime,\Delta}$ and $B^{-1}U_gg(B)\in\GL_d\big(\mathcal{O}_{K^\prime_\Delta}\big)$ for all $g\in G_{K^\prime,\Delta}$.
\end{prop}

\begin{proof}
By continuity, there exists a finite Galois subextension $K^\prime$ of $\Kbar/K$ such that $U_g\in\I_d+p^{2+\delta}\Mat_d\big(\mathcal{O}_{C_\Delta}\big)$ for all $g\in G_{K^\prime,\Delta}$. Proposition \ref{propSenH} applied with $H_0=H_{K^\prime}$ and $r=2+\delta$ implies the existence of $B_1\in\I_d+p^2\Mat_d\big(\mathcal{O}_{C_\Delta}\big)$ such that $B_1^{-1}U_hh(B_1)=\I_d$ for all $h\in H_{K^\prime,\Delta}$. By construction, we have $U^\prime_g:=B_1^{-1}U_gg(B_1)=\I_d+p^2\Mat_d\big(\mathcal{O}_{C_\Delta}\big)$ for all $g\in G_{K^\prime,\Delta}$. Let $g\in G_{K,\Delta}$ and $h\in H_{K^\prime,\Delta}$. As $H_{K^\prime,\Delta}\lhd G_{K,\Delta}$, we have $h^\prime=g^{-1}hg\in H_{K^\prime,\Delta}$, so that $U^\prime_hh(U^\prime_g)=U^\prime_gg(U^\prime_{h^\prime})$ \ie $h(U^\prime_g)=U^\prime_g$ (since $U^\prime_h=U^\prime_{h^\prime}=\I_d$). As this holds for all $h\in H_{K^\prime,\Delta}$, this implies that $U_g\in\Mat_d\big(C_\Delta^{H_{K^\prime,\Delta}}\big)=\Mat_d(L^\prime_\Delta)$ (\cf theorem \ref{theocohoHC}). This shows that in fact, we have $U^\prime_g\in\GL_d(\mathcal{O}_{L^\prime_\Delta}\big)$ for all $g\in G_{K,\Delta}$, and that, in particular, the restriction of $U^\prime$ defines a continuous cocycle $U^\prime\colon\Gamma_{K^\prime,\Delta}\to\I_d+p^2\Mat_d\big(\mathcal{O}_{L^\prime_\Delta}\big)\subset\GL_d\big(\mathcal{O}_{L^\prime_\Delta}\big)$.

\noindent
By corollary \ref{corodecompletion3}, there exist $B_2\in\I_d+p\Mat_d\big(\mathcal{O}_{L^\prime_\Delta}\big)$ and $n\in\NN$ such that $B_2^{-1}U^\prime_gg(B_2)\in\I_d+p^2\Mat_d\big(\mathcal{O}_{K^\prime_{n,\Delta}}\big)$ for all $g\in\Gamma_{K^\prime,\Delta}$. Replacing $K^\prime$ by $K^\prime_n$, we may assume that $K^\prime_n=K^\prime$. Put $B=B_2B_1\in\GL_d(C_\Delta)$ and $U^{\prime\prime}_g=B^{-1}U_gg(B)$ for all $g\in G_{K,\Delta}$. By construction, we have $U^{\prime\prime}_g=\I_d$ for all $g\in H_{K^\prime,\Delta}$ and $U^{\prime\prime}_g\in\I_d+p^2\Mat_d\big(\mathcal{O}_{K^\prime_\Delta}\big)$ for all $g\in G_{K^\prime,\Delta}$.

\noindent
Let $g\in G_{K,\Delta}$ and $\gamma\in G_{K^\prime,\Delta}$. As $K^\prime/K$ is Galois, we have $G_{K^\prime,\Delta}\lhd G_{K,\Delta}$, so that $\gamma^\prime:=g^{-1}\gamma g\in G_{K^\prime,\Delta}$. By the cocycle condition, we have $\gamma(U_g)=U_\gamma^{-1}U_gg(U_{\gamma^\prime})$. A repeated application of lemma \ref{lemmdecompletion0} (for each $\alpha\in\Delta$) thus implies that $U_g\in\GL_d\big(\mathcal{O}_{K^\prime_\Delta}\big)$.
\end{proof}

\begin{coro}\label{coroSen}
(\cf \cite[Th\'eor\`eme 3.1]{AB}) The natural map
$$\varinjlim\limits_{K^\prime}\H^1(G_{K,\Delta}/H_{K^\prime,\Delta},\GL_d(K^\prime_\Delta))\to\H^1(G_{K,\Delta},\GL_d(C_\Delta))$$
(where $K^\prime$ runs among the finite Galois subextensions of $\Kbar/K$) induced by inflation maps is bijective.
\end{coro}

\begin{proof}
Surjectivity is nothing but proposition \ref{propdecompletion}: it remains to prove the injectivity. Let $K^\prime$ be a finite subextension of $\Kbar/K$ and $U,U^\prime\colon G_{K,\Delta}/H_{K^\prime,\Delta}\to\GL_g(K^\prime_\Delta)$ be two continuous cocycles that induce cohomologous cocycles $G_{K,\Delta}\to\GL_d(C_\Delta)$. This means that there exists $B\in\GL_d(C_\Delta)$ such that $U^\prime_g=B^{-1}U_gg(B)$ for all $g\in G_{K,\Delta}$. By continuity, we may enlarge $K^\prime$ and assume that $U_g,U^\prime_g\in\I_d+p^2\Mat_d(\mathcal{O}_{K^\prime,\Delta})$ for all $g\in G_{K^\prime,\Delta}$. If $g\in H_{K^\prime,\Delta}$, we have $U_g,U^\prime_g=\I_d$, so that $g(B)=B$: this shows that $B\in\GL_d(L^\prime_\Delta)$, where $L^\prime=C^{H_{K^\prime}}$. Then we have $\gamma(B)=U^{-1}_\gamma BU^\prime_\gamma$ for all $\gamma\in G_{K,\Delta}/H_{K^\prime,\Delta}$. Applying lemma \ref{lemmdecompletion0} finitely many times (for each $\alpha\in\Delta$) shows that $B\in\GL_d(K^\prime_{\Delta,\infty})$. Replacing $K^\prime$ by $K^\prime_n$ for $n\in\NN$ large enough implies that $U$ and $U^\prime$ are cohomologous as cocycles with values in $\GL_d(K^\prime_\Delta)$, proving the injectivity.
\end{proof}

We now refine the previous statement by translating it in terms of $C_\Delta$-representations of $G_{K,\Delta}$.

\begin{theo}\label{theoSen}
(\cf \cite[Theorem 3]{Sen}). Let $W$ be a free $C_\Delta$-representation of $G_{K,\Delta}$ of rank $d$. Then there exists a free $K_{\Delta,\infty}$-representation $Y$ of $\Gamma_{K,\Delta}$ of rank $d$ and such that $W\simeq C_\Delta\otimes_{K_{\Delta,\infty}}Y$ (as $C_\Delta$-representations of $G_{K,\Delta}$).
\end{theo}

\begin{proof}
Corollary \ref{coroSen} implies that there exists a finite Galois subextension $K^\prime$ of $\Kbar/K$ and a free $K^\prime_\Delta$-representation of rank $d$ of $G_{K,\Delta}/H_{K^\prime,\Delta}$ such that $W\simeq C_\Delta\otimes_{K^\prime_\Delta}W^\prime$ as $C_\Delta$-representations of $G_{K,\Delta}$. Enlarging $K^\prime$ if necessary, we may furthermore assume that $K^\prime/F_0$ is Galois. By restriction, the group $\Gal(K^\prime_\infty/K_\infty)$ identifies with a subgroup $G=\Gal(K^\prime/K^\prime\cap K_\infty)$ of $\Gal(K^\prime/F_0)$. Put $F=K^{\prime G}=K^\prime\cap K_\infty$. Note that $G$ is the kernel of the map $\Gal(K^\prime/F_0)\to\Gamma_K$ induced by the restriction to $K_\infty$: this implies that $F/F_0$ is Galois. If $\Delta^\prime\subset\Delta$ and $\alpha\in\Delta\setminus\Delta^\prime$, the finite Galois extension $F\to K^\prime$ induces a finite \'etale extension
$$F_{\Delta^\prime\cup\{\alpha\}}\otimes_KK^\prime_{\Delta\setminus(\Delta^\prime\cup\{\alpha\})}\to F_{\Delta^\prime}\otimes_KK^\prime_{\Delta\setminus\Delta^\prime}$$
with group $G_\alpha$. By Galois descent, if $W_\alpha$ is a rank $d$ projective $F_{\Delta^\prime}\otimes_KK^\prime_{\Delta\setminus\Delta^\prime}$-representation of $G_\alpha$, then $W_\alpha^{G_\alpha}$ is a rank $d$ projective $F_{\Delta^\prime\cup\{\alpha\}}\otimes_KK^\prime_{\Delta\setminus(\Delta^\prime\cup\{\alpha\})}$-module of finite rank and the natural map
$$(F_{\Delta^\prime}\otimes_KK^\prime_{\Delta\setminus\Delta^\prime})\otimes_{F_{\Delta^\prime\cup\{\alpha\}}\otimes_KK^\prime_{\Delta\setminus(\Delta^\prime\cup\{\alpha\})}}W_\alpha^{G_\alpha}\to W_\alpha$$
is an isomorphism. Starting from $W^\prime$ and applying what precedes for each $\alpha\in\Delta$ implies that $W^{\prime G_\Delta}$ is a projective $F_\Delta$-module of rank $d$, and that the map
$$K^\prime_\Delta\otimes_{F_\Delta}W^{\prime G_\Delta}\to W^\prime$$
is an isomorphism. As $F/F_0$ is finite and Galois, $F_\Delta$ is a finite product of copies of $F$ (indexed by $\Gal(F/F_0)^{\delta-1}$), so that a projective $F_\Delta$-module is necessarily free of rank $d$ (the dimension over $F$ of all its localizations is $d$, since it is free of rank $d$ after tensoring with $C_\Delta$). Then we have $G_{K,\Delta}$-equivariant isomorphisms
$$W\simeq C_\Delta\otimes_{K^\prime_\Delta}W^\prime\simeq C_\Delta\otimes_{F_\Delta}W^{\prime G_\Delta}\simeq C_\Delta\otimes_{K_{\Delta,\infty}}Y$$
where $Y=K_{\Delta,\infty}\otimes_{F_\Delta}W^{\prime G_\Delta}$ is a $K_{\Delta,\infty}$-representation of $\Gamma_{K,\infty}$ which is free of rank $d$.
\end{proof}

\begin{coro}\label{coroSen0}
If $W$ is a $C_\Delta$-representation of $G_{K,\Delta}$, then $W^{H_{K,\Delta}}$ is a free $L_\Delta$-representation of $\Gamma_{K,\Delta}$ of rank $d$, and the natural map
$$C_\Delta\otimes_{L_\Delta}W^{H_{K,\Delta}}\to W$$
is a $G_{K,\Delta}$-equivariant isomorphism.
\end{coro}

\begin{proof}
By theorem \ref{theoSen}, we may assume that $W=C_\Delta\otimes_{K_{\infty,\Delta}}Y$ where $Y$ ia a free $K_{\infty,\Delta}$-representation of $\Gamma_{K,\Delta}$ of rank $d$. Then $W^{H_{K,\Delta}}=L_\Delta\otimes_{K_{\infty,\Delta}}Y$ is a free $L_\Delta$-representation of $\Gamma_{K,\Delta}$ of rank $d$, and the natural map $C_\Delta\otimes_{L_\Delta}W^{H_{K,\Delta}}\to W$ is a $G_{K,\Delta}$-equivariant isomorphism.
\end{proof}

\begin{coro}\label{coroSen1}
(\cf \cite[Theorem 1]{Sen}) The natural map
$$\H^1(\Gamma_{K,\Delta},\GL_d(K_{\Delta,\infty}))\to\H^1(G_{K,\Delta},\GL_d(C_\Delta))$$
is bijective.
\end{coro}

\begin{proof}
Here again the surjectivity follows from theorem \ref{theoSen}, and the injectivity is proved exactly as in the proof of corollary \ref{coroSen}.
\end{proof}

\begin{prop}\label{propfond}
If $\mathcal{E}\subset L_\Delta$ is a sub-$K_\Delta$-module of finite type stable by $\Gamma_{K,\Delta}$, then $\mathcal{E}\subset K_{\Delta,\infty}$ (more precisely, there exists $n\in\NN$ such that $\mathcal{E}\subset K_{n,\Delta}$).
\end{prop}

\begin{proof}
Enlarging $K$, we may assume that $K/F_0$ is Galois, with group $G_{K/F_0}$. The map
\begin{align*}
K\otimes_{F_0}K &\to K^{G_{K/F_0}}\\
x\otimes y &\mapsto (x\sigma(y))_{\sigma\in G_{K/F_0}}
\end{align*}
is an isomorphism of $K$-algebras (with the left structure on the LHS, and the diagonal structure on the RHS). Fix an ordering $\alpha_1<\alpha_2<\cdots<\alpha_{\delta}$ of $\Delta$: by induction, what precedes provides an isomorphism of $K$-algebras
$$\iota\colon K_\Delta\isomto K^{G_{K/F_0}^{\delta-1}}$$
where the component of index $\underline{\sigma}=(\sigma_2,\ldots,\sigma_{\delta})\in G_{K/F_0}^{\delta-1}$ of $\iota(x_1\otimes\cdots\otimes x_{\delta})$ is given by
$$x_1\sigma_2(x_2\sigma_3(x_3\cdots x_{\delta-1}\sigma_\delta(x_{\delta})\cdots)$$
(here the $K$-algebra structure on the LHS is through the factor of index $\alpha_1$ and that on the RHS is the diagonal one). For each $\underline{\sigma}\in G_{K/F_0}^{\delta-1}$, we thus have a surjective morphism of $K$-algebras $\iota_{\underline{\sigma}}\colon K_\Delta\to K$ corresponding to the projection onto the factor of index $\underline{\sigma}$. Similarly, we have an injective map $\iota_\infty\colon K_{\Delta,\infty}\to K_\infty^{G_{K_\infty/F_0}^{\delta-1}}$, that extends into an injective map $\widehat{\iota}_\infty\colon L_\Delta\to L^{G_{K_\infty/F_0}^{\delta-1}}$ (because $\iota_\infty$ maps $\mathcal{O}_{K_{\Delta,\infty}}$ into $\mathcal{O}_L^{G_{K_\infty/F_0}^{\delta-1}}$ which is $p$-adically separated and complete). Let $\underline{\sigma}\in G_{K/F_0}^{\delta-1}$: the map $\iota_{\underline{\sigma}}$ is a localization, so that the natural map
$$\widehat{\iota}_{\infty,\underline{\sigma}}\colon K_{\iota_{\underline{\sigma}}}\!\otimes_{K_\Delta}L_\Delta\to\prod\limits_{\substack{\underline{\gamma}\in G_{K_\infty/F_0}^{\delta-1}\\ \underline{\gamma}\mapsto\underline{\sigma}}}L$$
is injective. Note that if $\underline{g}=(g_1,\ldots,g_{\delta})\in\Gamma_{K,\Delta}$ and $x_1,\ldots,x_{\delta}\in K_\infty$, we have
\begin{align*}
\iota_\infty\big(\underline{g}(x_1\otimes\cdots\otimes x_{\delta})\big) &= \iota_\infty\big(g_1(x_1)\otimes\cdots\otimes g_{\delta}(x_{\delta})\big)\\
&= \big(g_1(x_1)\gamma_2(g_2(x_2)\gamma_3(g_3(x_3)\cdots g_{\delta-1}(x_{\delta-1})\gamma_{\delta}(g_{\delta}(x_{\delta})))\cdots)\big)_{\underline{\gamma}\in G_{K_\infty/F_0}^{\delta-1}}\\
&=\Big(g_1\Big(x_1\big(g_1^{-1}\gamma_2g_2\big)\big(x_2\big(g_2^{-1}\gamma_3g_3\big)\big(x_3\cdots \big(g_{\delta-1}^{-1}\gamma_{\delta}g_{\delta}\big)(x_{\delta}\big)\cdots\big)\big)\Big)\Big)_{\underline{\gamma}\in G_{K_\infty/F_0}^{\delta-1}}.
\end{align*}
This shows that $\widehat{\iota}_{\infty,\underline{\sigma}}$ is $\Gamma_{K,\Delta}$-equivariant when the $\prod\limits_{\substack{\underline{\gamma}\in G_{K_\infty/F_0}^{\delta-1}\\ \underline{\gamma}\mapsto\underline{\sigma}}}L$ is equipped with the action given by
$$\underline{g}\cdot(x_{\underline{\gamma}})_{\underline{\gamma}}=\big(g_1(x_{\underline{g}\cdot\underline{\gamma}})\big)_{\underline{\gamma}},$$
where $\underline{g}\cdot\underline{\gamma}=(g_1^{-1}\gamma_2g_2,g_2^{-1}\gamma_3g_3,\ldots,g_{\delta-1}^{-1}\gamma_{\delta}g_{\delta})\in G_{K/F_0}^{\delta-1}$ if $\underline{\gamma}=(\gamma_2,\ldots,\gamma_{\delta})$ (note that this indeed maps to $\underline{\sigma}$ if $\underline{\gamma}$ does).

By hypothesis, the localization $\mathcal{E}_{\underline{\sigma}}:=K_{\iota_{\underline{\sigma}}}\!\otimes_{K_\Delta}\mathcal{E}$ is a finite dimensional sub-$K$-vector space of $K_{\iota_{\underline{\sigma}}}\!\otimes_{K_\Delta}L_\Delta$ that is stable under the action of $\Gamma_{K,\Delta}$. If $\underline{\gamma}_0=(\gamma_2,\ldots,\gamma_{\delta})\in G_{K_\infty/F_0}^{\delta-1}$, the projection $\mathrm{pr}_{\underline{\gamma}_0}\circ\widehat{\iota}_{\infty,\underline{\sigma}}\colon K_{\iota_{\underline{\sigma}}}\!\otimes_{K_\Delta}L_\Delta\to L$ onto the factor of index $\underline{\gamma}_0$ maps $\mathcal{E}_{\underline{\sigma}}$ onto a finite dimensional sub-$K$-vector space $\mathcal{E}_{\underline{\gamma}_0}$ of $L$. Moreover, if $g\in\Gamma_K$, define the element $\underline{g}=(g_1,\ldots,g_{\delta})\in\Gamma_{K,\Delta}$ by $g_1=g$ and $g_i=\gamma_i^{-1}g_{i-1}\gamma_i$ for all $i\in\{2,\ldots,{\delta}\}$ (we have indeed $\gamma_i^{-1}g_{i-1}\gamma_i\in\Gamma_K$ since $\Gamma_K$ is normal in $G_{K_\infty/F_0}$ because $K/F_0$ is Galois). By construction, we have $\underline{g}\cdot\underline{\gamma}_0=\underline{\gamma}_0$, and the component of index $\underline{\gamma}_0$ in
$$\underline{g}\big((x_{\underline{\gamma}})_{\underline{\gamma}}\big)$$
is precisely $g(x_{\underline{\gamma}_0})$ for all $(x_{\underline{\gamma}})_{\underline{\gamma}}\in\prod\limits_{\substack{\underline{\gamma}\in G_{K_\infty/F_0}^{\delta-1}\\ \underline{\gamma}\mapsto\underline{\sigma}}}L$. As $\mathcal{E}_{\underline{\sigma}}$ is stable under $\Gamma_{K,\Delta}$, this implies in particular that $\mathcal{E}_{\underline{\gamma}_0}$ is stable under $\Gamma_K$. By \cite[Proposition 3]{Sen}, this implies that there exists an integer $n_{\underline{\sigma}}$ such that $\mathcal{E}_{\underline{\gamma}_0}\subset K_{n_{\underline{\sigma}}}$.

As $\Gamma_K^{\delta-1}\subset\Gamma_{K,\Delta}$ acts transitively on the set of those elements $\underline{\gamma}\in G_{K_\infty/F_0}^{\delta-1}$ mapping to $\underline{\sigma}$ (by the action given by $(g_2,\ldots,g_{\delta})\cdot(\gamma_2,\ldots,\gamma_{\delta})=(\gamma_2g_2,g_2^{-1}\gamma_3g_3,\ldots,g_{\delta-1}^{-1}\gamma_{\delta}g_{\delta})$), and as $\underline{g}\cdot\big((x_{\underline{\gamma}})_{\underline{\gamma}}\big)=\big(x_{\underline{g}\cdot\underline{\gamma}}\big)_{\underline{\gamma}}$ when $g_1=\Id_{K_\infty}$, the stability of $\mathcal{E}_{\underline{\sigma}}$ under $\Gamma_{K,\Delta}$ implies that the projection $\mathrm{pr}_{\underline{\gamma}}\circ\widehat{\iota}_{\infty,\underline{\sigma}}\colon K_{\iota_{\underline{\sigma}}}\!\otimes_{K_\Delta}L_\Delta\to L$ onto the factor of index $\underline{\gamma}$ maps $\mathcal{E}_{\underline{\sigma}}$ into $K_{n_{\underline{\sigma}}}$ for \emph{all} $\underline{\gamma}\in G_{K_\infty/F_0}^{\delta-1}$ mapping to $\underline{\sigma}$.

The injectivity and equivariance of $\widehat{\iota}_{\infty,\underline{\sigma}}$ imply that $\mathcal{E}_{\underline{\sigma}}$ is fixed by $\Gamma_{K_{n_{\underline{\sigma}}},\Delta}$. Now if we call $n$ the maximum of those $n_{\underline{\sigma}}$ (there are finitely many of these, since $G_{K/F_0}^{\delta-1}$ is finite), this shows that all the localizations $K_{\iota_{\underline{\sigma}}}\!\otimes_{K_\Delta}\mathcal{E}$ are invariants under $\Gamma_{K_n,\Delta}$: the same holds for $\mathcal{E}$.
\end{proof}

\begin{rema}\label{remaRR}
The previous proposition shows that our geometric setting is that of a finite discrete space, corresponding to a finite product of fields: there are finitely many finite extensions $E_i/K$ such that $K_\Delta\simeq\prod\limits_{i \in I}E_i$. The data of a $K_\Delta$-module $M$ is thus equivalent to that of the collection of its localizations $(E_i\times_{K_\Delta}M)_{i \in I}$. In particular, the $K_\Delta$-module $M$ is free of rank $d$ if and only if $\dim_{E_i}(E_i \times_{K_\Delta}M)=d$ for all $i \in I$.
\end{rema}

\begin{defi}
Let $X$ be a $L_\Delta$-representation of $\Gamma_{K,\Delta}$. An element $x\in X$ is said to be $K_\Delta$-finite if its orbit under $\Gamma_{K,\Delta}$ generates a $K_\Delta$-module of finite type in $X$. We denote by $X_{\free}$ the subset of elements elements that $K_\Delta$-finite in $X$. In other words, $X_{\free}$ is the union of all sub-$K_\Delta$-modules of $X$ that are of finite type and stable by $\Gamma_{K,\Delta}$. Note that $X_{\free}$ is a sub-$K_\Delta$-module of $X$, and that it is stable under $\Gamma_{K,\Delta}$.
\end{defi}

\begin{coro}\label{corodecomplf}
If $X\in\Rep_{L_\Delta}(\Gamma_{K,\Delta})$ is free of rank $d$, then $X_{\free}$ is a free $K_{\Delta,\infty}$-module of rank $d$, and the natural map
$$L_\Delta\otimes_{K_{\Delta,\infty}}X_{\free}\to X$$
is a $\Gamma_{K,\Delta}$-equivariant isomorphism.
\end{coro}

\begin{proof}
By corollary \ref{coroSen1}, we can find a $L_\Delta$-basis $\mathfrak{B}=(e_1,\ldots,e_d)$ of $X$ such that $Y:=\bigoplus\limits_{i=1}^dK_{\Delta,\infty}e_i\subset W$ is stable by $\Gamma_{K,\Delta}$. By construction, the natural map $L_\Delta\otimes_{K_{\Delta,\infty}}Y\to X$ is a $\Gamma_{K,\Delta}$-equivariant isomorphism. We have to show that $Y=X_{\free}$. The action of $\Gamma_{K,\Delta}$ on $Y$ is described in the basis $\mathfrak{B}$ by a continuous cocycle $U\colon\Gamma_{K,\Delta}\to\GL_d(K_{\Delta,\infty})$: there exists $m\in\NN$ such that $U$ has values in $K_{m,\Delta}$. This implies that $Y_m:=\bigoplus\limits_{i=1}^dK_{m,\Delta}e_i$ is stable under the action of $\Gamma_{K,\Delta}$, so that $Y_m\subset X_{\free}$. As $X_{\free}$ is a $K_{\infty,\Delta}$-module, this implies that $Y\subset X_{\free}$. Conversely, let $x\in X_{\free}$. As $x\in X$, we can write uniquely $x=\sum\limits_{i=1}^d\lambda_ie_i$ where $\lambda_1,\ldots,\lambda_d\in L_\Delta$. Let $\mathcal{E}$ be the sub-$K_{m,\Delta}$-module of $L_\Delta$ generated by the coordinates in $\mathfrak{B}$ of all the elements $g(x)$ with $g\in\Gamma_{K,\Delta}$. As $x\in X_{\free}$, the $K_{m,\Delta}$-module $\mathcal{E}$ is of finite type, and stable under the action of $\Gamma_{K,\Delta}$ by definition (because the coordinates of $g(x)$ in $\mathfrak{B}$ are given by the matrix product $U_g\left(\begin{smallmatrix}g(\lambda_1)\\[-2mm] \vdots\\ g(\lambda_d)\end{smallmatrix}\right)$ and $U_g\in\GL_d(K_{m,\Delta})$). By proposition \ref{propfond}, we have $\mathcal{E}\subset K_{\infty,\Delta}$, hence $x\in Y$.
\end{proof}

\begin{prop}\label{propextfinet}
Let $K^\prime$ be a finite extension of $K$ in $\Kbar$. Recall that $L=\widehat{K_\infty}$ and put $L^\prime=\widehat{K^\prime_\infty}$. The extensions $K_{\Delta,\infty}\to K^\prime_{\Delta,\infty}$ and $L_\Delta\to L^\prime_\Delta$ are finite \'etale, and Galois with group $\Gal(K^\prime_\infty/K_\infty)^\Delta$ if $K^\prime_\infty/K_\infty$ is Galois. Moreover, we have $K_{\Delta,\infty}^\prime\otimes_{K_{\Delta,\infty}}L_\Delta\isomto L^\prime_\Delta$.
\end{prop}

\begin{proof}
$\bullet$ If $\Delta^\prime\subset\Delta$ and $\alpha\in\Delta\setminus\Delta^\prime$, the finite separable extension $K_n\to K_n^\prime$ tensored with $K_n^{\prime\otimes\Delta^\prime}\otimes_F K_n^{\otimes(\Delta\setminus(\Delta^\prime\cup\{\alpha\}))}$ over $F$ provides a finite \'etale map
$$K_n^{\prime(\Delta\cup\{\alpha\})}\otimes_F K_n^{\otimes(\Delta\setminus(\Delta^\prime\cup\{\alpha\}))}$$
for all $n\in\NN$. the latter is Galois with group $\Gal(K^\prime_\infty/K_\infty)$ when $n\gg0$ if $K^\prime_\infty/K_\infty$ is Galois. The composition of all these maps thus provides a finite \'etale map
$$K_n^{\otimes\Delta}\to K_n^{\prime\otimes\Delta}$$
which is Galois with group $\Gal(K^\prime_\infty/K_\infty)^\Delta$ when $n\gg0$ if $K^\prime_\infty/K_\infty$ is Galois. Put together, this shows that the map $K_{\Delta,\infty}\to K^\prime_{\Delta,\infty}$ is finite \'etale, and Galois with group $\Gal(K^\prime_\infty/K_\infty)^\Delta$ if $K^\prime_\infty/K_\infty$ is Galois.

\noindent
$\bullet$ There exists $n_0\in\NN$ such that $n\geq n_0\Rightarrow[K^\prime_n:K_n]=[K^\prime_\infty:K_\infty]$, so that $K_n^\prime\otimes_{K_n}K_\infty\isomto K^\prime_\infty$. By \cite[Corollary 3.10]{And06}, making $n_0$ larger, we may assume that the cokernel of the map $\mathcal{O}_{K_n^\prime}\otimes_{\mathcal{O}_{K_n}}\mathcal{O}_{K_\infty}\to\mathcal{O}_{K^\prime_\infty}$ is killed by $p$ whenever $n\geq n_0$ (we could replace $p$ by any element on positive valuation in $\mathcal{O}_{K_\infty}$, but we will not use this). This implies that for $n\geq n_0$, there is an exact sequence
$$0\to\mathcal{O}_{K_n^\prime}^{\otimes\Delta}\otimes_{\mathcal{O}_{K_n}^{\otimes\Delta}}\mathcal{O}_{K_\infty}^{\otimes\Delta}\to\mathcal{O}_{K^\prime_\infty}^{\otimes\Delta}\to T\to0$$
where $T$ is a group killed by $p$. If $r\in\NN_{>0}$, we deduce the exact sequence
$$\Tor_1^{\ZZ}(\ZZ/p^r\ZZ,\mathcal{O}_{K^\prime_\infty}^{\otimes\Delta})\to\Tor_1^{\ZZ}(\ZZ/p^r\ZZ,T)\to\mathcal{O}_{K_n^\prime}^{\otimes\Delta}\otimes_{\mathcal{O}_{K_n}^{\otimes\Delta}}\mathcal{O}_{K_\infty}^{\otimes\Delta}/(p^r)\to\mathcal{O}_{K^\prime_\infty}^{\otimes\Delta}/(p^r)\to T/(p^r)\to0.$$
As $\mathcal{O}_{K^\prime_\infty}^{\otimes\Delta}$ has no $p$-torsion (resp. $T$ is killed by $p$), we have $\Tor_1^{\ZZ}(\ZZ/p^r\ZZ,\mathcal{O}_{K^\prime_\infty}^{\otimes\Delta})=\{0\}$ (resp. $T/(p^r)=T$ and $\Tor_1^{\ZZ}(\ZZ/p^r\ZZ,T)\simeq T$), hence an exact sequence
$$0\to T\to\mathcal{O}_{K_n^\prime}^{\otimes\Delta}\otimes_{\mathcal{O}_{K_n}^{\otimes\Delta}}\mathcal{O}_{K_\infty}^{\otimes\Delta}/(p^r)\xrightarrow{f_r}\mathcal{O}_{K^\prime_\infty}^{\otimes\Delta}/(p^r)\to T\to0.$$
It splits into two exact sequences
$$0\to T\to\mathcal{O}_{K_n^\prime}^{\otimes\Delta}\otimes_{\mathcal{O}_{K_n}^{\otimes\Delta}}\mathcal{O}_{K_\infty}^{\otimes\Delta}/(p^r)\to\im(f_r)\to0$$
$$0\to\im(f_r)\to\mathcal{O}_{K^\prime_\infty}^{\otimes\Delta}/(p^r)\to T\to0.$$
Passing to inverse limits gives exact sequences
$$0\to T\to\varprojlim\limits_r\mathcal{O}_{K_n^\prime}^{\otimes\Delta}\otimes_{\mathcal{O}_{K_n}^{\otimes\Delta}}\mathcal{O}_{K_\infty}^{\otimes\Delta}/(p^r)\to\varprojlim\limits_r\im(f_r)\to0$$
$$0\to\varprojlim\limits_r\im(f_r)\to\mathcal{O}_{L^\prime_\Delta}\to T$$
(the exactness on the right in the first sequence follows from the fact that the constant inverse system $(T)_{r\geq1}$ has the Mittag-Leffler property). As $\mathcal{O}_{K^\prime_n}$ is free over $\mathcal{O}_{K_n}$, so is $\mathcal{O}_{K_n^\prime}^{\otimes\Delta}$ over $\mathcal{O}_{K_n}^{\otimes\Delta}$: this implies that
$$\varprojlim\limits_r\mathcal{O}_{K_n^\prime}^{\otimes\Delta}\otimes_{\mathcal{O}_{K_n}^{\otimes\Delta}}\mathcal{O}_{K_\infty}^{\otimes\Delta}/(p^r)\simeq\mathcal{O}_{K_n^\prime}^{\otimes\Delta}\otimes_{\mathcal{O}_{K_n}^{\otimes\Delta}}\varprojlim\limits_r\mathcal{O}_{K_\infty}^{\otimes\Delta}/(p^r)=\mathcal{O}_{K_n^\prime}^{\otimes\Delta}\otimes_{\mathcal{O}_{K_n}^{\otimes\Delta}}\mathcal{O}_{L_\Delta}.$$
The previous exact sequences thus provide an exact sequence
$$0\to T\to\mathcal{O}_{K_n^\prime}^{\otimes\Delta}\otimes_{\mathcal{O}_{K_n}^{\otimes\Delta}}\mathcal{O}_{L_\Delta}\to\mathcal{O}_{L^\prime_\Delta}\to T.$$
Inverting $p$ gives an isomorphism
$$K_n^{\prime\otimes\Delta}\otimes_{K_n^{\otimes\Delta}}L_\Delta\isomto L^\prime_\Delta,$$
hence an isomorphism $K_{\Delta,\infty}^\prime\otimes_{K_{\Delta,\infty}}L_\Delta\isomto L^\prime_\Delta$, showing the last assertion. The statements on the map $L_\Delta\to L^\prime_\Delta$ follow.
\end{proof}

\begin{prop}\label{propLfidplat}
The $K_{\Delta,\infty}$-algebra $L_\Delta$ is faithfully flat.
\end{prop}

\begin{proof}
$\bullet$ Assume that $K/F_0$ is Galois: so is $K_n/F_0$ for all $n\in\NN$. Recall that the choice of an ordering on $\Delta$ provides an isomorphism $K_{n,\Delta}\simeq K_n^{G_{K_n/F_0}^{\delta-1}}$ (where $G_{K_n/F_0}=\Gal(K_n/F_0)$): the localizations of $K_{n,\Delta}$ at maximal ideals are given by the projections $\iota_{\underline{\sigma}}\colon K_{n,\Delta}\to K_n$ indexed by the elements $\underline{\sigma}\in G_{K_n/F_0}^{\delta-1}$. The corresponding localization $K_n\to K_{n,\iota_{\underline{\sigma}}}\!\otimes_{K_{n,\Delta}}L_\Delta$ is flat (because $K_n$ is a field!): this proves that the map $K_{n,\Delta}\to L_\Delta$ is faithfully flat.

\noindent
Let $I\subset K_{\Delta,\infty}$ be a nonzero ideal. For each $n\in\NN$, put $I_n=I\cap K_{n,\Delta}$ (where $K_{n,\Delta}$ is seen as a subring of $K_{\Delta,\infty}$). By flatness, the natural map $I_n\otimes_{K_{n,\Delta}}L_\Delta\to L_\Delta$ is injective for all $n\in\NN$. The commutative diagram
$$\xymatrix@R=2pt@C=25pt{
I_n\otimes_{K_{n,\Delta}}L_\Delta\ar@{^(->}[rd]\ar[dd]_{i_n} & \\
 & L_\Delta\\
I_{n+1}\otimes_{K_{n+1,\Delta}}L_\Delta\ar@{^(->}[ru] &}$$
thus implies that the natural map $i_n\colon I_n\otimes_{K_{n,\Delta}}L_\Delta\to I_{n+1}\otimes_{K_{n+1,\Delta}}L_\Delta$ is injective. Passing to the inductive limit, this implies that the following composite of the natural maps
$$\varinjlim\limits_n(I_n\otimes_{K_{n,\Delta}}L_\Delta)\to I\otimes_{K_{\Delta,\infty}}L_\Delta\to L_\Delta$$
is injective. As the first map is surjective, it is an isomorphism, so that the natural map $I\otimes_{K_{\Delta,\infty}}L_\Delta\to L_\Delta$ is injective. This proves the flatness in this case.

\noindent
As seen in the proof of proposition \ref{propfond}, if $\underline{\sigma}\in G_{K/F_0}^{\delta-1}$, there is an injection $\widehat{\iota}_{\infty,\underline{\sigma}}\colon K_{\iota_{\underline{\sigma}}}\!\otimes_{K_\Delta}L_\Delta\to\prod\limits_{\substack{\underline{\gamma}\in G_{K_\infty/F_0}^{\delta-1}\\ \underline{\gamma}\mapsto\underline{\sigma}}}L$. It inserts in the commutative diagram
$$\xymatrix{
K_{\iota_{\underline{\sigma}}}\!\otimes_{K_\Delta}K_{\Delta,\infty}\ar[r]\ar@{^(->}[d] & \prod\limits_{\substack{\underline{\gamma}\in G_{K_\infty/F_0}^{\delta-1}\\ \underline{\gamma}\mapsto\underline{\sigma}}}K_\infty\ar@{^(->}[d]\\
K_{\iota_{\underline{\sigma}}}\!\otimes_{K_\Delta}L_\Delta\ar[r]^{\widehat{\iota}_{\infty,\underline{\sigma}}} & \prod\limits_{\substack{\underline{\gamma}\in G_{K_\infty/F_0}^{\delta-1}\\ \underline{\gamma}\mapsto\underline{\sigma}}}L}.$$
As the top horizontal map is faithful (because its components are precisely the localizations of $K_{\iota_{\underline{\sigma}}}\!\otimes_{K_\Delta}K_{\Delta,\infty}$ at its maximal ideals), and as the right vertical map is faithful, so is the composite
$$K_{\iota_{\underline{\sigma}}}\!\otimes_{K_\Delta}K_{\Delta,\infty}\to K_{\iota_{\underline{\sigma}}}\!\otimes_{K_\Delta}L_\Delta\to\prod\limits_{\substack{\underline{\gamma}\in G_{K_\infty/F_0}^{\delta-1}\\ \underline{\gamma}\mapsto\underline{\sigma}}}L$$
hence $K_{\iota_{\underline{\sigma}}}\!\otimes_{K_\Delta}K_{\Delta,\infty}\to K_{\iota_{\underline{\sigma}}}\!\otimes_{K_\Delta}L_\Delta$ is faithful as well. As this holds for all $\underline{\sigma}\in G_{K/F_0}^{\delta-1}$, this implies that $K_{\Delta,\infty}\to L_\Delta$ is faithful.

\noindent
$\bullet$ In general, let $K^\prime\subset\Kbar$ be the Galois closure of $K/F_0$. By what precedes, the map $K^\prime_{\Delta,\infty}\to L^\prime_\Delta$ is faithfully flat. Moreover, $K^\prime_{\Delta,\infty}$ is free as a $K_{\Delta,\infty}$-module, so that $K_{\Delta,\infty}\to L^\prime_\Delta$ is flat as well. Let $I\subset K_{\Delta,\infty}$ be a nonzero ideal. The natural map $I\otimes_{K_{\Delta,\infty}}L^\prime_\Delta\to L^\prime_\Delta$ is thus injective. On the other hand, $L^\prime_\Delta$ is free over $L_\Delta$ since $K^\prime_{\Delta,\infty}$ is over $K_{\Delta,\infty}$ and $K^\prime_{\Delta,\infty}\otimes_{K_{\Delta,\infty}}L_\Delta\isomto L^\prime_\Delta$ by proposition \ref{propextfinet}. This implies that $I\otimes_{K_{\Delta,\infty}}L_\Delta\to I\otimes_{K_{\Delta,\infty}}L^\prime_\Delta$ is injective, so that the composite map $I\otimes_{K_{\Delta,\infty}}L_\Delta\to L^\prime_\Delta$ is injective. As it factors through $L_\Delta$, this implies that $I\otimes_{K_{\Delta,\infty}}L_\Delta\to L_\Delta$ is injective, showing the flatness of $K_{\Delta,\infty}\to L_\Delta$.

\noindent
If $M$ is a $K_{\Delta,\infty}$-module such that $M\otimes_{K_{\Delta,\infty}}L_\Delta=0$, we have
$$(M\otimes_{K_{\Delta,\infty}}K^\prime_{\Delta,\infty})\otimes_{K^\prime_{\Delta,\infty}}L^\prime_\Delta\simeq M\otimes_{K_{\Delta,\infty}}L^\prime_\Delta\simeq(M\otimes_{K_{\Delta,\infty}}L_\Delta)\otimes_{L_\Delta}L^\prime_\Delta=0$$
so that the faithfulness of $L^\prime_\Delta$ over $K^\prime_{\Delta,\infty}$ implies that $M\otimes_{K_{\Delta,\infty}}K^\prime_{\Delta,\infty}=0$, hence $M=0$ since $K^\prime_{\Delta,\infty}$ is free over $K_{\Delta,\infty}$.
\end{proof}

\begin{coro}\label{corocomplf2}
(\cf \cite[Lemme 3.15]{AB2}) Let $X_1$ and $X_2$ be free $L_\Delta$-representations of finite rank of $\Gamma_{K,\Delta}$. The natural maps
\begin{align*}
\Hom_{K_{\infty,\Delta}}(X_{1,\free},X_{2,\free}) &\to \Hom_{L_\Delta}(X_1,X_2)_{\free}\\
\Hom_{\Rep_{K_{\infty,\Delta}}(\Gamma_{K,\Delta})}(X_{1,\free},X_{2,\free}) &\to \Hom_{\Rep_{L_\Delta}(\Gamma_{K,\Delta})}(X_1,X_2)\\
\Ext^1_{\Rep_{K_{\infty,\Delta}}(\Gamma_{K,\Delta})}(X_{1,\free},X_{2,\free}) &\to \Ext^1_{\Rep_{L_\Delta}(\Gamma_{K,\Delta})}(X_1,X_2)
\end{align*}
are bijective.
\end{coro}

\begin{proof}
By corollary \ref{corodecomplf}, we have $L_\Delta\otimes_{K_{\Delta,\infty}}X_{1,\free}\isomto X_1$ and $L_\Delta\otimes_{K_{\infty,\Delta}}X_{2,\free}\isomto X_2$, so that the map
$$L_\Delta\otimes_{K_{\infty,\Delta}}\Hom_{K_{\infty,\Delta}}(X_{1,\free},X_{2,\free})\to\Hom_{L_\Delta}(X_1,X_2)$$
is an isomorphism of $L_\Delta$-representations of $\Gamma_{K,\Delta}$ (since $X_{1,\free}$ and $X_{2,\free}$ are free of finite rank  over $K_{\infty,\Delta}$). If we pick bases of $X_{1,\free}$ and $X_{2,\free}$ over $K_{\infty,\Delta}$, the same proof as in the previous corollary shows that there is a natural isomorphism
$$\Hom_{K_{\infty,\Delta}}(X_{1,\free},X_{2,\free})\isomto\Hom_{L_\Delta}(X_1,X_2)_{\free}$$
of $K_{\infty,\Delta}$-representations of $\Gamma_{K,\Delta}$. Taking invariants under $\Gamma_{K,\Delta}$ gives the $K_\Delta$-linear isomorphism
$$\Hom_{\Rep_{K_{\infty,\Delta}}(\Gamma_{K,\Delta})}(X_{1,\free},X_{2,\free})\isomto\Hom_{\Rep_{L_\Delta}(\Gamma_{K,\Delta})}(X_1,X_2).$$
As $\Ext^1_{\Rep_{K_{\infty,\Delta}}(\Gamma_{K,\Delta})}(X_{1,\free},X_{2,\free})=\H^1(\Gamma_{K,\Delta},X_{\free})$ and $\Ext^1_{\Rep_{L_\Delta}(\Gamma_{K,\Delta})}(X_1,X_2)=\H^1(\Gamma_{K,\Delta},X)$ where we have put $X=\Hom_{L_\Delta}(X_1,X_2)\in\Rep_{L_\Delta}(\Gamma_{K,\Delta})$, the last statement reduces to showing that the natural map
$$\H^1(\Gamma_{K,\Delta},X_{\free})\to\H^1(\Gamma_{K,\Delta},X)$$
is bijective. Let $c\colon\Gamma_{K,\Delta}\to X_{\free}$ be a cocycle whose image in $\H^1(\Gamma_{K,\Delta},X)$ is trivial: there exists $x\in X$ such that $(\forall g\in\Gamma_{K,\Delta})\,c(g)=g(x)-x$. Fix $\mathfrak{B}$ a $K_{\infty,\Delta}$-basis of $X_{\free}$. The action of $\Gamma_{K,\Delta}$ on $X_{\free}$ and $X$ is given, in the basis $\mathfrak{B}$, by a continuous cocycle $U\colon\Gamma_{K,\Delta}\to\GL_d(K_{\infty,\Delta})$ (where $d$ is the rank of $X$ over $L_\Delta$). Let $u_g\in K_{\infty,\Delta}^d$ (resp. $v\in L_\Delta^d$) the column vector whose coefficients are the coordinnates of $c(g)$ (resp. $x$) in the basis $\mathfrak{B}$: we have $u_g=U_gg(v)-v$ \ie $g(v)=U_g^{-1}(u_g+v)$ for all $g\in\Gamma_{K,\Delta}$. Taking $m\in\NN$ large enough, we can assume that $U$ has values in $\GL_d(K_{m,\Delta})$ and $u_g\in K_{m,\Delta}$ for all $g\in\Gamma_{K,\Delta}$. Let $\mathcal{E}$ be the sub-$K_{m,\Delta}$-module of $L_\Delta$ generated by $1$ and the entries of $v$: it is of finite type and stable under the action of $\Gamma_{K,\Delta}$. By proposition \ref{propfond}, we have $\mathcal{E}\subset K_{\infty,\Delta}$, so that $x\in X_{\free}$, which means that the class of $\H^1(\Gamma_{K,\Delta},X_{\free})$ is trivial. This shows the injectivity of the map. To prove the surjectivity, start from a continuous cocycle $c\colon\Gamma_{K,\Delta}\to X$. It defines an extension $\widetilde{X}$ of $L_\Delta$ par $X$. As $L_\Delta$-modules, we have $\widetilde{X}=X\oplus L_\Delta$, the action of $g\in\Gamma_{K,\Delta}$ is given by $g(x,\lambda)=(g(x)+ c(g)g(\lambda),g(\lambda))$. By corollary \ref{corodecomplf}, the $K_{\infty,\Delta}$-module $\widetilde{X}_{\free}$ is free of rank $d+1$ and the map $L_\Delta\otimes_{K_{\infty,\Delta}}\widetilde{X}_{\free}\to X$ is an isomorphism. This shows that the sequence
$$0\to X_{\free}\to\widetilde{X}_{\free}\to K_{\Delta,\infty}\to0$$
is exact when tensored by $L_\Delta$ over $K_{\Delta,\infty}$: as $L_\Delta$ is faithfully flat over $K_{\Delta,\infty}$ by proposition \ref{propLfidplat}, this implies that it is exact, so that $\widetilde{X}_{\free}$ defines an extension of $K_{\infty,\Delta}$ by $X_{\free}$, that corresponds to a cohomology class in $\H^1(\Gamma_{K,\Delta},X_{\free})$ mapping to the class of $c$ in $\H^1(\Gamma_{K,\Delta},X)$.
\end{proof}

\begin{theo}\label{theoequivSenCDelta}
The functors
\begin{center}
\begin{tabular}{lcr}
\begin{minipage}{.35\linewidth}
{\begin{align*}
\Rep^{\free}_{C_\Delta}(G_{K,\Delta}) &\leftrightarrow \Rep^{\free}_{L_\Delta}(\Gamma_{K,\Delta})\\
W &\mapsto W^{H_{K,\Delta}}\\
C_\Delta\otimes_{L_\Delta}X &\mapsfrom X
\end{align*}}
\end{minipage}
&&
\begin{minipage}{.35\linewidth}
{\begin{align*}
\Rep^{\free}_{L_\Delta}(\Gamma_{K,\Delta}) &\leftrightarrow \Rep^{\free}_{K_{\Delta,\infty}}(\Gamma_{K,\Delta})\\
X &\mapsto X_{\free}\\
L_\Delta\otimes_{K_{\Delta,\infty}}Y &\mapsfrom Y
\end{align*}}
\end{minipage}
\end{tabular}
\end{center}
are equivalences of categories.
\end{theo}

\begin{coro}\label{corocohoHKDW}
If $W\in\Rep_{C_\Delta}^{\free}(G_{K,\Delta})$, we have $\H^i(H_{K,\Delta},W)=0$ for all $i\in\NN_{>0}$.
\end{coro}

\begin{proof}
By what precedes, we have $W\simeq C_\Delta\otimes_{K_{\infty,\Delta}}Y$ where $Y=\big(W^{H_{K,\Delta}}\big)_{\free}\in\Rep_{K_{\infty,\Delta}}^{\free}(\Gamma_{K,\Delta})$. As $Y$ is free of finite rank over $K_{\infty,\Delta}$, we have
$$\H^i(H_{K,\Delta},W)\simeq\H^i(H_{K,\Delta},C_\Delta)\otimes_{K_{\infty,\Delta}}Y=0$$
for all $i\in\NN_{>0}$ by theorem \ref{theocohoHC}.
\end{proof}


\subsection{Generalized Sen operators}\label{sectSenoperators}

Let $Y$ be a free $K_{\Delta,\infty}$-representation of rank $d$ of $\Gamma_{K,\Delta}$.

\begin{defi}
The Sen operators of $Y$ are the maps $(\varphi_\alpha)_{\alpha\in\Delta}$ given by
$$\varphi_\alpha(y)=\lim\limits_{\substack{\gamma\in\Gamma_{K_\alpha}\\ \gamma\to\Id}}\frac{\gamma(y)-y}{\log(\chi(\gamma))}.$$
Note that $\varphi_\alpha\in\End_{K_{\Delta,\infty}}(Y)$ for all $\alpha\in\Delta$. As $\Gamma_{K,\Delta}$ is commutative (since $\Gamma_K$ is), the operators $\varphi_\alpha$ commute in $\End_{K_{\Delta,\infty}}(Y)$. Also, each $\varphi_\alpha$ commutes with the action of $\Gamma_{K,\Delta}$ by construction.
\end{defi}

These operators describe the infinitesimal action of $\Gamma_{K,\Delta}$ on $Y$. More precisely, we have

\begin{prop}\label{propSenop}
(\cf \cite[Theorem 4]{Sen}). For all $y\in Y$, there is an open subgroup $\Gamma_{K,\Delta,y}\lhd\Gamma_{K,\Delta}$ such that for all $\alpha\in\Delta$ and $\gamma\in\Gamma_{K,\alpha}\cap\Gamma_{K,\Delta,y}$, we have
$$\gamma(y)=\exp\big(\log(\chi(\gamma))\varphi_\alpha\big)y.$$
\end{prop}

\begin{coro}\label{coroSenop}
The set of elements in $Y$ on which the action of $\Gamma_{K,\Delta}$ is finite (\ie factors through a finite quotient) is $\bigcap\limits_{\alpha\in\Delta}\Ker(\varphi_\alpha)$.
\end{coro}

\begin{proof}
Let $y\in\bigcap\limits_{\alpha\in\Delta}\Ker(\varphi_\alpha)$. By proposition \ref{propSenop}, there exists an open subgroup $\Gamma_{K,\Delta,y}\lhd\Gamma_{K,\Delta}$ such that for all $\alpha\in\Delta$ and $\gamma\in\Gamma_{K,\alpha}\cap\Gamma_{K,\Delta,y}$, we have $\gamma(y)=\exp\big(\log(\chi(\gamma))\varphi_\alpha\big)y=y$, so that the action of $\Gamma_{K,\Delta}$ on $y$ factors through the finite quotient $\Gamma_{K,\Delta}/\Gamma_{K,\Delta,y}$. Conversely, if the action of $\Gamma_{K,\Delta}$ on $y\in Y$ is finite, there exists a open normal subgroup $\Gamma_y\lhd\Gamma_{K,\Delta}$ such that $\gamma(y)=y$ for all $\gamma\in\Gamma_y$. This implies in particular that $\varphi_\alpha(y)=0$ for all $\alpha\in\Delta$.
\end{proof}

\begin{prop}\label{propSenop2}
If $\underline{n}\in\ZZ^\Delta$, the $K_\Delta$-module $Y(\underline{n})^{\Gamma_{K,\Delta}}$ is of finite type, and vanishes for all but finitely many values of $\underline{n}$. Moreover, the map $K_{\Delta,\infty}\otimes_{K_\Delta}Y^{\Gamma_{K,\Delta}}\to Y$ is injective, and its image is precisely the set of elements in $Y$ on which the action of $\Gamma_{K,\Delta}$ is finite.
\end{prop}

\begin{proof}
There exists $m_0\in\NN$ and a basis $\mathfrak{B}=(e_1,\ldots,e_d)$ of $Y$ over $K_{\Delta,\infty}$ such that the cocycle describing the action of $\Gamma_{K,\Delta}$ on $Y$ in $\mathfrak{B}$ has values in $\GL_d(K_{m_0,\Delta})$. If $m\in\NN_{\geq m_0}$, put $Y_m=\bigoplus\limits_{i=1}^dK_{m,\Delta}e_i$. If $\Delta^\prime\subset\Delta$ and $\alpha\in\Delta\setminus\Delta^\prime$, we have a semi-linear action of $\Gamma_{K,\Delta}$ on $\big(K_{m,\Delta\setminus\{\alpha\}}\otimes_{F_0}K_{\infty,\{\alpha\}}\big)\otimes_{K_{m,\Delta}}Y_m$: restricting the action to $\iota_\alpha(\Gamma_K)\subset\Gamma_{K,\Delta}$ provides a $K_\infty$-representation of $\Gamma_K$. By \cite[Theorem 6]{Sen} (\cf also \cite[Proposition 2.6]{Font04}), the $K$-vector space
$$\Big(\big(K_{m,\Delta\setminus\{\alpha\}}\otimes_{F_0}K_{\infty,\{\alpha\}}\big)\otimes_{K_{m,\Delta}}Y_m(\underline{n})\Big)^{\iota_\alpha(\Gamma_K)}$$
is finite dimensional, and vanishes for all but finitely many values of $n_\alpha$ (the other components of $\underline{n}$ being fixed). Making $m$ larger if necessary, we thus may assume that it lies in $Y_m(\underline{n})$. Then we have
\begin{multline*}
\Big(\big(K_{m,\Delta^\prime}\otimes_{F_0}K_{\infty,\Delta\setminus\Delta^\prime}\big)\otimes_{K_{m,\Delta}}Y_m(\underline{n})\Big)^{\iota_\alpha(\Gamma_K)}\\
=\big(K_{\Delta^\prime\cup\{\alpha\}}\otimes_{F_0}K_{\infty,\Delta\setminus(\Delta^\prime\cup\{\alpha\})}\big)\otimes_{K_{\Delta^\prime\cup\{\alpha\}}\otimes_{F_0}K_{m,\Delta\setminus\Delta^\prime\cup\{\alpha\}}}\Big(\big(K_{m,\Delta\setminus\{\alpha\}}\otimes_{F_0}K_{\infty,\{\alpha\}}\big)\otimes_{K_{m,\Delta}}Y_m(\underline{n})\Big)^{\iota_\alpha(\Gamma_K)}\\
\subset \big(K_{m,\Delta^\prime\cup\{\alpha\}}\otimes_{F_0}K_{\infty,\Delta\setminus(\Delta^\prime\cup\{\alpha\})}\big)\otimes_{K_{m,\Delta}}Y_m(\underline{n}).
\end{multline*}
A finite induction thus proves that $Y(\underline{n})^{\Gamma_{K,\Delta}}$ lies in $Y_m(\underline{n})$ if $m\gg0$ (in particular it is of finite rank over $K_\Delta$), and vanishes for all but finitely many values of $\underline{n}$.

Put $Y_m^\prime=\bigcap\limits_{\alpha\in\Delta}\Ker(\varphi_{\alpha|Y_m})\subset Y_m$. This is a $K_{m,\Delta}$-module of finite type endowed with a \emph{discrete} semi-linear action of $\Gamma_{K,\Delta}$. By flatness of $K_{m,\Delta}$ over $K_{m_0,\Delta}$, we have $Y_m^\prime\simeq K_{m,\Delta}\otimes_{K_{m_0,\Delta}}Y_{m_0}^\prime$. Take a $K$-basis of $Y_{m_0}^\prime$: for $m\gg0$, it is fixed by $\Gamma_{K_m,\Delta}$, so that the action of $\Gamma_{K,\Delta}$ on $Y_m^\prime$ factors through $\Gamma_{K,\Delta}/\Gamma_{K_m,\Delta}\simeq\Gal(K_m/K)^\Delta$. By Galois descent, the natural map $K_{m,\Delta}\otimes_{K_\Delta}Y_m^{\prime\Gamma_{K,\Delta}}\to Y_m^\prime$ is an isomorphism. Tensoring with $K_{\Delta,\infty}$ thus shows that
$$K_{\Delta,\infty}\otimes_{K_\Delta}Y_m^{\prime\Gamma_{K,\Delta}}\to Y^\prime:=\bigcap\limits_{\alpha\in\Delta}\Ker(\varphi_\alpha)$$
is an isomorphism when $m\gg0$, implying that $Y_m^{\prime\Gamma_{K,\Delta}}=Y^{\prime\Gamma_{K,\Delta}}$ when $m\gg0$. As $Y^{\Gamma_{K,\Delta}}\subset Y^\prime\subset Y$, we have $Y^{\Gamma_{K,\Delta}}=Y^{\prime,\Gamma_{K,\Delta}}$, hence an isomorphism
$$K_{\Delta,\infty}\otimes_{K_\Delta}Y^{\Gamma_{K,\Delta}}\isomto Y^\prime\subset Y.$$
\end{proof}

The previous proposition can be further refined in order to include all \emph{Hodge-Tate} weights:

\begin{prop}\label{propSenHT}
The natural map
$$\bigoplus\limits_{\underline{n}\in\ZZ^\Delta}K_{\Delta,\infty}(-\underline{n})\otimes_{K_\Delta}\big(Y(\underline{n})\big)^{\Gamma_{K,\Delta}}\to Y$$
is injective.
\end{prop}

\begin{proof}
Keep notations of the previous proof. Let $N\in\NN$ and $E_N=\{-N,\ldots,N\}\subset\ZZ$. Take $N$ large enough so that $\big(Y(\underline{n})\big)^{\Gamma_{K,\Delta}}$ vanishes when $\underline{n}\notin E_N^\Delta$ (\cf proposition \ref{propSenop2}). If $\Delta^\prime\subset\Delta$ and $\underline{n}=(n_\alpha)_{\alpha\in\Delta^\prime}\in\ZZ^{\Delta^\prime}$, we denote $(\underline{n},\underline{0}_{\Delta\setminus\Delta^\prime})\in\ZZ^\Delta$ the element whose component of index $\alpha$ is $n_\alpha$ if $\alpha\in\Delta^\prime$ and $0$ otherwise. Assume $\alpha\in\Delta\setminus\Delta^\prime$. We can see $\big(Y_m(\underline{n},\underline{0}_{\Delta\setminus\Delta^\prime})\big)^{\Gamma_{K,\Delta^\prime}}$ (where we identify $\Gamma_{K,\Delta^\prime}$ with $\Gamma_{K,\Delta^\prime}\times\{\Id\}_{\Delta\setminus\Delta^\prime}\subset\Gamma_{K,\Delta}$)
 as a $K_{\Delta^\prime}\otimes_{F_0}K_{m,\Delta\setminus\Delta^\prime}$-module endowed with a semi-linear action of $\Gamma_{K,\Delta\setminus\Delta^\prime}$. We can view it as a $K_m$-vector space (via the map $i_\alpha\colon K_m\to 
K_{\Delta^\prime}\otimes_{F_0}K_{m,\Delta\setminus\Delta^\prime}$ given by $x\mapsto1\otimes\cdots\otimes1\otimes x\otimes1\otimes\cdots\otimes1$, where $x$ is the factor of index $\alpha$) endowed with a semi-linear action of $\Gamma_K$ (via $\iota_\alpha\colon\Gamma_K\to\Gamma_{K,\Delta\setminus\Delta^\prime}$). By \cite[Lemma 2.3.1]{Hawaii}, the map
$$\bigoplus\limits_{q=-N}^NK_m(-q)\otimes_{K_{i_\alpha}}\big(\big(Y_m(\underline{n},\underline{0}_{\Delta\setminus\Delta^\prime})\big)^{\Gamma_{K,\Delta^\prime}}(q)\big)^{\iota_\alpha(\Gamma_K)}\to\big(Y_m(\underline{n},\underline{0}_{\Delta\setminus\Delta^\prime})\big)^{\Gamma_{K,\Delta^\prime}}$$
is injective. For all $q\in\ZZ$, let $(\underline{n},q,\underline{0}_{\Delta\setminus(\Delta^\prime\cup\{\alpha\})})\in\ZZ^\Delta$ be the element whose components are all equal to those of $(\underline{n},\underline{0}_{\Delta\setminus\Delta^\prime})$, except that of index $\alpha$ which is equal to $q$. As
$$\big(Y_m(\underline{n},q,\underline{0}_{\Delta\setminus(\Delta^\prime\cup\{\alpha\})})\big)^{\Gamma_{K,\Delta^\prime\cup\{\alpha\}}}=\Big(\big(Y_m(\underline{n},\underline{0}_{\Delta\setminus\Delta^\prime})\big)^{\Gamma_{K,\Delta^\prime}}(q)\Big)^{\iota_\alpha(\Gamma_K)},$$
if we tensor with $K_{m,\Delta^\prime}(\underline{n})\otimes_{F_0}K_{\Delta\setminus\Delta^\prime}$ over $K_\Delta$ and sum over all the $\underline{n}\in E_N^{\Delta^\prime}$, we deduce that the map
$$\bigoplus\limits_{\underline{\ell}\in E_N^{\Delta^\prime\cup\{\alpha\}}}K_{m,\Delta^\prime\cup\{\alpha\}}(-\underline{\ell})\otimes_{K_{\Delta^\prime\cup\{\alpha\}}}\big(Y_m(\underline{\ell},\underline{0}_{\Delta\setminus(\Delta^\prime\cup\{\alpha\})})\big)^{\Gamma_{K,\Delta^\prime\cup\{\alpha\}}}\to\bigoplus\limits_{\underline{n}\in E_N^{\Delta^\prime}}\big(Y_m(\underline{n},\underline{0}_{\Delta\setminus\Delta^\prime})\big)^{\Gamma_{K,\Delta^\prime}}$$
is injective. The composition of these maps for growing $\Delta^\prime$ gives the natural map
$$\bigoplus\limits_{\underline{n}\in E_N^\Delta}K_{m,\Delta}(-\underline{n})\otimes_{K_\Delta}\big(Y_m(\underline{n})\big)^{\Gamma_{K,\Delta}}\to Y_m$$
which is thus injective. The inductive limit (as $m\to\infty$) of these is the natural map
$$\bigoplus\limits_{\underline{n}\in\ZZ^\Delta}K_{\Delta,\infty}(-\underline{n})\otimes_{K_\Delta}\big(Y(\underline{n})\big)^{\Gamma_{K,\Delta}}=\bigoplus\limits_{\underline{n}\in E_N^\Delta}K_{\Delta,\infty}(-\underline{n})\otimes_{K_\Delta}\big(Y(\underline{n})\big)^{\Gamma_{K,\Delta}}\to Y$$
which is injective as well.
\end{proof}

\begin{coro}\label{coroSenop2}
If $W\in\Rep_{C_\Delta}^{\free}(G_{K,\Delta})$, there are finitely many $\underline{n}\in\ZZ^\Delta$ such that $\big(W(\underline{n})\big)^{G_{K,\Delta}}\neq\{0\}$, and the natural map
$$\alpha_{\HT,0}(W)\colon\bigoplus\limits_{\underline{n}\in\ZZ^\Delta}C_\Delta(-\underline{n})\otimes_{K_\Delta}\big(W(\underline{n})\big)^{G_{K,\Delta}}\to W$$
is injective.
\end{coro}

\begin{nota}
If $V\in\Rep_{\QQ_p}(G_{K,\Delta})$, we put $\D_{\Sen}(V)=\big((C_\Delta\otimes_{\QQ_p}V)^{H_{K,\Delta}}\big)_{\free}\in\Rep_{K_{\Delta,\infty}}^{\free}(\Gamma_{K,\Delta})$. By what precedes, the infinitesimal action of $\Gamma_{K,\Delta}$ on $\D_{\Sen}(V)$ endows it with $K_{\Delta,\infty}$-linear operators $\varphi_\alpha$ for $\alpha\in\Delta$. Similarly, we put $\D_{C_\Delta}(V)=(C_\Delta\otimes_{\QQ_p}V)^{G_{K,\Delta}}$: this is a $K_\Delta$-module.
\end{nota}

\begin{coro}\label{coroSenop3}
If $V\in\Rep_{\QQ_p}(G_{K,\Delta})$ and $\underline{n}\in\ZZ^\Delta$, the $K_\Delta$-module $\D_{C_\Delta}(V(\underline{n}))$ is of finite type and vanishes for all but finitely many $\underline{n}$. The natural map $K_{\Delta,\infty}\otimes_{K_\Delta}\D_{C_\Delta}(V)\to\D_{\Sen}(V)$ is injective, with image $\bigcap\limits_{\alpha\in\Delta}\Ker(\varphi_\alpha)$. Moreover, the natural map
$$\alpha_{\HT,0}(V)\colon\bigoplus\limits_{\underline{n}\in\ZZ^\Delta}C_\Delta(-\underline{n})\otimes_{K_\Delta}\D_{C_\Delta}(V(\underline{n}))\to C_\Delta\otimes_{\QQ_p}V$$
is injective.
\end{coro}


\section{Multivariable period rings, de Rham and Hodge-Tate representations}

\subsection{Construction and first properties of \texorpdfstring{$\B_{\dR,\Delta}$}{B\unichar{"005F}\unichar{"007B}dR,\unichar{"0394}\unichar{"007D}}}

Let $C^\flat=\varprojlim\limits_{x\mapsto x^p}C$ be the tilt of $C$: this is an algebraically closed complete valued field of characteristic $p$, endowed with a continuous action of $G_K$. Denote by $\mathcal{O}_{C^\flat}$ its ring of integers. Recall there is a surjective map
$$\theta\colon\W\big(\mathcal{O}_{C^\flat}\big)\to\mathcal{O}_C$$
whose kernel is principal, generated by $\xi=p-[\widetilde{p}]$ where $\widetilde{p}=(p,p^{1/p},\ldots)\in\mathcal{O}_{C^\flat}$. It extends into a surjective ring homomorphism $\theta\colon\W\big(\mathcal{O}_{C^\flat}\big)[p^{-1}]\to C$. We denote by $\B_{\dR}^+$ the completion of $\W\big(\mathcal{O}_{C^\flat}\big)[p^{-1}]$ with respect to the $\Ker(\theta)=(\xi)$-adic topology. The map $\theta$ extends into a surjective ring homomorphism $\theta\colon\B_{\dR}^+\to C$. The ring $\B_{\dR}^+$ is a DVR with uniformizer $\xi$ and residue field $C$. Another unifomizer is given by
$$t=\log[\varepsilon]=-\sum\limits_{n=1}^\infty\frac{(1-[\varepsilon])^n}{n}$$
where $\varepsilon=(1,\zeta_p,\zeta_{p^2},\ldots)\in\mathcal{O}_{C^\flat}$ is a compatible sequence of primitive $p^n$-th roots of unity. The natural map $k\to\mathcal{O}_{C^\flat}$ gives rise to a ring homomorphisms $\W(k)[p^{-1}]\to\W\big(\mathcal{O}_{C^\flat}\big)[p^{-1}]\to\B_{\dR}^+$. It extends into a field extension $\Kbar\to\B_{\dR}^+$.

\medskip

\begin{nota}
$\bullet$ We have $\mathcal{O}_{C_\Delta}/(p)\simeq(\mathcal{O}_C/(p))^{\otimes\Delta}$ (the tensor product is taken over $k$). As the Frobenius map on $\mathcal{O}_C/(p)$ is surjective, the same holds on $\mathcal{O}_{C_\Delta}/(p)$. We can thus define the \emph{tilt} of $\mathcal{O}_{C_\Delta}$ as
$$\mathcal{O}_{C_\Delta}^\flat=\varprojlim\limits_{x\mapsto x^p}\mathcal{O}_{C_\Delta}=\Big\{(x^{(n)})_{n\in\NN}\in\mathcal{O}_{C_\Delta}^{\NN}\,;\,(\forall n\in\NN)\,(x^{(n+1)})^p=x^{(n)}\Big\}.$$
This is a perfect $k$-algebra (the map $k\to\mathcal{O}_{C_\Delta}^\flat$ being given by $x\mapsto([x],[x^{1/p}],[x^{1/p^2}],\ldots)$), endowed with an action of $G_{K,\Delta}$, induced by its action on $\mathcal{O}_{C_\Delta}$. Moreover, the map
\begin{align*}
\theta_\Delta\colon\W(\mathcal{O}_{C_\Delta}^\flat) &\to \mathcal{O}_{C_\Delta}\\
(a_n)_{n\in\NN} &\mapsto \sum\limits_{n=0}^\infty p^na_n^{(n)}
\end{align*}
is a $G_{K,\Delta}$-equivariant surjective morphism of $\W(k)$-algebras (\cf \cite[\S5.1]{perfectoid}). By localization, it induces a $G_{K,\Delta}$-equivariant surjective morphisms of $F_0$-algebras
$$\theta_\Delta\colon\W(\mathcal{O}_{C_\Delta}^\flat)\big[\tfrac{1}{p}\big]\to C_\Delta.$$

\noindent
$\bullet$ Put $\big(\mathcal{O}_{C^\flat}\big)^{\otimes\Delta}=\mathcal{O}_{C^\flat}\otimes_k\cdots\otimes_k\mathcal{O}_{C^\flat}$ (where the copies of $\mathcal{O}_{C^\flat}$ are indexed by $\Delta$). If $\alpha\in\Delta$, put
$$\widetilde{p}_\alpha=1\otimes\cdots\otimes1\otimes\widetilde{p}\otimes1\otimes\cdots\otimes1$$
(where $\widetilde{p}$ is the factor of index $\alpha$).
\end{nota}

Denote by $I_{\widetilde{p}}\subset\big(\mathcal{O}_{C^\flat}\big)^{\otimes\Delta}$ the ideal generated by $\{\widetilde{p}_\alpha\}_{\alpha\in\Delta}$.

\begin{lemm}\label{lemmtiltOCDelta}
The ring $\big(\mathcal{O}_{C^\flat}\big)^{\otimes\Delta}$ is $I_{\widetilde{p}}$-adically separated.
\end{lemm}

\begin{proof}
Let $\{e_\lambda\}_{\lambda\in\Lambda}\subset\mathcal{O}_{C^\flat}$ be a subset whose image modulo $\widetilde{p}$ is a $k$-basis of $\mathcal{O}_{C^\flat}/(\widetilde{p})$. Then the familly $\{\widetilde{p}^ie_\lambda\}_{\substack{i\in\NN\\ \lambda\in\Lambda}}$ is linearly independent over $k$, and generates a dense subspace of $\mathcal{O}_{C^\flat}$. In particular, there is an injective $k$-linear map $f\colon\mathcal{O}_{C^\flat}\to k^{\NN\times\Lambda}$ such that $f\big(\widetilde{p}^n\mathcal{O}_{C^\flat}\big)\subset k^{\NN_{\geq n}\times\Lambda}$. The tensor product of these provides an injective $k$-linear map $f_\Delta\colon\big(\mathcal{O}_{C^\flat}\big)^{\otimes\Delta}\to k^{\NN^\Delta\times\Lambda^{\Delta}}$, and $f_\Delta(I_{\widetilde{p}}^n)\subset k^{E_n\times\Lambda^\Delta}$, where $E_n=\Big\{(i_\alpha)_{\alpha\in\Delta}\in\NN^\Delta\,;\,\sum\limits_{\alpha\in\Delta}i_\alpha\geq n\Big\}$. In particular, we have $f_\Delta\Big(\bigcap\limits_{n=0}^\infty I_{\widetilde{p}}^n\Big)\subset\bigcap\limits_{n=0}^\infty k^{E_n\times\Lambda^\Delta}=\{0\}$ (since $\bigcap\limits_{n=0}^\infty E_n=\varnothing$), so that $\bigcap\limits_{n=0}^\infty I_{\widetilde{p}}^n=\{0\}$.
\end{proof}

\begin{prop}\label{proptiltOCDelta}
There is a natural injective morphism of $k$-algebras $\big(\mathcal{O}_{C^\flat}\big)^{\otimes\Delta}\to\mathcal{O}_{C_\Delta}^\flat$, that induces an isomorphism $\widetilde{\pi}_{1,\Delta}\colon\big(\mathcal{O}_{C^\flat}\big)^{\otimes\Delta}/I_{\widetilde{p}}\isomto\mathcal{O}_{C_\Delta}/(p)$. Moreover, $\mathcal{O}_{C_\Delta}^\flat$ is isomorphic to the $I_{\widetilde{p}}$-adic completion of $\big(\mathcal{O}_{C^\flat}\big)^{\otimes\Delta}$.
\end{prop}

\begin{proof}
Recall that for each $m\in\NN$, we have a surjective morphism of $k$-algebras
\begin{align*}
\pi_m\colon\mathcal{O}_{C^\flat} &\to \mathcal{O}_C/(p)\\
(x^{(n)})_{n\in\NN} &\mapsto \overline{x^{(m)}}
\end{align*}
(where the $k$-algebra structure on $\mathcal{O}_C/(p)$ is given by $x\mapsto x^{1/p^m}$), and $\Ker(\pi_m)=\widetilde{p}^{p^m}\mathcal{O}_{C^\flat}$, hence an isomorphism $\widetilde{\pi}_m\colon\mathcal{O}_{C^\flat}/(\widetilde{p}^{p^m})\isomto\mathcal{O}_C/(p)$. Taking the tensor product of these, we get an isomorphism of $k$-algebras $\widetilde{\pi}_{m,\Delta}\colon\big(\mathcal{O}_{C^\flat}/(\widetilde{p}^{p^m})\big)^{\otimes\Delta}\isomto\big(\mathcal{O}_C/(p)\big)^{\otimes\Delta}\simeq\mathcal{O}_{C_\Delta}/(p)$. This means that there is a natural surjective morphism of $k$-algebras $\pi_{m,\Delta}\colon\big(\mathcal{O}_{C^\flat}\big)^{\otimes\Delta}\to\mathcal{O}_{C_\Delta}/(p)$, whose kernel is the ideal $I_{\widetilde{p}}^{(m)}$ generated by $\big\{\widetilde{p}^{p^m}_\alpha\big\}_{\alpha\in\Delta}$. The diagrams
$$\xymatrix@C=35pt@R=15pt{0\ar[r] &I_{\widetilde{p}}^{(m)}\ar[r] & \big(\mathcal{O}_{C^\flat}\big)^{\otimes\Delta}\ar[r]^-{\pi_{m,\Delta}} & \mathcal{O}_{C_\Delta}/(p)\ar[r] &0\\
0\ar[r] &I_{\widetilde{p}}^{(m+1)}\ar[r]\ar@{^(->}[u] & \big(\mathcal{O}_{C^\flat}\big)^{\otimes\Delta}\ar[r]^-{\pi_{m+1,\Delta}}\ar@{=}[u] & \mathcal{O}_{C_\Delta}/(p)\ar[r]\ar[u]_F &0}$$
(where $F$ is the Frobenius map) commutative. Passing to inverse limits provides the exact sequence
$$0\to\varprojlim\limits_m I_{\widetilde{p}}^{(m)}\to\big(\mathcal{O}_{C^\flat}\big)^{\otimes\Delta}\to\mathcal{O}_{C_\Delta}^\flat.$$
As $I_{\widetilde{p}}^{(m)}\subset I_{\widetilde{p}}^{p^m}$, we have $\varprojlim\limits_m I^{(m)}=\bigcap\limits_{m=1}^\infty I_{\widetilde{p}}^{(m)}\subset\bigcap\limits_{m=1}^\infty I_{\widetilde{p}}^{p^m}=\{0\}$ (\cf lemma \ref{lemmtiltOCDelta}), \ie $\big(\mathcal{O}_{C^\flat}\big)^{\otimes\Delta}$ is separated for the $I_{\widetilde{p}}$-adic topology. This provides the injective morphism $\big(\mathcal{O}_{C^\flat}\big)^{\otimes\Delta}\to\mathcal{O}_{C_\Delta}^\flat$ and the isomorphism $\widetilde{\pi}_{1,\Delta}$.

As seen above, the isomorphisms $\widetilde{\pi}_{m,\Delta}$ insert in the commutative diagrams
$$\xymatrix@C=30pt{\big(\mathcal{O}_{C^\flat}\big)^{\otimes\Delta}/I_{\widetilde{p}}^{(m)}\ar[r]^-{\widetilde{\pi}_{m,\Delta}} & \mathcal{O}_{C_\Delta}/(p)\\
\big(\mathcal{O}_{C^\flat}\big)^{\otimes\Delta}/I_{\widetilde{p}}^{(m+1)}\ar[r]^-{\widetilde{\pi}_{m+1,\Delta}}\ar@{->>}[u] & \mathcal{O}_{C_\Delta}/(p)\ar[u]_F}$$
Passing to inverse limits gives an isomorphism $\varprojlim\limits_m\big(\mathcal{O}_{C^\flat}\big)^{\otimes\Delta}/I_{\widetilde{p}}^{(m)}\isomto\mathcal{O}_{C_\Delta}^\flat$, so that $\mathcal{O}_{C_\Delta}^\flat$ is the completion of $\mathcal{O}_{C_\Delta}$ with respect to the topology associated to the family of ideals $(I_{\widetilde{p}}^{(m)})_{m\in\NN_{>0}}$. As $I_{\widetilde{p}}^{p^m\delta}\subset I_{\widetilde{p}}^{(m)}\subset I_{\widetilde{p}}^{p^m}$ for all $m\in\NN_{>0}$, this topology coincides with the $I_{\widetilde{p}}$-adic topology, proving the last part of the proposition.
\end{proof}

\begin{nota}
For $\alpha\in\Delta$, put $\xi_\alpha=p-[\widetilde{p}_\alpha]\in\W(\mathcal{O}_{C_\Delta}^\flat)$. We have $\xi_\alpha\in\Ker(\theta_\Delta)$.
\end{nota}

\begin{coro}\label{coroKerthetaDelta}
The ideal $\Ker(\theta_\Delta)$ is generated by $\{\xi_\alpha\}_{\alpha\in\Delta}$.
\end{coro}

\begin{proof}
We have $\theta_\Delta(\xi_\alpha)=0$ for all $\alpha\in\Delta$, so that the ideal generated by $\{\xi_\alpha\}_{\alpha\in\Delta}$ lies in $\Ker(\theta_\Delta)$. To prove that this inclusion is an equality, it is enough to check it modulo $p$ (because the source and the target of $\theta_\Delta$ are $p$-adically separated and complete), \ie that the kernel of the map
$$\mathcal{O}_{C_\Delta}^\flat\to\mathcal{O}_{C_\Delta}/(p)$$
(given by $x\mapsto x^{(0)}$) is the ideal generated by $I_{\widetilde{p}}$. This kernel contains $I_{\widetilde{p}}$: passing to the quotient, we get a morphism $\mathcal{O}_{C_\Delta}/I_{\widetilde{p}}\to\mathcal{O}_{C_\Delta}$, which is nothing but $\widetilde{\pi}_{1,\Delta}$. We conclude using proposition \ref{proptiltOCDelta}.
\end{proof}

\begin{defi}
$\bullet$ Let $\A_{\inf,\Delta}$ the completion of $\W(\mathcal{O}_{C_\Delta}^\flat)$ with respect to the $\theta_\Delta^{-1}\big(p\mathcal{O}_{C_\Delta}\big)=\langle p,\Ker(\theta_\Delta)\rangle$-adic topology (\cf \cite[D\'efinition 2.2]{Bri2006}). This is a $\W(k)$-algebra endowed with an action of $G_{K,\Delta}$ (because $\theta_\Delta$ is equivariant). The map $\theta_\Delta$ extends to a surjective $G_{K,\Delta}$-equivariant morphism of $\W(k)$-algebras
$$\theta_\Delta\colon\A_{\inf,\Delta}\to\mathcal{O}_{C_\Delta}$$
(because $\mathcal{O}_{C_\Delta}$ is $p$-adically complete).

\noindent
$\bullet$ Let $\B_{\dR,\Delta}^+$ the completion of $\A_{\inf,\Delta}\big[\frac{1}{p}\big]$ with respect to the $\Ker(\theta_\Delta)$-adic topology. This is a $F_0$-algebra endowed with an action of $G_{K,\Delta}$, and the map $\theta_\Delta$ extends to a surjective $\theta_\Delta$-equivariant morphism of $F_0$-algebras
$$\theta_\Delta\colon\B_{\dR,\Delta}^+\to C_\Delta.$$
By definition, we have $\B_{\dR,\Delta}^+=\varprojlim\limits_{m\geq1}\A_{\inf,\Delta}\big[\frac{1}{p}\big]/\Ker(\theta_\Delta)^m$. We endow each quotient $\A_{\inf,\Delta}\big[\frac{1}{p}\big]/\Ker(\theta_\Delta)^m$ with the Banach space topology (induced by the $p$-adic topology on $\A_{\inf,\Delta}$), and $\B_{\dR,\Delta}^+$ with the inverse limit (\ie product) topology, which we call the \emph{canonical topology}. Otherwise mentioned, $\B_{\dR,\Delta}^+$ is considered as a topological ring with this topology.
\end{defi}

\begin{rema}
In contrast with $\B_{\dR}^+$, the ring $\B_{\dR,\Delta}^+$ depends of $k$ when $\delta>1$.
\end{rema}

\begin{prop}
If $F$ is a finite subextension of $\Kbar/F_0$, the $F_0$-algebra structure of $\B_{\dR,\Delta}^+$ extends uniquely to a $F_\Delta$-algebra structure.
\end{prop}

\begin{proof}
This follows from the fact that $F_\Delta$ is a finite \'etale $F_0$-algebra.
\end{proof}

\begin{nota}
If $\alpha\in\Delta$, the $\mathcal{O}_{C^\flat}$-algebra structure of $\mathcal{O}_{C_\Delta}^\flat$ (induced by the map $x\mapsto1\otimes\cdots\otimes1\otimes x\otimes1\otimes\cdots\otimes1$, where $x$ is the factor of index $\alpha$) induces a $\B_{\dR}^+$-algebra structure on $\B_{\dR,\Delta}^+$. All together, this provides a $\B_{\dR}^{+\otimes\Delta}$-algebra structure on $\B_{\dR,\Delta}^+$. In particular, any element $b\in\B_{\dR}^+$ provides an element $b_\alpha\in\B_{\dR,\Delta}^+$ (which is nothing but the image of $1\otimes\cdots\otimes1\otimes b\otimes1\otimes\cdots\otimes1$). For instance, we have the elements $t_\alpha$ and $(p-[\widetilde{p}])_\alpha=\xi_\alpha$.
\end{nota}

\begin{prop}
The ideal $\Ker(\theta_\Delta)\subset\B_{\dR,\Delta}^+$ is generated by $\{t_\alpha\}_{\alpha\in\Delta}$.
\end{prop}

\begin{proof}
This follows from corollary \ref{coroKerthetaDelta} and the fact that $t$ and $\xi$ both generate $\Ker(\theta)$ in $\B_{\dR}^+$, so that they differ by a unit in $\B_{\dR}^+$.
\end{proof}

We endow $\B_{\dR,\Delta}^+$ with the filtration defined by
$$\Fil^i\B_{\dR,\Delta}^+=\Ker(\theta_\Delta)^i$$
for all $i\in\NN$. This filtration is decreasing and exhaustive. Let $\gr\B_{\dR,\Delta}^+=\bigoplus\limits_{i=0}^\infty\gr^i\B_{\dR,\Delta}^+$ be the associated graded ring.

\begin{prop}\label{propsuitereg}
Let $R$ be a ring, $x_1,\ldots,x_n$ a regular sequence in $R$ and $\widehat{R}=\varprojlim\limits_mR/I^m$ the $I$-adic completion of $R$, where $I\subset R$ the ideal generated by $\{x_1,\ldots,x_n\}$. Then the sequence $x_1,\ldots,x_n$ is regular in $\widehat{R}$.
\end{prop}

\begin{proof}
This follows from \cite[Theorem 39 (2)]{Matlis}.
\end{proof}

\begin{lemm}\label{lemmsuitereg}
The sequence $\big(\widetilde{p}_\alpha\big)_{\alpha\in\Delta}$ is regular in $\mathcal{O}_{C_\Delta}^\flat$.
\end{lemm}

\begin{proof}
By proposition \ref{proptiltOCDelta}, $\mathcal{O}_{C_\Delta}^\flat$ is isomorphic to the $I_{\widetilde{p}}$-adic completion of $\big(\mathcal{O}_{C^\flat}\big)^{\otimes\Delta}$: lemma \ref{propsuitereg} shows that it is enough to check that the sequence $\big(\widetilde{p}_\alpha\big)_{\alpha\in\Delta}$ is regular in $\big(\mathcal{O}_{C^\flat}\big)^{\otimes\Delta}$, \ie that for each $\Delta^\prime\subset\Delta$ and $\alpha\in\Delta\setminus\Delta^\prime$, the element $\widetilde{p}_\alpha$ is regular in the quotient 
$\big(\mathcal{O}_{C^\flat}\big)^{\otimes\Delta}/\langle\widetilde{p}_\beta\rangle_{\beta\in\Delta^\prime}\simeq\big(\mathcal{O}_{C^\flat}/\langle\widetilde{p}\rangle\big)^{\otimes\Delta^\prime}\otimes_k\big(\mathcal{O}_{C^\flat}\big)^{\otimes(\Delta\setminus\Delta^\prime)}$. This follows from the regularity of $\widetilde{p}$ in the domain $\mathcal{O}_{C^\flat}$, by taking the tensor product with $\big(\mathcal{O}_{C^\flat}/\langle\widetilde{p}\rangle\big)^{\otimes\Delta^\prime}\otimes_k\big(\mathcal{O}_{C^\flat}\big)^{\otimes(\Delta\setminus(\Delta^\prime\cup\{\alpha\})}$ above the field $k$.
\end{proof}

\begin{prop}\label{propsuiteregBdR}
The sequence $\big(\xi_\alpha\big)_{\alpha\in\Delta}$ is regular in $\B_{\dR,\Delta}^+$. Similarly, the sequence $(t_\alpha)_{\alpha\in\Delta}$ is regular in $\B_{\dR,\Delta}^+$.
\end{prop}

\begin{proof}
By lemma \ref{lemmsuitereg}, the sequence $\big(\widetilde{p}_\alpha\big)_{\alpha\in\Delta}$ is regular in $\mathcal{O}_{C_\Delta}^\flat$. As $p$ is regular in $\W(\mathcal{O}_{C_\Delta}^\flat)$, this implies that $\big(p,\xi_\alpha\big)_{\alpha\in\Delta}$ is regular in $\W(\mathcal{O}_{C_\Delta}^\flat)$. By proposition \ref{propsuitereg}, the same holds in $\A_{\inf,\Delta}$. As $\A_{\inf,\Delta}$ is separated with respect to the $\langle p,\xi_\alpha\rangle_{\alpha\in\Delta}$-adic topology (the same holds for $\A_{\inf,\Delta}/\langle p,\xi_\alpha\rangle_{\alpha\in\Delta^\prime}$ for any subset $\Delta^\prime\subset\Delta$), \cite[Th\'eor\`eme 1]{Algebre10} implies that any permutation of the sequence $(p,\xi_\alpha)_{\alpha\in\Delta}$ is regular in $\A_{\inf,\Delta}$ as well. This shows in particular that $\big(\xi_\alpha\big)_{\alpha\in\Delta}$ is regular in $\A_{\inf,\Delta}$. As localization is an exact functor, this implies that $\big(\xi_\alpha\big)_{\alpha\in\Delta}$ is also regular in $\A_{\inf,\Delta}\big[\frac{1}{p}\big]$. Lemma \ref{propsuitereg} then implies that it is also a regular sequence in $\B_{\dR,\Delta}^+$. The second part of the proposition follows.
\end{proof}

Note that $\gr^0\B_{\dR,\Delta}^+=\B_{\dR,\Delta}^+/\Ker(\theta_\Delta)\simeq C_\Delta$. By an abuse of notation, we still denote by $t_\alpha$ its image in $\gr^1\B_{\dR,\Delta}^+$.

\begin{coro}\label{corosuiteregBdR}
The morphism of $C_\Delta$-algebras $C_\Delta[X_\alpha]_{\alpha\in\Delta}\to\gr\B_{\dR,\Delta}^+$ mapping $X_\alpha$ to $t_\alpha$ for all $\alpha\in\Delta$ is an isomorphism.
\end{coro}

\begin{proof}
This follows from proposition \ref{propsuiteregBdR} and \cite[Th\'eor\`eme 1]{Algebre10}.
\end{proof}

\begin{coro}\label{coroinvariantsBdR}
We have $\H^0(G_{K,\Delta},\B_{\dR,\Delta}^+)\simeq K_\Delta$.
\end{coro}

\begin{proof}
Let $x\in\H^0(G_{K,\Delta},\B_{\dR,\Delta}^+)$: we have $y:=\theta_\Delta(x)\in C_\Delta^{G_{K,\Delta}}=K_\Delta$ since $\theta_\Delta$ is equivariant. This implies that $z:=x-y\in\H^0(G_{K,\Delta},\Fil^1\B_{\dR,\Delta})$. Assume $z\neq0$: as the filtration is separated, there exists $i\in\NN_{>0}$ such that $z\in\Fil^i\B_{\dR,\Delta}^+\setminus\Fil^{i+1}\B_{\dR,\Delta}^+$, so that the image $\overline{z}$ of $z$ in $\gr^i\B_{\dR,\Delta}^+$ is not zero. This implies in particular that $\H^0(G_{K,\Delta},\gr^i\B_{\dR,\Delta}^+)\neq0$. By corollary \ref{corosuiteregBdR}, we have $\gr^i\B_{\dR,\Delta}^+\simeq\bigoplus\limits_{\substack{\underline{n}\in\NN^\Delta\\ \abs{\underline{n}}=i}}C_\Delta(\underline{n})$: by theorem \ref{theocohoGCn}, this implies that $\H^0(G_{K,\Delta},\gr^i\B_{\dR,\Delta}^+)=0$: contradiction. This shows that $z=0$ \ie $x=y\in K_\Delta$. The reverse inclusion is obvious.
\end{proof}

\begin{coro}\label{coroinvariantsBdR2}
For all $r\in\NN$, we have $\H^1(H_{K,\Delta},\gr^r\B_{\dR,\Delta}^+)=\H^1(H_{K,\Delta},\Fil^r\B_{\dR,\Delta}^+)=\{0\}$.
\end{coro}

\begin{proof}
The equalities $\H^1(H_{K,\Delta},\gr^r\B_{\dR,\Delta}^+)=\{0\}$ follow from the fact that $\gr^r\B_{\dR,\Delta}^+\simeq\Sym^r_{C_\Delta}\big(\bigoplus\limits_{\alpha\in\Delta}C_\Delta t_\alpha\big)$(as $G_{K,\Delta}$-modules) by corollary \ref{corosuiteregBdR} and theorem \ref{theocohoHC}.

\noindent
If $s>r$, we prove by induction on $s-r$ that $\H^1(H_{K,\Delta},B_{r,s})=\{0\}$, where $B_{r,s}:=\Fil^r\B_{\dR,\Delta}^+/\Fil^s\B_{\dR,\Delta}^+)=\{0\}$. This is obvious if $s-r=1$ since $B_{r,r+1}=\gr^r\B_{\dR,\Delta}^+$. Assuming that $\H^1(H_{K,\Delta},B_{r,s})=\{0\}$, we have the exact sequence of $G_{K,\Delta}$-modules
$$0\to\gr^s\B_{\dR,\Delta}^+\to B_{r,s+1}\to B_{r,s}\to0$$
so that $\H^1(H_{K,\Delta},\gr^s\B_{\dR,\Delta}^+)\to\H^1(H_{K,\Delta},B_{r,s+1})\to\H^1(H_{K,\Delta},B_{r,s})$ is exact, hence $\H^1(H_{K,\Delta},B_{r,s+1})=0$. To conclude we use the exact sequence
$$0\to\varprojlim_s{}^{(1)}\H^0(\H_{K,\Delta},B_{r,s})\to\H^1(H_{K,\Delta},\Fil^r\B_{dR,\Delta}^+)\to\varprojlim\limits_s\H^1(H_{K,\Delta},B_{r,s})\to0$$
and the fact that the sequence $\big(\H^0(\H_{K,\Delta},B_{r,s})\big)_{s>r}$ has the Mittag-Leffler property (the transition maps are surjective, which follows from the vanishing of the cohomology of $\gr^s\B_{\dR,\Delta}^+$ for all $s\geq r$).
\end{proof}

\begin{defi}
Put $t_\Delta=\prod\limits_{\alpha\in\Delta}t_\alpha\in\B_{\dR,\Delta}^+$, and $\B_{\dR,\Delta}=\B_{\dR,\Delta}^+\big[\frac{1}{t_\Delta}\big]$. As $G_K$ acts on $t$ by multiplication by the cyclotomic character, the group $G_{K,\Delta}$ acts on $t_\Delta$ by multiplication by $\chi_\Delta^{\underline{1}}$, where $\underline{1}$ it the element in $\ZZ^\Delta$ whose components are all equal to $1$. In particular, the action of $G_{K,\Delta}$ on $\B_{\dR,\Delta}^+$ extends into an action on $\B_{\dR,\Delta}$. Also, we endow $\B_{\dR,\Delta}=\varinjlim\limits_it_\Delta^{-i}\B_{\dR,\Delta}^+$ with the inductive limit topology. Finally, we endow $\B_{\dR,\Delta}$ with a filtration indexed by $\ZZ$ by putting
$$\Fil^r\B_{\dR,\Delta}=\varinjlim\limits_{i\geq r}t_\Delta^{-i}\Fil^{i\delta+r}\B_{\dR,\Delta}^+$$
for all $r\in\ZZ$. This defines a decreasing separated and exhaustive filtration on $\B_{\dR,\Delta}$. Note that $\Fil^r\B_{\dR,\Delta}$ is stable under the action of $G_{K,\Delta}$ for all $r\in\ZZ$.
\end{defi}

\begin{prop}\label{propgradBdRD}
We have $\gr^r\B_{\dR,\Delta}\simeq\bigoplus\limits_{\substack{\underline{n}=(n_\alpha)_{\alpha\in\Delta}\in\ZZ^\Delta\\ \sum_\alpha n_\alpha=r}}C_\Delta t_\Delta^{\underline{n}}$, where $t_\Delta^{\underline{n}}=\prod\limits_{\alpha\in\Delta}t_\alpha^{n_\alpha}$ if $\underline{n}=(n_\alpha)_{\alpha\in\Delta}$. In particular, we have
$$\H^0(G_{K,\Delta},\gr^r\B_{\dR,\Delta})=\begin{cases}K_\Delta&\text{ if }r=0\\ 0&\text{ if }r\neq0\end{cases}.$$
\end{prop}

\begin{proof}
By definition we have
$$\gr^r\B_{\dR,\Delta}=\varinjlim\limits_{i\geq r}t_\Delta^{-i}\gr^{i\delta+r}\B_{\dR,\Delta}^+\simeq\varinjlim\limits_{i\geq r}\bigoplus\limits_{\substack{\underline{m}=(m_\alpha)_{\alpha\in\Delta}\in\NN^\Delta\\ \sum_\alpha m_\alpha=i\delta+r}}C_\Delta t_\Delta^{\underline{m}-i\underline{1}}=\bigoplus\limits_{\substack{\underline{n}=(n_\alpha)_{\alpha\in\Delta}\in\ZZ^\Delta\\ \sum_\alpha n_\alpha=r}}C_\Delta t_\Delta^{\underline{n}}.$$
The second part follows from theorems \ref{theocohoGC} and \ref{theocohoGCn}.
\end{proof}

\begin{defi}
Put $\B_{\HT,\Delta}=\gr\B_{\dR,\Delta}$. By what precedes, this is a graded $C_\Delta$-algebra endowed with an action of $G_{K,\Delta}$, and $\B_{\HT,\Delta}\simeq C_\Delta[t_\alpha,t_\alpha^{-1}]_{\alpha\in\Delta}\simeq\bigoplus\limits_{\underline{n}\in\ZZ^\Delta}C_\Delta(\underline{n})$ (as $G_{K,\Delta}$-modules).
\end{defi}

\begin{coro}\label{coroinvBdRD}
We have $\H^0(G_{K,\Delta},\B_{\dR,\Delta})=\H^0(G_{K,\Delta},\B_{\HT,\Delta})=K_\Delta$.
\end{coro}


\subsection{De Rham and Hodge-Tate representations}\label{sectdRHTrepresentations}

Put $G_{F_0}=\Gal(\Kbar/F_0)$. Similarly as we have seen in the proof of proposition \ref{propfond}, the choice of an ordering $\alpha_1<\cdots<\alpha_{\delta}$ of $\Delta$ provides an injective and $G_{K,\Delta}$-equivariant ring homomorphism
$$\widehat{\iota}\colon C_\Delta\to\prod\limits_{\underline{\gamma}\in G_{F_0}^{\delta-1}}C$$
(induced by the map sending $x_1\otimes\cdots\otimes x_{\delta}$ to the element whose component of index $\underline{\gamma}=(\gamma_2,\ldots,\gamma_{\delta})\in G_{F_0}^{\delta-1}$ is $x_1\gamma_2(x_2\gamma_3(x_3\cdots \gamma_{\delta}(x_{\delta})\cdots))$), where the action of $\underline{g}=(g_1,\ldots,g_{\delta})\in G_{K,\Delta}$ on the LHS is induced by the action defined by $\underline{g}\cdot(x_1\otimes\cdots x_{\delta})=g_1(x_1)\otimes\cdots\otimes g_{\delta}(x_{\delta})$ on $\Kbar^{\otimes\Delta}$, and that on the RHS is given by $\underline{g}\cdot(x_{\underline{\gamma}})_{\underline{\gamma}\in G_K^{\delta-1}}=\big(g_1x_{\underline{g}\cdot\underline{\gamma}}\big)_{\underline{\gamma}}$, where $\underline{g}\cdot\underline{\gamma}=(g_1^{-1}\gamma_2g_2,g_2^{-1}\gamma_3g_3,\ldots,g_{\delta-1}^{-1}\gamma_{\delta}g_{\delta})$ if $\underline{\gamma}=(\gamma_2,\ldots,\gamma_{\delta})$. The tilt of this map gives an injective $G_{K,\Delta}$-equivariant ring homomorphism $\widehat{\iota}^\flat\colon C_\Delta^\flat\to\prod\limits_{\underline{\gamma}\in G_{F_0}^{\delta-1}}C^\flat$. The diagram
$$\xymatrix{
\W(\mathcal{O}_{C_\Delta}^\flat)\ar@{^(->}[r]^-{\W(\widehat{\iota}^\flat)}\ar[d]_{\theta_\Delta} & \prod\limits_{\underline{\gamma}\in G_{F_0}^{\delta-1}}\W(\mathcal{O}_{C^\flat})\ar[d]^{\prod_{\underline{\gamma}}\theta}\ar@{^(->}[r] & \prod\limits_{\underline{\gamma}\in G_{F_0}^{\delta-1}}\B_{\dR}^+\ar[ld]\\
\mathcal{O}_{C_\Delta}\ar[r]^-{\widehat{\iota}} & \prod\limits_{\underline{\gamma}\in G_{F_0}^{\delta-1}}\mathcal{O}_C & }$$
is commutative. This induces an injective and $G_{K,\Delta}$-equivariant ring homomorphism
$$\iota_{\dR}\colon\B_{\dR,\Delta}^+\to\prod\limits_{\underline{\gamma}\in G_{F_0}^{\delta-1}}\B_{\dR}^+.$$
The component of index $\underline{\gamma}$ of the image of $t_{\alpha_i}$ by the previous map is $\chi(\gamma_2\cdots\gamma_i)t$: this shows that it extends into an injective and $G_{K,\Delta}$-equivariant map
$$\iota_{\dR}\colon\B_{\dR,\Delta}\to\prod\limits_{\underline{\gamma}\in G_{F_0}^{\delta-1}}\B_{\dR}$$
(which shows in particular that $\B_{\dR,\Delta}$ is reduced).

\begin{nota}
If $V\in\Rep_{\QQ_p}(G_{K,\Delta})$, we put $\D_{\dR}(V)=(\B_{\dR,\Delta}\otimes_{\QQ_p}V)^{G_{K,\Delta}}$. This is a $K_\Delta$-module which is endowed with the filtration induced by that on $\B_{\dR,\Delta}$ (\ie $\Fil^r\D_{\dR}(V)=(\Fil^r\B_{\dR,\Delta}\otimes_{\QQ_p}V)^{G_{K,\Delta}}$). By $\B_{\dR,\Delta}$-linearity, the inclusion $\D_{\dR}(V)\subset\B_{\dR,\Delta}\otimes_{\QQ_p}V$ extends into a $\B_{\dR,\Delta}$-linear and $G_{K,\Delta}$-equivariant map
$$\alpha_{\dR}(V)\colon\B_{\dR,\Delta}\otimes_{K_\Delta}\D_{\dR}(V)\to\B_{\dR,\Delta}\otimes_{\QQ_p}V.$$
We define $\D_{\HT}(V)$ and $\alpha_{\HT}(V)$ similarly.
\end{nota}

\begin{prop}\label{propalphainj}
The $K_\Delta$-modules $\D_{\dR}(V)$ and $\D_{\HT}(V)$ are of finite type, and the maps $\alpha_{\dR}(V)$ and $\alpha_{\HT}(V)$ are injective.
\end{prop}

\begin{proof}
$\bullet$ Assume that $K/F_0$ is Galois with group $G_{K/F_0}$, and put $\widetilde{\B}_{\dR,\Delta}=\prod\limits_{\underline{\gamma}\in G_{F_0}^{\delta-1}}\B_{\dR}$: as recalled above, there is an injective $G_{K,\Delta}$-equivariant ring homomorphism $\iota_{\dR}\colon\B_{\dR,\Delta}\to\widetilde{\B}_{\dR,\Delta}$ (that depends on the choice of an ordering $\alpha_1<\cdots<\alpha_{\delta}$ on $\Delta$). Put $\widetilde{\D}_{\dR}(V)=\big(\widetilde{\B}_{\dR,\Delta}\otimes_{\QQ_p}V\big)^{G_{K,\Delta}}\subset\prod\limits_{\underline{\gamma}\in G_{F_0}^{\delta-1}}\B_{\dR}\otimes_{\QQ_p}V$:  the map $\iota_{\dR}$ induces a $K_\Delta$-linear injective map $\iota\colon\D_{\dR}(V)\to\widetilde{\D}_{\dR}(V)$. If $(x_{\underline{\gamma}})_{\underline{\gamma}}\in\prod\limits_{\underline{\gamma}\in G_{F_0}^{\delta-1}}\B_{\dR}\otimes_{\QQ_p}V$ and $\underline{g}=(g_1,\ldots,g_{\delta})\in G_{K,\Delta}$, we have $\underline{g}\cdot(x_{\underline{\gamma}})_{\underline{\gamma}}=(g_1(x_{\underline{g}\cdot\underline{\gamma}}))_{\underline{\gamma}}$, thus $(x_{\underline{\gamma}})_{\underline{\gamma}}\in\widetilde{\D}_{\dR}(V)$ if and only if $g_1(x_{\underline{g}\cdot\underline{\gamma}})=x_{\underline{\gamma}}$ for all $\underline{g}\in G_{K,\Delta}$ and $\underline{\gamma}\in G_{F_0}^{\delta-1}$. If $g\in G_K$ and $\underline{\gamma}_0=(\gamma_2,\ldots,\gamma_{\delta})\in G_{F_0}^{\delta-1}$, define the element $\underline{g}=(g_1,\ldots,g_{\delta})\in G_{K,\Delta}$ by $g_1=g$ and $g_i=\gamma_i^{-1}g_{i-1}\gamma_i$ for all $i\in\{2,\ldots,\delta\}$ (we have indeed $\gamma_i^{-1}g_{i-1}\gamma_i\in G_K$ since $G_K$ is normal in $G_{F_0}$ because $K/F_0$ is Galois). Then we have $\underline{g}\cdot\underline{\gamma}_0=\underline{\gamma}_0$, and the component of index $\underline{\gamma}_0$ of
$\underline{g}\big(x_{\underline{\gamma}}\big)_{\underline{\gamma}}$ is $g(x_{\underline{\gamma}_0})$: this shows that $x_{\underline{\gamma}_0}\in\B_{\dR}\otimes_{\QQ_p}V$ is fixed under $G_K$, \ie that $x_{\underline{\gamma}_0}\in\D_{\dR,\alpha_1}(V)=(\B_{\dR}\otimes_{\QQ_p}V)^{G_K}$ (where the action of $G_K$ on $V$ is via the map $\iota_{\alpha_1}\colon G_K\to G_{K,\Delta}$). This implies that $(x_{\underline{\gamma}})_{\underline{\gamma}}$ is fixed by $G_{K,\Delta}$ if and only if its components all belong to $\D_{\dR,\alpha_1}(V)$ and $x_{\underline{g}\cdot\underline{\gamma}}=x_{\underline{\gamma}}$ for all $\underline{\gamma}\in G_{F_0}^{\delta-1}$ and $\underline{g}=(\Id_{\Kbar},g_2,\ldots,g_{\delta})\in\{\Id_{\Kbar}\}\times G_K^{\delta-1}$. As $G_K^{\delta-1}$ acts transitively on those $\underline{\gamma}$ that map to a fixed $\underline{\sigma}\in G_{K/F_0}^{\delta-1}$, this shows that
$$\widetilde{\D}_{\dR}(V)=\prod\limits_{\underline{\sigma}\in G_{K/F_0}^{\delta-1}}\D_{\dR,\alpha_1}(V)\subset\prod\limits_{\underline{\gamma}\in G_{F_0}^{\delta-1}}\D_{\dR,\alpha_1}(V)$$
(where we embed the factor $\D_{\dR,\alpha_1}(V)$ of index $\underline{\sigma}$ diagonally in $\prod\limits_{\substack{\underline{\gamma}\in G_{F_0}^{\delta-1}\\ \underline{\gamma}\mapsto\underline{\sigma}}}\D_{\dR,\alpha_1}(V)$). As $\D_{\dR,\alpha_1}(V)$ is a finite dimensional $K$-vector space, this shows in particular that $\widetilde{\D}_{\dR}(V)$ is a $K_\Delta$-module of finite type.

The natural map $\B_{\dR}\otimes_K\D_{\dR,\alpha_1}(V)\to\B_{\dR}\otimes_{\QQ_p}V$ is injective: so are the maps
$$\Big(\prod\limits_{\substack{\underline{\gamma}\in G_{F_0}^{\delta-1}\\ \underline{\gamma}\mapsto\underline{\sigma}}}\B_{\dR}\Big)\otimes_K\D_{\dR,\alpha_1}(V)\to\Big(\prod\limits_{\substack{\underline{\gamma}\in G_{F_0}^{\delta-1}\\ \underline{\gamma}\mapsto\underline{\sigma}}}\B_{\dR}\Big)\otimes_{\QQ_p}V$$
for all $\underline{\sigma}\in G_{K/F_0}^{\delta-1}$, so that the map
$$\widetilde{\alpha}_{\dR}(V)\colon\widetilde{\B}_{\dR,\Delta}\otimes_{K_\Delta}\widetilde{\D}_{\dR}(V)\to\widetilde{\B}_{\dR,\Delta}\otimes_{\QQ_p}V$$
is injective (since its localizations via the projection maps $\iota_{\underline{\sigma}}\colon K_\Delta\to K$ are precisely the maps above). The diagram
$$\xymatrix@C=35pt{
\B_{\dR,\Delta}\otimes_{K_\Delta}\D_{\dR}(V)\ar[r]^-{\alpha_{\dR}(V)}\ar@{^(->}[d]_{\iota_{\dR}\otimes1} & \B_{\dR,\Delta}\otimes_{\QQ_p}V\ar@{^(->}[dd]^{\iota_{\dR}\otimes1}\\
\widetilde{\B}_{\dR,\Delta}\otimes_{K_\Delta}\D_{\dR}(V)\ar@{^(->}[d]_{1\otimes \iota} & \\
\widetilde{\B}_{\dR,\Delta}\otimes_{K_\Delta}\widetilde{\D}_{\dR}(V)\ar@{^(->}[r]^-{\widetilde{\alpha}_{\dR}(V)} & \widetilde{\B}_{\dR,\Delta}\otimes_{\QQ_p}V}$$
is commutative: this implies that $\alpha_{\dR}(V)$ is injective.

\medskip

\noindent
$\bullet$ If $K/F_0$ is not assumed to be Galois, let $K^\prime\subset\Kbar$ the Galois closure of $K$ over $F_0$: what precedes shows that $\D_{\dR,K^\prime}(V):=(\B_{\dR,\Delta}\otimes_{\QQ_p}V)^{G_{K^\prime,\Delta}}$ is a $K^\prime_\Delta$-module of finite type (hence projective of finite rank) endowed with a semi-linear action of $\Gal(K^\prime/K)^\Delta$, and that the natural map
$$\alpha_{\dR,K^\prime}(V)\colon\B_{\dR,\Delta}\otimes_{K^\prime_\Delta}\D_{\dR,K^\prime}(V)\to\B_{\dR,\Delta}\otimes_{\QQ_p}V$$
is injective. By Galois descent, we have $\D_{\dR,K^\prime}(V)\simeq K^\prime_\Delta\otimes_{K_\Delta}\D_{\dR}(V)$ as $\Gal(K^\prime/K)^\Delta$-modules, and $\D_{\dR}(V)$ is of finite type over $K_\Delta$. The commutative diagram
$$\xymatrix@R=10pt@C=40pt{
\B_{\dR,\Delta}\otimes_{K_\Delta}\D_{\dR}(V)\ar[rd]^-{\alpha_{\dR}(V)}\eq[dd] & \\
& \B_{\dR,\Delta}\otimes_{\QQ_p}V\\
\B_{\dR,\Delta}\otimes_{K^\prime_\Delta}\D_{\dR,K^\prime}(V)\ar[ru]_-{\alpha_{\dR,K^\prime}(V)} & }$$
this implies that $\alpha_{\dR}(V)$ is injective.

\medskip

\noindent
$\bullet$ Consider now the Hodge-Tate side. By corollary \ref{coroSenop3}, the $K_\Delta$-module
$$\D_{\HT}(V)=\Big(\bigoplus\limits_{\underline{n}\in\ZZ^\Delta}C_\Delta(\underline{n})\otimes_{\QQ_p}V\Big)^{G_{K,\Delta}}=\bigoplus\limits_{\underline{n}\in\ZZ^\Delta}(C_\Delta(\underline{n})\otimes_{\QQ_p}V)^{G_{K,\Delta}}$$
is of finite type (the sum is finite and all but finitely many factors are zero). Moreover, the natural map
$$\alpha_{\HT,0}(V)\colon\bigoplus\limits_{\underline{n}\in\ZZ^\Delta}C_\Delta(-\underline{n})\otimes_{K_\Delta}(C_\Delta(\underline{n})\otimes_{\QQ_p}V)^{G_{K,\Delta}}\to C_\Delta\otimes_{\QQ_p}V$$
is injective and $G_{K,\Delta}$-equivariant. If we tensor with $C_\Delta(\underline{m})$ over $C_\Delta$ and sum over all $\underline{m}\in\ZZ^\Delta$, this shows that the maps in the diagram
$$\xymatrix{\bigoplus\limits_{\underline{n},\underline{m}\in\ZZ^\Delta}C_\Delta(\underline{m}-\underline{n})\otimes_{K_\Delta}(C_\Delta(\underline{n})\otimes_{\QQ_p}V)^{G_{K,\Delta}}\ar[r]\eq[d] & \bigoplus\limits_{\underline{m}\in\ZZ^\Delta}C_\Delta(\underline{m})\otimes_{\QQ_p}V\eq[d]\\
\bigoplus\limits_{\underline{\ell}\in\ZZ^\Delta}C_\Delta(\underline{\ell})\otimes_{K_\Delta}\D_{\HT}(V)\ar[r]^-{\alpha_{\HT}(V)} & \B_{\HT}\otimes_{\QQ_p}V}$$
are injective.
\end{proof}

\begin{defi}
A $p$-adic representation $V$ of $G_{K,\Delta}$ is said \emph{de Rham} (resp. \emph{Hodge-Tate}) when the map $\alpha_{\dR}(V)$ (resp. $\alpha_{\HT}(V)$) is bijective. We denote by $\Rep_{\dR}(G_{K,\Delta})$ (resp. $\Rep_{\HT}(G_{K,\Delta})$) the full subcategory of $\Rep_{\QQ_p}(G_{K,\Delta})$ whose objects are de Rham (resp. Hodge-Tate) representations.
\end{defi}

Recall that $\D_{\dR}(V)$ is equipped with a decreasing filtration $\Fil^\bullet\D_{\dR}(V)$: denote by $\gr\D_{\dR}(V)$ the corresponding graded module. Note that the map $\alpha_{\dR}(V)$ is compatible with filtrations ($\B_{\dR,\Delta}\otimes_{K_\Delta}\D_{\dR}(V)$ being endowed with the tensor product filtration, given by $\Fil^r(\B_{\dR,\Delta}\otimes_{K_\Delta}\D_{\dR}(V))=\sum\limits_{i\in\ZZ}\Fil^i\B_{\dR,\Delta}\otimes_{K_\Delta}\Fil^{r-i}\D_{\dR}(V)$ for all $i\in\ZZ$).

\begin{prop}\label{propdRHT}
The filtration $\Fil^\bullet\D_{\dR}(V)$ is separated and exhaustive. There is a canonical injective map $\gr\D_{\dR}(V)\to\D_{\HT}(V)$, and $\alpha_{\dR}(V)$ is strictly compatible with filtrations. Moreover, if $V$ is de Rham, then it is Hodge-Tate and the map $\gr\D_{\dR}(V)\to\D_{\HT}(V)$ is an isomorphism.
\end{prop}

\begin{proof}
The first assertion follows from the corresponding fact on $\B_{\dR,\Delta}$ and the finiteness of $\D_{\dR}(V)$ as a $K_\Delta$-module. If $r\in\ZZ$, the exact sequence
$$0\to\Fil^{r+1}\B_{\dR,\Delta}\to\Fil^r\B_{\dR,\Delta}\to\gr^r\B_{\dR,\Delta}\to0$$
tensored with $V$ induces the exact sequence
$$0\to\Fil^{r+1}\D_{\dR}(V)\to\Fil^r\D_{\dR}(V)\to(\gr^r\B_{\dR,\Delta}\otimes_{\QQ_p}V)^{G_{K,\Delta}}$$
\ie an injective map $\gr^r\D_{\dR}(V)\to\bigoplus\limits_{\substack{\underline{n}=(n_\alpha)_{\alpha\in\Delta}\in\ZZ^\Delta\\ \sum_\alpha n_\alpha=r}}(C_\Delta(\underline{n})\otimes_{\QQ_p}V)^{G_{K,\Delta}}$ (\cf proposition \ref{propgradBdRD}). Summing over $r\in\ZZ$ provides an injective $K_\Delta$-linear map $i_V\colon\gr\D_{\dR}(V)\to\D_{\HT}(V)$.

As $\alpha_{\dR}(V)$ is compatible with filtrations, it induces a $K_\Delta$-linear map
$$\gr\alpha_{\dR}(V)\colon\gr(\B_{\dR,\Delta}\otimes_{K_\Delta}\D_{\dR}(V))\to\gr\B_{\dR,\Delta}\otimes_{K_\Delta}V.$$
We have $\gr\B_{\dR,\Delta}=\B_{\HT,\Delta}$ and $\gr(\B_{\dR,\Delta}\otimes_{K_\Delta}\D_{\dR}(V))\simeq\gr\B_{\dR,\Delta}\otimes_{K_\Delta}\gr\D_{\dR}(V)\simeq\B_{\HT}\otimes_{K_\Delta}\gr\D_{\dR}(V)$ (as $K_\Delta$ is a product of fields, the proof of this isomorphism reduces to the case of tensor products of filtered vector spaces). Via these isomorphisms, we have the commutative diagram
$$\xymatrix@C=50pt{
\B_{\HT}\otimes_{K_\Delta}\gr\D_{\dR}(V)\ar[r]^-{\gr\alpha_{\dR}(V)}\ar@{^(->}[d]_{1\otimes i_V} & \B_{\HT,\Delta}\otimes_{\QQ_p}V\\
\B_{\HT,\Delta}\otimes_{K_\Delta}\D_{\HT}(V)\ar@{^(->}[ru]_-{\alpha_{\HT}(V)} & }$$
which implies that $\gr\alpha_{\dR}(V)$ is injective, so that $\alpha_{\dR}(V)$ is strictly compatible with filtrations.

Assume $V$ is de Rham, so that $\alpha_{\dR}(V)$ is an isomorphism. Then $\gr\alpha_{\dR}(V)$ is an isomorphism: the preceding diagram shows that $\alpha_{\HT}(V)$ is surjective: it is an isomorphism by proposition \ref{propalphainj}, and $V$ is Hodge-Tate. This implies that $1\otimes i_V$ is an isomorphism. As $\B_{\HT,\Delta}$ is faithfully flat over $K_\Delta$ (because $C_\Delta$ is), this shows that $i_V$ is an isomorphism.
\end{proof}

\begin{prop}\label{propdRHT2}
If $V\in\Rep_{\QQ_p}(G_{K,\Delta})$ is de Rham (resp. Hodge-Tate), then $\D_{\dR}(V)$ (resp. $\D_{\HT}(V)$) is free of rank $\dim_{\QQ_p}(V)$ over $K_\Delta$.
\end{prop}

\begin{proof}
Assume $V$ is de Rham: the map $\alpha_{\dR}(V)\colon\B_{\dR,\Delta}\otimes_{K_\Delta}\D_{\dR}(V)\to\B_{\dR,\Delta}\otimes_{\QQ_p}V$ is an isomorphism. Use remark \ref{remaRR} and its notations: for each $i\in I$, its localization
$$E_i\otimes\alpha_{\dR}(V)\colon(E_i\otimes_{K_\Delta}\B_{\dR,\Delta})\otimes_{E_i}(E_i\otimes_{K_\Delta}\D_{\dR}(V))\to(E_i\otimes_{K_\Delta}\B_{\dR,\Delta})\otimes_{\QQ_p}V$$
is a $E_i\otimes_{K_\Delta}\B_{\dR,\Delta}$-linear isomorphism. This implies that
$$\dim_{E_i}(E_i\otimes_{K_\Delta}\D_{\dR}(V))=\dim_{\QQ_p}(V).$$
As it holds for all $i\in I$, this shows that $\D_{\dR}(V)$ is free as a $K_\Delta$-module. The proof of the Hodge-Tate case is the same.
\end{proof}

\medskip

\noindent
Question : Is the converse true, \ie is it true that if $\D_{\dR}(V)$ is free of rank $\dim_{\QQ_p}(V)$ over $K_\Delta$, then $V$ is de Rham?


\section{Sen theory for \texorpdfstring{$\B_{\dR,\Delta}^+$}{B\unichar{"005F}\unichar{"007B}dR,\unichar{"0394}\unichar{"007D}\unichar{"207A}-representations}-representations}

\subsection{Almost \'etale descent}

Put $\L_{\dR,\Delta}^+=\H^0(H_{K,\Delta},\B_{\dR,\Delta}^+)$. We have
$$\l_{\dR,\Delta}^+:=K_{\Delta,\infty}[\![t_\alpha]\!]_{\alpha\in\Delta}\subset\L_{\dR,\Delta}^+.$$
These subrings of $\B_{\dR,\Delta}^+$ are endowed with the filtration induced by the latter. In the sequel, we follow rather closely \cite{AB2}. 

\begin{lemm}\label{lemmgrL}
For all $r\in\NN$, we have natural isomorphisms
$$\xymatrix@R=15pt{
\Sym^r_{K_{\Delta,\infty}}\Big(\bigoplus\limits_{\alpha\in\Delta}K_{\Delta,\infty}t_\alpha\Big)\ar[r]^-\sim\ar@{^(->}[d] & \gr^r\l_{\dR,\Delta}^+\ar@{^(->}[d]\\
\Sym^r_{L_\Delta}\Big(\bigoplus\limits_{\alpha\in\Delta}L_\Delta t_\alpha\Big)\ar[r]^-\sim\ar@{^(->}[d] & \gr^r\L_{\dR,\Delta}^+\ar@{^(->}[d]\\
\Sym^r_{C_\Delta}\Big(\bigoplus\limits_{\alpha\in\Delta}C_\Delta t_\alpha\Big)\ar[r]^-\sim & \gr^r\B_{\dR,\Delta}^+.}$$
Moreover, the map $\l_{\dR,\Delta}^+/\Fil^r\l_{\dR,\Delta}^+\to\L_{\dR,\Delta}^+/\Fil^r\L_{\dR,\Delta}^+$ is faithfully flat.
\end{lemm}

\begin{proof}
The bottom map is an isomorphism by corollary \ref{corosuiteregBdR}. This implies that the filtration on $\l_{\dR,\Delta}^+$ is given by the powers of the ideal generated by $(t_\alpha)_{\alpha\in\Delta}$, showing that the top map is an isomorphism as well. To check that the middle map is also an isomorphism, we start from the exact sequence
$$0\to\Fil^{r+1}\B_{\dR,\Delta}^+\to\Fil^r\B_{\dR,\Delta}^+\to\gr^r\B_{\dR,\Delta}^+\to0.$$
Taking invariants under $H_{K,\Delta}$ gives the exact sequence
$$0\to\Fil^{r+1}\L_{\dR,\Delta}^+\to\Fil^r\L_{\dR,\Delta}^+\to\H^0(H_{K,\Delta},\gr^r\B_{\dR,\Delta}^+)\to\H^1(H_{K,\Delta},\Fil^{r+1}\B_{\dR,\Delta}^+).$$
By corollary \ref{coroinvariantsBdR2}, we have $\H^1(H_{K,\Delta},\Fil^{r+1}\B_{\dR,\Delta}^+)=\{0\}$, so that the natural map
$$\gr^r\L_{\dR,\Delta}^+\to\H^0(H_{K,\Delta},\gr^r\B_{\dR,\Delta}^+)\simeq\Sym^r_{L_\Delta}\Big(\bigoplus\limits_{\alpha\in\Delta}L_\Delta t_\alpha\Big)$$
is an isomorphism. The last statement follows from proposition \ref{propLfidplat}.
\end{proof}

\begin{coro}\label{coroextfinet}
Under the assumptions of proposition \ref{propextfinet}, put
$$\L_{\dR,\Delta}^{\prime+}=\H^0(H_{K^\prime,\Delta},\B_{\dR,\Delta}^+)\text{ and }\l_{\dR,\Delta}^{\prime+}=K_{\Delta,\infty}^\prime[\![t_\alpha]\!]_{\alpha\in\Delta}.$$
The natural maps $\L_{\dR,\Delta}^+\to\L_{\dR,\Delta}^{\prime+}$ and $\l_{\dR,\Delta}^+\to\l_{\dR,\Delta}^{\prime+}$ are finite \'etale, and Galois with group $\Gal(K_\infty^\prime/K_\infty)^\Delta$ when $K_\infty^\prime/K_\infty$ is Galois.
\end{coro}

\begin{proof}
By lemma \ref{lemmgrL} and proposition \ref{propextfinet}, the natural maps $K_\infty^{\prime\otimes\Delta}\otimes_{K_\infty^{\otimes\Delta}}\L_{\dR,\Delta}^+\to\L_{\dR,\Delta}^{\prime+}$ and $K_\infty^{\prime\otimes\Delta}\otimes_{K_\infty^{\otimes\Delta}}\l_{\dR,\Delta}^+\to\l_{\dR,\Delta}^{\prime+}$ induce isomorphisms on graded rings: these are isomorphisms. The statements thus follow from proposition \ref{propextfinet}.
\end{proof}

\begin{prop}\label{propSenBdRH}
(\cf \cite[Proposition 3.4]{AB2}) The maps
\begin{align*}
\varinjlim\limits_{\substack{H\lhd H_{K,\Delta}\\ H\text{ \emph{open}}}}\H^1(H_{K,\Delta}/H,\GL_d(\B_{\dR,\Delta}^{+H})) &\to \H^1(H_{K,\Delta},\GL_d(\B_{\dR,\Delta}^+))\\
\H^1(\Gamma_{K,\Delta},\GL_d(\L_{\dR,\Delta}^+)) &\to \H^1(G_{K,\Delta},\GL_d(\B_{\dR,\Delta}^+))
\end{align*}
induced by inflation maps are bijective.
\end{prop}

\begin{proof}
Being inductive limits of inflation maps, they are injective. Let $U\colon H_{K,\Delta}\to\GL_d(\B_{\dR,\Delta}^+)$ be a continuous cocycle. The composite $\theta_\Delta\circ U\colon H_{K,\Delta}\to\GL_d(C_\Delta)$ is a continuous cocycle (by definition of the canonical topology on $\B_{\dR,\Delta}^+$): by corollary \ref{coroSenH}, there exists a normal open subgroup $H$ of $H_{K,\Delta}$ (which we may and will assume of the form $H_{K^\prime,\Delta}$ for some finite Galois extension $K^\prime$ of $K$ in $\Kbar$) such that the restriction of $\theta_\Delta\circ U$ to $H$ has a trivial cohomology class. This implies that there exists $B_0\in\GL_d(\B_{\dR,\Delta}^+)$ such that the cocycle $U_0\colon g\mapsto B_0^{-1}U_gg(B_0)$ is such that $\theta_\Delta(U_{0,g})=\I_d$ for all $g\in H$.

\noindent
Let $M\in\NN$ and assume sequences $(B_m)_{0\leq m<M}$ and $(U_m)_{0\leq m<M}$ have been constructed such that:
\begin{itemize}
\item[(i)] $B_m\in\I_d+\Mat_d(\Fil^m\B_{\dR,\Delta}^+)$ and $U_m\colon H_{K,\Delta}\to\GL_d(\B_{\dR,\Delta}^+)$ is a continuous cocycle;
\item[(ii)] $U_{m,g}=B_m^{-1}U_{m-1,g}g(B_m)$ for all $g\in H_{K,\Delta}$ and $1\leq m<M$;
\item[(iii)] $U_{m,g}\in\Id_d+\Mat_d(\Fil^{m+1}\B_{\dR,\Delta}^+)$ for all $g\in H$ and $0\leq m<M$.
\end{itemize}
Denote by $V_{M,g}$ the image of $U_{M-1,g}$ in
$$\Mat_d(\Fil^M\B_{\dR,\Delta}^+)/\Mat_d(\Fil^{M+1}\B_{\dR,\Delta}^+)\simeq\Mat_d(\gr^M\B_{\dR,\Delta}^+).$$
If $g,h\in H$, we have
$$U_{M-1,gh}-\I_d=U_{M-1,g}g(U_{M-1,h})-\I_d=(U_{M-1,g}-\I_d)g(U_{M-1,h})+g(U_{M-1,h}-\I_d)$$
so that
$$U_{M-1,gh}-\I_d\equiv U_{M-1,g}-\I_d+g(U_{M-1,h}-\I_d)\mod\Mat_d(\Fil^{M+1}\B_{\dR,\Delta}^+)$$
since $g(U_{M-1,h})\equiv\I_d\mod\Mat_d(\Fil^M\B_{\dR,\Delta}^+)$. Reducing modulo $\Mat_d(\Fil^{M+1}\B_{\dR,\Delta}^+)$ gives
$$V_{M,gh}=V_{M,g}+g(V_{M,h}),$$
so that $g\mapsto V_{M,g}$ is a continuous cocycle $H=H_{K^\prime,\Delta}\to\Mat_d(\gr^M\B_{\dR,\Delta}^+)$. In particular, the entries of $V_M$ are cocycles $H\to\gr^M\B_{\dR,\Delta}^+$. By corollary \ref{coroinvariantsBdR2}, these are trivial: there exists $B_M\in\I_d+\Mat_d(\Fil^M\B_{\dR,\Delta}^+)$ such that if we put $U_{M,g}=B_M^{-1}U_{M-1,g}g(B_M)$ for all $g\in H_{K,\Delta}$, then $U_{M,g}\in\Id_d+\Mat_d(\Fil^{M+1}\B_{\dR,\Delta}^+)$ for all $g\in H$.

\noindent
By induction, we thus construct an infinite sequence $(B_m)_{m\in\NN}$. Property (i) ensures that the infinite product $B=B_0B_1B_2\cdots$ converges in $\GL_d(\B_{\dR,\Delta}^+)$, and condition (ii) and (iii) imply that $B^{-1}U_gg(B)=\I_d$ for all $g\in H$. This shows that the image of $U$ in $\H^1(H,\GL_d(\B_{\dR,\Delta}^+))$ is trivial. The inflation-restriction exact sequence
$$\{1\}\to\H^1(H_{K,\Delta}/H,\GL_d(\B_{\dR,\Delta}^{+H}))\to\H^1(H_{K,\Delta},\GL_d(\B_{\dR,\Delta}^+))\to\H^1(H,\GL_d(\B_{\dR,\Delta}^+))$$
thus implies that the class of $U$ in $\H^1(H_{K,\Delta},\GL_d(\B_{\dR,\Delta}^+))$ lies in the image of $\H^1(H_{K,\Delta}/H,\GL_d(\B_{\dR,\Delta}^{+H}))$, showing the surjectivity of the first map.

\noindent
The proof of the bijectivity of the second map is identical, replacing $H_{K,\Delta}$ by $G_{K,\Delta}$. Note that in that case, we can take $H=H_{K,\Delta}$ by corollary \ref{coroSen1}.
\end{proof}

\begin{coro}\label{coroSenBdRH1}
If $W$ is a free $\B_{\dR,\Delta}^+$-representation of rank $d$ of $H_{K,\Delta}$, then $W^{H_{K,\Delta}}$ is a projective $\L_{\dR,\Delta}^+$-module of rank $d$, and the natural map
$$\B_{\dR,\Delta}^+\otimes_{\L_{\dR,\Delta}^+}W^{H_{K,\Delta}}\to W$$
is an isomorphism of $\B_{\dR,\Delta}^+$-representations of $H_{K,\Delta}$.
\end{coro}

\begin{proof}
By proposition \ref{propSenBdRH}, there exists a finite Galois extension $K^\prime$ of $K$ in $\Kbar$ and a basis $\mathfrak{B}$ of $W$ over $\B_{\dR,\Delta}^+$ fixed under the action of $H_{K^\prime,\Delta}$. The $\L_{\dR,\Delta}^{\prime+}$-module $W^\prime$ generated by $\mathfrak{B}$ coincides with $W^{H_{K^\prime,\Delta}}$(it is thus stable under the action of $H_{K,\Delta}/H_{K^\prime,\Delta}\simeq\Gal(K^\prime_\infty/K_\infty)^\Delta$) and is free of rank $d$. Moreover, $\B_{\dR,\Delta}^+\otimes_{\L_{\dR,\Delta}^{\prime+}}W^\prime\to W$ is an isomorphism of $\B_{\dR,\Delta}^+$-representations of $H_{K,\Delta}$. As $\L_{\dR,\Delta}^+\to\L_{\dR,\Delta}^{\prime+}$ is finite Galois \'etale with Galois group $H_{K,\Delta}/H_{K^\prime,\Delta}\simeq\Gal(K^\prime_\infty/K_\infty)^\Delta$ (\cf corolary \ref{coroextfinet}), we have $\L_{\dR,\Delta}^{\prime+}\otimes_{\L_{\dR,\Delta}^+}W^{\prime H_{K,\Delta}/H_{K^\prime,\Delta}}\isomto W^\prime$ by Galois descent, and $W^{\prime H_{K,\Delta}/H_{K^\prime,\Delta}}=W^{H_{K,\Delta}}$ is projective of rank $d$ over $\L_{\dR,\Delta}^+$. Extending the scalars to $\B_{\dR,\Delta}^+$ provides the isomorphism $\B_{\dR,\Delta}^+\otimes_{\L_{\dR,\Delta}^+}W^{H_{K,\Delta}}\to W$ of $\B_{\dR,\Delta}^+$-representations of $H_{K,\Delta}$.
\end{proof}

\begin{coro}\label{coroSenBdRH2}
(\cf \cite{Font04}) Let $W_1,W_2$ be free $\B_{\dR,\Delta}^+$-representations of $H_{K,\Delta}$. Then
$$\Hom_{\Rep_{\B_{\dR,\Delta}^+}(H_{K,\Delta})}(W_1,W_2)\simeq\Hom_{\L_{\dR,\Delta}^+}\big(W_1^{H_{K,\Delta}},W_2^{H_{K,\Delta}}\big)\text{ and }\Ext^1_{\Rep_{\B_{\dR,\Delta}^+}(H_{K,\Delta})}(W_1,W_2)=\{0\}.$$
\end{coro}

\begin{proof}
We have $\B_{\dR,\Delta}^+\otimes_{\L_{\dR,\Delta}^+}W_1^{H_{K,\Delta}}\simeq W_1$ and $\B_{\dR,\Delta}^+\otimes_{\L_{\dR,\Delta}^+}W_2^{H_{K,\Delta}}\simeq W_2$, thus
\begin{align*}
\Hom_{\B_{\dR,\Delta}^+}(W_1,W_2) &\simeq \Hom_{\B_{\dR,\Delta}^+}\big(\B_{\dR,\Delta}^+\otimes_{\L_{\dR,\Delta}^+}W_1^{H_{K,\Delta}},\B_{\dR,\Delta}^+\otimes_{\L_{\dR,\Delta}^+}W_2^{H_{K,\Delta}}\big)\\
&\simeq \B_{\dR,\Delta}^+\otimes_{\L_{\dR,\Delta}^+}\Hom_{\L_{\dR,\Delta}^+}\big(W_1^{H_{K,\Delta}},W_2^{H_{K,\Delta}}\big)
\end{align*}
as $\B_{\dR,\Delta}^+$-representations of $H_{K,\Delta}$. Taking invariants under $H_{K,\Delta}$ provides an isomorphism
$$\Hom_{\Rep_{\B_{\dR,\Delta}^+}(H_{K,\Delta})}(W_1,W_2)\simeq\Hom_{\L_{\dR,\Delta}^+}\big(W_1^{H_{K,\Delta}},W_2^{H_{K,\Delta}}\big)$$
(recall that $\Hom_{\L_{\dR,\Delta}^+}\big(W_1^{H_{K,\Delta}},W_2^{H_{K,\Delta}}\big)$ is projective over $\L_{\dR,\Delta}^+$ since $W_1^{H_{K,\Delta}}$ and $W_2^{H_{K,\Delta}}$ are). Moreover, we have
\begin{align*}
\Ext^1_{\Rep_{\B_{\dR,\Delta}^+}(H_{K,\Delta})}(W_1,W_2) &\simeq \Ext^1_{\Rep_{\B_{\dR,\Delta}^+}(H_{K,\Delta})}\big(\B_{\dR,\Delta}^+,\Hom_{\B_{\dR,\Delta}^+}(W_1,W_2)\big)\\
&\simeq \H^1\big(H_{K,\Delta},\B_{\dR,\Delta}^+\otimes_{\L_{\dR,\Delta}^+}\Hom_{\L_{\dR,\Delta}^+}\big(W_1^{H_{K,\Delta}},W_2^{H_{K,\Delta}}\big)\big)\\
&\simeq \H^1\big(H_{K,\Delta},\B_{\dR,\Delta}^+\big)\otimes_{\L_{\dR,\Delta}^+}\Hom_{\L_{\dR,\Delta}^+}\big(W_1^{H_{K,\Delta}},W_2^{H_{K,\Delta}}\big)\\
&= \{0\}
\end{align*}
since $\Hom_{\L_{\dR,\Delta}^+}\big(W_1^{H_{K,\Delta}},W_2^{H_{K,\Delta}}\big)$ is a projective $\L_{\dR,\Delta}^+$-module and $\H^1\big(H_{K,\Delta},\B_{\dR,\Delta}^+\big)=\{0\}$ by corollary \ref{coroinvariantsBdR2}.
\end{proof}

\begin{defi}
We say a $\L_{\dR,\Delta}^+$-module $X$ is \emph{potentially free} if there exists a finite extension $K^\prime$ of $K$ in $\Kbar$ such that $\L_{\dR,\Delta}^{\prime+}\otimes_{\L_{\dR,\Delta}^+}X$ is a free $\L_{\dR,\Delta}^{\prime+}$-module of finite rank. We denote by $\Mod^{\pfree}(\L_{\dR,\Delta}^+)$ the corresponding category.
\end{defi}

\begin{coro}\label{coroequivBdRDH}
The functors
\begin{center}
\begin{tabular}{ccc}
\begin{minipage}{.35\linewidth}
{\begin{align*}
\Rep_{\B_{\dR,\Delta}^+}^{\free}(H_{K,\Delta}) &\to \Mod^{\pfree}(\L_{\dR,\Delta}^+)\\
W &\mapsto W^{H_{K,\Delta}}
\end{align*}}%
\end{minipage}
&
\text{ and }
&
\begin{minipage}{.35\linewidth}
{\begin{align*}
\Rep_{\B_{\dR,\Delta}^+}^{\free}(G_{K,\Delta}) &\to \Rep_{\L_{\dR,\Delta}^+}^{\free}(\Gamma_{K,\Delta})\\
W &\mapsto W^{H_{K,\Delta}}
\end{align*}}%
\end{minipage}
\end{tabular}
\end{center}
are equivalences of categories.
\end{coro}


\subsection{Decompletion along \texorpdfstring{$K_{\Delta,\infty}/K_\Delta$}{K\unichar{"005F}\unichar{"007B}\unichar{"0394},\unichar{"221E}\unichar{"007D}/K\unichar{"005F}\unichar{"0394}}}

\begin{defi}
Let $X$ be a $\L_{\dR,\Delta}^+$-representation of $\Gamma_{K,\Delta}$.
\begin{itemize}
\item[(i)] If $X$ is killed by $\Fil^{r+1}\L_{\dR,\Delta}^+$ for some $r\in\NN$, we denote by $X_{\free}$ the union of its sub-$K_\Delta$-modules of finite type that are stable under the action of $\Gamma_{K,\Delta}$ (this is the set of elements in $X$ whose orbit under $\Gamma_{K,\Delta}$ generates a $K_\Delta$-module of finite type). Note that this matches with definition of \S \ref{SenC} in the case $r=0$.

\item[(ii)] In the general case, we put $X_{\free}=\varprojlim\limits_r\big(X/\Fil^{r+1}\L_{\dR,\Delta}^+X\big)_{\free}$.
\end{itemize}
Note that $X_{\free}$ is a sub-$\l_{\dR,\Delta}^+$-module of $X$ and is stable by the action of $\Gamma_{K,\Delta}$.
\end{defi}

\begin{prop}\label{propdecompletionLdRr}
Let $r\in\NN$ and $X$ a free $\L_{\dR,\Delta}^+/\Fil^{r+1}\L_{\dR,\Delta}^+$-representation of rank $d$ of $\Gamma_{K,\Delta}$. Then $X_{\free}$ is free of rank $d$ over $\l_{\dR,\Delta}^+/\Fil^{r+1}\l_{\dR,\Delta}^+$, and the natural map
$$\L_{\dR,\Delta}^+\otimes_{\l_{\dR,\Delta}^+}X_{\free}\to X$$
is an isomorphism of $\L_{\dR,\Delta}^+$-representations of $\Gamma_{K,\Delta}^+$.
\end{prop}

\begin{proof}
The proof is similar to that of \cite[Proposition 3.17]{AB2}. We use induction on $r$, the case $r=0$ being corollary \ref{corodecomplf}: assume $r>0$. Put $X^\prime=\Fil^r\L_{\dR,\Delta}^+ X$: this is a free $L_\Delta$-representation of rank $d\binom{r+\delta-1}{\delta-1}$ of $\Gamma_{K,\Delta}$; and $X^{\prime\prime}=X/X^\prime\simeq(\L_{\dR,\Delta}^+/\Fil^r\L_{\dR,\Delta}^+)\otimes_{\L_{\dR,\Delta}^+}X$: this is a free $\L_{\dR,\Delta}^+/\Fil^r\L_{\dR,\Delta}^+$-representation of rank $d$ of $\Gamma_{K,\Delta}$. By the induction hypothesis and corollary \ref{corodecomplf}, the natural maps
\begin{align*}
L_\Delta\otimes_{K_{\infty,\Delta}}X^\prime_{\free} &\to X^\prime\\
\L_{\dR,\Delta}^+\otimes_{\l_{\dR,\Delta}^+}X^{\prime\prime}_{\free} &\to X^{\prime\prime}
\end{align*} 
are isomorphisms. In particular, we can find a basis $\mathfrak{B}^{\prime\prime}$ of $X^{\prime\prime}$ over $\L_{\dR,\Delta}^+/\Fil^r\L_{\dR,\Delta}^+$ such that the cocycle $U^{\prime\prime}$ giving the action of $\Gamma_{K,\Delta}$ on $X^{\prime\prime}$ in the basis $\mathfrak{B}^{\prime\prime}$ has values in $\GL_d(\l_{\dR,\Delta}^+/\Fil^r\l_{\dR,\Delta}^+)$. Let $\mathfrak{B}$ be a basis of $X$ over $\L_{\dR,\Delta}^+/\Fil^{r+1}\L_{\dR,\Delta}^+$ lifting $\mathfrak{B}^{\prime\prime}$: by construction, the cocycle $U$ giving the action of $\Gamma_{K,\Delta}$ on $X$ in the basis $\mathfrak{B}$ has values in $\GL_d\big(\pr_r^{-1}(\l_{\dR,\Delta}^+/\Fil^r\l_{\dR,\Delta}^+)\big)$, where $\pr_r\colon\L_{\dR,\Delta}^+/\Fil^{r+1}\L_{\dR,\Delta}^+\to\L_{\dR,\Delta}^+/\Fil^r\L_{\dR,\Delta}^+$ is the canonical surjection. By lemma \ref{lemmgrL}, we have
\begin{align*}
\pr_r^{-1}(\l_{\dR,\Delta}^+/\Fil^r\l_{\dR,\Delta}^+) &= \l_{\dR,\Delta}^+/\Fil^{r+1}\l_{\dR,\Delta}^++\gr^r(\L_{\dR,\Delta}^+)\\
&= K_{\infty,\Delta,<r}[T_\alpha]_{\alpha\in\Delta}\oplus\gr^r(\L_{\dR,\Delta}^+)
\end{align*}
where $K_{\infty,\Delta,<r}[T_\alpha]_{\alpha\in\Delta}$ is the sub-$K_{\infty,\Delta}$-module of elements in $K_{\infty,\Delta}[T_\alpha]_{\alpha\in\Delta}$ of total degree $<r$. For all $g\in\Gamma_{K,\Delta}$, we thus can write uniquely $U_g=\widehat{U}^{\prime\prime}_g+\widetilde{U}_g$ where $\widehat{U}_g\in\GL_d(K_{\infty,\Delta,<r}[T_\alpha]_{\alpha\in\Delta})$ and $\widetilde{U}_g\in\Mat_d(\gr^r\L_{\dR,\Delta}^+)$. Note that $\widehat{U}^{\prime\prime}$ is a lift of $U^{\prime\prime}$, but it is not a cocycle. As $\Gamma_{K,\Delta}$ is finitely generated, we can find $n\geq n_K$ (\cf proposition \ref{propRn}) such that $\widehat{U}^{\prime\prime}_g\in\GL_d(K_{n,\Delta,<r}[T_\alpha]_{\alpha\in\Delta})$ for all $g\in\Gamma_{K,\Delta}$.

Recall that if $\alpha\in\Delta$, there is a $L_{\Delta,\Delta\setminus\{\alpha\},n}$-linear and $\Gamma_{K,\Delta}$-equivariant projector $R_{n,\alpha}\colon L_\Delta\to L_{\Delta,\Delta\setminus\{\alpha\},n}$ (\cf notations before theorem \ref{theocohoGC}). They commute to each other: their composite provides a $K_{n,\Delta}$-linear projector $R_{n,\Delta}\colon L_\Delta\to K_{n,\Delta}$. We extend $R_{n,\Delta}$ to $\gr^r(\L_{\dR,\Delta}^+)$ by putting $R_{n,\Delta}(t_\alpha)=t_\alpha$ for all $\alpha\in\Delta$ (\cf lemma \ref{lemmgrL}). As $U$ is a cocycle, we have $U_{g\gamma}=U_gg(U_\gamma)$ \ie
\begin{align*}
\widehat{U}^{\prime\prime}_{g\gamma}+\widetilde{U}_{g\gamma} &= \big(\widehat{U}^{\prime\prime}_g+\widetilde{U}_g\big)g\big(\widehat{U}^{\prime\prime}_\gamma+\widetilde{U}_\gamma\big)\\
&= \widehat{U}^{\prime\prime}_gg\big(\widehat{U}^{\prime\prime}_\gamma\big)+\widehat{U}^{\prime\prime}_gg\big(\widetilde{U}_\gamma\big)+\widetilde{U}_gg\big(\widehat{U}^{\prime\prime}_\gamma\big)
\end{align*}
for all $g,\gamma\in\Gamma_{K,\Delta}$ (since $\widetilde{U}_gg\big(\widetilde{U}_\gamma\big)=0$ in $\Mat_d(\L_{\dR,\Delta}^+/\Fil^{r+1}\L_{\dR,\Delta}^+)$ because $\Fil^{2r}\L_{\dR,\Delta}^+\subset\Fil^{r+1}\L_{\dR,\Delta}^+$). Applying $R_{n,\Delta}$, we get
\begin{align*}
\widehat{U}^{\prime\prime}_{g\gamma}+R_{n,\Delta}\big(\widetilde{U}_{g\gamma}\big) &= \widehat{U}^{\prime\prime}_gg\big(\widehat{U}^{\prime\prime}_\gamma\big)+\widehat{U}^{\prime\prime}_gg\big(R_{n,\Delta}\big(\widetilde{U}_\gamma\big)\big)+R_{n,\Delta}\big(\widetilde{U}_g\big)g\big(\widehat{U}^{\prime\prime}_\gamma\big)\\
&= \big(\widehat{U}^{\prime\prime}_g+R_{n,\Delta}\big(\widetilde{U}_g\big)\big)g\big(\widehat{U}^{\prime\prime}_\gamma+R_{n,\Delta}\big(\widetilde{U}_\gamma\big)\big)
\end{align*}
since $R_{n,\Delta}$ is $K_{n,\Delta}$-linear and $\Gamma_{K,\Delta}$-equivariant. If we put $U_g^{(n)}=\widehat{U}^{\prime\prime}_g+R_{n,\Delta}\big(\widetilde{U}_g\big)$ for all $g\in\Gamma_{K,\Delta}$, this means that $U^{(n)}\colon\Gamma_{K,\Delta}\to\GL_d\big(\l_{\dR,\Delta}^+/\Fil^{r+1}\l_{\dR,\Delta}^+\big)$ is a cocycle. For $g\in\Gamma_{K,\Delta}$, put $\widetilde{U}^{(n)}_g=U_g-U^{(n)}_g\in\Mat_d\big(\gr^r\L_{\dR,\Delta}^+\big)$. By the computation above (with $U^{(n)}$ and $\widetilde{U}^{(n)}$ instead of $\widehat{U}^{\prime\prime}$ and $\widetilde{U}$), we have
$$U^{(n)}_{g\gamma}+\widetilde{U}^{(n)}_{g\gamma}=U^{(n)}_gg\big(U^{(n)}_\gamma\big)+U^{(n)}_gg\big(\widetilde{U}^{(n)}_\gamma\big)+\widetilde{U}^{(n)}_gg\big(U^{(n)}_\gamma\big)$$
so that
$$\widetilde{U}^{(n)}_{g\gamma}=U^{(n)}_gg\big(\widetilde{U}^{(n)}_\gamma\big)+\widetilde{U}^{(n)}_gg\big(U^{(n)}_\gamma\big)=U_gg\big(\widetilde{U}^{(n)}_\gamma\big)+\widetilde{U}^{(n)}_gg\big(U_\gamma\big)$$
since $U^{(n)}$ is a cocycle. This means that $\widetilde{U}^{(n)}$ is a cocycle that defines an extension of $L_\Delta\otimes_{K_{\infty,\Delta}}X^{\prime\prime}_{\free}$ by $L_\Delta\otimes_{K_{\infty,\Delta}}X^\prime_{\free}$ as $L_\Delta$-representations of $\Gamma_{K,\Delta}$. By corollary \ref{corocomplf2}, this extension comes from an extension of $X^{\prime\prime}_{\free}$ by $X^\prime_{\free}$ as $K_{\infty,\Delta}$-representations of $\Gamma_{K,\Delta}$, by extension of scalars: there exists $N\in\Mat_d(\gr^r\L_{\dR,\Delta}^+)$ such that
$$\widetilde{U}^{(n)}_g+U_gg(N)-NU_g\in\Mat_d(\gr^r\l_{\dR,\Delta}^+)$$
for all $g\in\Gamma_{K,\Delta}$. Now put $B=\I_d+N\in\GL_d(\L_{\dR,\Delta}^+)$: for all $g\in\Gamma_{K,\Delta}$ we have
\begin{align*}
B^{-1}U_gg(B) &= (\I_d-N)\big(U^{(n)}_g+\widetilde{U}^{(n)}_g\big)(\I_d+g(N))\\
&= U^{(n)}_g+\widetilde{U}^{(n)}_g-NU^{(n)}_g+U^{(n)}_gg(N)\in\GL_d(\l_{\dR,\Delta}^+/\Fil^{r+1}\l_{\dR,\Delta}^+).
\end{align*}
This means that if we make a change of basis from $\mathfrak{B}$ using $B$, we reduce to the case where the cocycle $U$ takes values in $\GL_d(\l_{\dR,\Delta}^+/\Fil^{r+1}\l_{\dR,\Delta}^+)$. Proposition \ref{propfond} then shows that $X_{\free}$ is nothing but the $\l_{\dR,\Delta}^+/\Fil^{r+1}\l_{\dR,\Delta}^+$-span of the basis just constructed: in particular, it is free, and the map
$$\L_{\dR,\Delta}^+\otimes_{\l_{\dR,\Delta}^+}X_{\free}\to X$$
is an isomorphism of $\L_{\dR,\Delta}^+$-representations of $\Gamma_{K,\Delta}^+$.
\end{proof}

\begin{prop}\label{propdecompletionLdR}
Let $X$ le a free $\L_{\dR,\Delta}^+$-representation of rank $d$ of $\Gamma_{K,\Delta}$. Then $X_{\free}$ is free of rank $d$ over $\l_{\dR,\Delta}^+$, and the natural map
$$\L_{\dR,\Delta}^+\otimes_{\l_{\dR,\Delta}^+}X_{\free}\to X$$
is an isomorphism of $\L_{\dR,\Delta}^+$-representations of $\Gamma_{K,\Delta}^+$. Moreover, $X_{\free}$ is the union of the sub-$\l_{\dR,\Delta}^+$-modules of $X$ that are of finite type and stable under the action of $\Gamma_{K,\Delta}$.
\end{prop}

\begin{proof}
For $r\in\NN$, put $X_r=(\L_{\dR,\Delta}^+/\Fil^{r+1}\L_{\dR,\Delta}^+)\otimes_{\L_{\dR,\Delta}^+}X$. This is a free $\L_{\dR,\Delta}^+/\Fil^{r+1}\L_{\dR,\Delta}^+$-representation of rank $d$ of $\Gamma_{K,\Delta}$. By proposition \ref{propdecompletionLdRr}, $X_{r,\free}$ is a free $\l_{\dR,\Delta}^+/\Fil^{r+1}\l_{\dR,\Delta}^+$-representation of rank $d$ of $\Gamma_{K,\Delta}$, and the natural map $\L_{\dR,\Delta}^+\otimes_{\l_{\dR,\Delta}^+}X_{r,\free}\to X_r$ is an isomorphism of $\L_{\dR,\Delta}^+/\Fil^{r+1}\L_{\dR,\Delta}^+$-representations of $\Gamma_{K,\Delta}$. The maps $X_{r+1}\to X_r$ are surjective: so are the maps $X_{r+1,\free}\to X_{r,\free}$ by faithful flatness of $\L_{\dR,\Delta}^+/\Fil^{r+1}\L_{\dR,\Delta}^+$ over $\l_{\dR,\Delta}^+/\Fil^{r+1}\l_{\dR,\Delta}^+$ (\cf lemma \ref{lemmgrL}). This implies that any basis of $X_{r,\free}$ can be lifted to a basis of $X_{r+1,\free}$: by induction, we can construct a sequence $(\mathfrak{B}_r)_{r\in\NN}$ such that $\mathfrak{B}_r$ is a basis of $X_{r,\free}$ over $\l_{\dR,\Delta}^+/\Fil^{r+1}\l_{\dR,\Delta}^+$ and is the image of $\mathfrak{B}_{r+1}$ in $(\l_{\dR,\Delta}^+/\Fil^{r+1}\l_{\dR,\Delta}^+)\otimes_{\l_{\dR,\Delta}^+}X_{r+1,\free}$ for all $r\in\NN$. This sequence defines a basis $\mathfrak{B}$ of $X_{\free}=\varprojlim\limits_rX_{r,\free}$ over $\l_{\dR,\Delta}^+$. By extension of scalars, $\mathfrak{B}$ is a basis of $X$ over $\L_{\dR,\Delta}^+$ as well, so that the natural map
$$\L_{\dR,\Delta}^+\otimes_{\l_{\dR,\Delta}^+}X_{\free}\to X$$
is an isomorphism of $\L_{\dR,\Delta}^+$-representations of $\Gamma_{K,\Delta}$.

Let $X_{\free}^\prime$ be the union of the sub-$\l_{\dR,\Delta}^+$-modules of $X$ that are of finite type and stable under the action of $\Gamma_{K,\Delta}$: we have $X_{\free}\subset X_{\free}^\prime$. The reverse inclusion is checked modulo $\Fil^{r+1}\l_{\dR,\Delta}^+$ for all $r\in\NN$. Let $Y$ be a finite type  sub-$\l_{\dR,\Delta}^+$-module of $X$ that is stable under the action of $\Gamma_{K,\Delta}$, and $Y_r=(\l_{\dR,\Delta}^+/\Fil^{r+1}\l_{\dR,\Delta}^+)\otimes_{\L_{\dR,\Delta}^+}Y\subset X_r$ (\cf lemma \ref{lemmgrL}). Let $y_1,\ldots,y_s$ be a generating family of $Y_r$ over $\l_{\dR,\Delta}^+$ and $g_1,\ldots,g_u$ a finite set of topological generators of $\Gamma_{K,\Delta}$. There exists $n\in\NN$ and coefficients $(c_{i,j}^{(a)})_{\substack{1\leq i,j\leq s\\ 1\leq a\leq u}}$ in $K_{n,\Delta}$ such that $g_a(y_i)=\sum\limits_{j=1}^sc_{i,j}^{(a)}y_j$. This implies that for all $m\geq n$, the sub-$K_{m,\Delta}$-module $Y_{r,m}$ of $Y_r$ generated by $y_1,\ldots,y_s$ is stable under the action of $\Gamma_{K,\Delta}$: we have $Y_{r,m}\subset X_{r,\free}$. As $Y_r=\bigcup\limits_{m\geq n}Y_{r,m}$, we thus have $Y_r\subset X_{r,\free}$.
\end{proof}

\begin{theo}\label{theoSenBdR}
The functor
\begin{align*}
\Rep_{\B_{\dR,\Delta}^+}^{\free}(G_{K,\Delta}) &\to \Rep_{\l_{\dR,\Delta}^+}^{\free}(\Gamma_{K,\Delta})\\
W &\mapsto \big(W^{H_{K,\Delta}}\big)_{\free}
\end{align*}
is an equivalence of categories.
\end{theo}

\begin{proof}
This follows from corollary \ref{coroequivBdRDH} and proposition \ref{propdecompletionLdR}.
\end{proof}

\begin{nota}
Put $\L_{\dR,\Delta}=\L_{\dR,\Delta}^+\big[\frac{1}{t_\Delta}\big]$ and $\l_{\dR,\Delta}=\l_{\dR,\Delta}^+\big[\frac{1}{t_\Delta}\big]=K_{\Delta,\infty}[\![t_\alpha]\!]\big[\frac{1}{t_\alpha}\big]_{\alpha\in\Delta}$. We have $\l_{\dR,\Delta}\subset\L_{\dR,\Delta}=\H^0(H_{K,\Delta},\B_{\dR,\Delta})$.
\end{nota}

\begin{defi}
\begin{itemize}
\item[(1)] A \emph{lattice} of a free $\B_{\dR,\Delta}$-module (resp. $\l_{\dR,\Delta}$-module) $M$ of finite rank, is a sub-$\B_{\dR,\Delta}^+$-module (resp. sub-$\l_{\dR,\Delta}^+$-module) generated by a basis of $M$.
\item[(2)] We denote by $\Rep_{\B_{\dR,\Delta}}^{\reg}(G_{K,\Delta})$ (resp. $\Rep_{\l_{\dR,\Delta}}^{\reg}(\Gamma_{K,\Delta})$) the isogeny category of $\Rep_{\B_{\dR,\Delta}^+}^{\free}(G_{K,\Delta})$ (resp. $\Rep_{\l_{\dR,\Delta}^+}^{\free}(\Gamma_{K,\Delta})$), \ie the category whose objects, which are called \emph{regular representations}, are free $\B_{\dR,\Delta}$-modules (resp. $\l_{\dR,\Delta}$-modules) of finite rank endowed with a semi-linear action of $G_{K,\Delta}$ (resp. $\Gamma_{K,\Delta}$), and admitting a $G_{K,\Delta}$-stable (resp. $\Gamma_{K,\Delta}$-stable) lattice which defines an object of $\Rep_{\B_{\dR,\Delta}^+}(G_{K,\Delta})$ (resp. $\Rep_{\l_{\dR,\Delta}^+}(\Gamma_{K,\Delta})$).
\end{itemize}
\end{defi}

\begin{theo}\label{theoSenBdR2}
The functor
$$\Rep_{\B_{\dR,\Delta}}^{\reg}(G_{K,\Delta})\to\Rep_{\l_{\dR,\Delta}}^{\reg}(\Gamma_{K,\Delta})$$
induced by that of theorem \ref{theoSenBdR} is an equivalence of categories.
\end{theo}


\subsection{The module with connection associated with a \texorpdfstring{$\l_{\dR,\Delta}^+$}{l\unichar{"005F}\unichar{"007B}dR,\unichar{"0394}\unichar{"007D}\unichar{"207A}}-representation}\label{sectSenconnection}

\begin{nota}
Let $\Omega^+=\Omega_{\l_{\dR,\Delta}^+/K_{\Delta,\infty}}^+$ (resp. $\Omega=\Omega_{\l_{\dR,\Delta}/K_{\Delta,\infty}}^+$) be the $\l_{\dR,\Delta}^+$-module of continuous $K_{\Delta,\infty}$-differentials of $\l_{\dR,\Delta}^+$ (resp. $\l_{\dR,\Delta}$) having poles of order $\leq1$ on the divisor $(t_\Delta=0)$. This is the free $\l_{\dR,\Delta}^+$-module (resp. $\l_{\dR,\Delta}$-module) with basis $\big(\frac{\dd t_\alpha}{t_\alpha}\big)_{\alpha\in\Delta}$.
\end{nota}

\begin{defi}
A \emph{module with connection} over $\l_{\dR,\Delta}^+$ is a free $\l_{\dR,\Delta}^+$-module $Y$ of finite rank equipped with a $K_{\Delta,\infty}$-linear map
$$\nabla_Y\colon Y\to Y\otimes_{\l_{\dR,\Delta}^+}\Omega^+$$
satisfying the Leibniz rule: $\nabla_Y(\lambda y)=y\otimes\dd y+\lambda\nabla_Y(y)$ for all $\lambda\in\l_{\dR,\Delta}^+$ and $y\in Y$. If $Y_1$ and $Y_2$ are two modules with connection over $\l_{\dR,\Delta}^+$, we endow $Y_1\otimes_{\l_{\dR,\Delta}^+}Y_2$ (resp. $\Hom_{\l_{\dR,\Delta}^+}(Y_1,Y_2)$) with the connection $\nabla_{Y_1}\otimes\Id_{Y_2}+\Id_{Y_1}\otimes\nabla_{Y_2}$ (resp. $f\mapsto\nabla_{Y_2}\circ f-(f\otimes\Id_{\Omega^+})\circ\nabla_{Y_1}$). A morphism between two modules with connection is an horizontal $\l_{\dR,\Delta}^+$-linear map. This defines a tensor category denoted $\mathscr{R}_{\l_{\dR,\Delta}^+}$.

Similarly, one defines the category $\mathscr{R}_{\l_{\dR,\Delta}}$ of free $\l_{\dR,\Delta}$-modules of finite rank with connection. If $(Y,\nabla_Y)\in\mathscr{R}_{\l_{\dR,\Delta}}$, the connection $\nabla_Y$ is said \emph{regular} if there exists a lattice $\mathcal{Y}$ in $Y$ such that $\nabla_Y(\mathcal{Y})\subset\mathcal{Y}\otimes_{\l_{\dR,\Delta}^+}\Omega^+$ so that $(\mathcal{Y},\nabla_{Y|\mathcal{Y}})\in\mathscr{R}_{\l_{\dR,\Delta}^+}$. We denote $\mathscr{R}_{\l_{\dR,\Delta}}^{\reg}$ the full subcategory of $\mathscr{R}_{\l_{\dR,\Delta}}$ made of modules with regular connection.
\end{defi}

\begin{nota}
Let $Y\in\Rep_{\l_{\dR,\Delta}^+}^{\free}(\Gamma_{K,\Delta})$. If $r\in\NN$, the quotient $Y_r:=\big(\l_{\dR,\Delta}^+/\Fil^{r+1}\l_{\dR,\Delta}^+\big)\otimes_{\l_{\dR,\Delta}^+}Y$ is a free $\l_{\dR,\Delta}^+/\Fil^{r+1}\l_{\dR,\Delta}^+$-module of finite rank endowed with a continuous semi-linear action of $\Gamma_{K,\Delta}$: in particular, it defines an object of $\Rep_{K_{\Delta,\infty}}^{\free}(\Gamma_{K,\Delta})$ (\cf lemma \ref{lemmgrL}). The infinitesimal action of $\Gamma_{K,\Delta}$ on $Y_r$ provides Sen operators $(\nabla_{Y_r,\alpha})_{\alpha\in\Delta}$: these are $K_{\Delta,\infty}$-linear endomorphisms of $Y_r$ characterized by the fact that for all $y\in Y_r$, there exists an open normal subgroup $\Gamma_{K,\Delta,y}\lhd\Gamma_{K,\Delta}$ such that for all $\alpha\in\Delta$ and $\gamma\in\Gamma_{K,\alpha}\cap\Gamma_{K,\Delta,y}$, we have $\gamma(y)=\exp\big(\log(\chi(\gamma))\nabla_{Y_r,\alpha}\big)y$ (\cf section \ref{sectSenoperators}).

These maps are compatible as $r$ grows: we thus obtain $K_{\Delta,\infty}$-linear endomorphisms $(\nabla_{Y,\alpha})_{\alpha\in\Delta}$ of $Y$ such that for all $y\in Y$ and all $r\in\NN$, there exists an open normal subgroup $\Gamma_{K,\Delta,y,r}\lhd\Gamma_{K,\Delta}$ such that for all $\alpha\in\Delta$ and $\gamma\in\Gamma_{K,\alpha}\cap\Gamma_{K,\Delta,y,r}$, we have
$$\gamma(y)\equiv\exp\big(\log(\chi(\gamma))\nabla_{Y_r,\alpha}\big)y\mod(\Fil^{r+1}\l_{\dR,\Delta}^+)Y.$$
\end{nota}

\begin{prop}\label{propderivY}
Let $Y\in\Rep_{\l_{\dR,\Delta}^+}(\Gamma_{K\Delta})$ and $\alpha\in\Delta$. Then $\nabla_{Y,\alpha}(\gamma(y))=\gamma(\nabla_{Y,\alpha}(y))$ for all $y\in Y$ and $\gamma\in\Gamma_{K,\Delta}$. Moreover, if $y\in Y$ and $\beta\in\Delta$, we have
$$\nabla_{Y,\alpha}(t_\beta y)=\begin{cases} t_\beta\nabla_{Y,\alpha}(y) &\text{ if }\beta\neq\alpha\\ t_\alpha y+t_\alpha\nabla_{Y,\Delta}(y) &\text{ if }\beta=\alpha\end{cases}.$$
\end{prop}

\begin{proof}
This is checked modulo $\Fil^{r+1}$ for all $r\in\NN$. The first statement follows from the corresponding property of Sen operators. For the second one, fix $y\in Y_r$: we have $\nabla_{Y_r,\alpha}(t_\beta y)=\lim\limits_{\substack{\gamma\in\Gamma_{K_\alpha}\\ \gamma\to\Id}}\frac{\gamma(t_\beta y)-t_\beta y}{\log(\chi(\gamma))}$. If $\gamma\in\Gamma_{K_\alpha}$ and $\beta\neq\alpha$, we have $\gamma(t_\beta y)=t_\beta\gamma(y)$, so that $\nabla_{Y_r,\alpha}(t_\beta y)=t_\beta\nabla_{Y_r,\alpha}(y)$: assume $\beta=\alpha$. We have $\gamma(t_\alpha y)=\chi(\gamma)t_\alpha\gamma(y)$, so that
$$\frac{\gamma(t_\alpha y)-t_\alpha y}{\log(\chi(\gamma)}=\frac{\chi(\gamma)t_\alpha\gamma(y)-t_\alpha y}{\log(\chi(\gamma)}=t_\alpha\Big(\frac{\chi(\gamma)-1}{\log(\chi(\gamma)}\gamma(y)+\frac{\gamma(y)-y}{\log(\chi(\gamma)}\Big)$$
which converges to $t_\alpha(y+\nabla_{Y_r,\alpha}(y))$ as $\gamma$ converges to $\Id$.
\end{proof}

\begin{defi}
Let $Y\in\Rep_{\l_{\dR,\Delta}^+}(\Gamma_{K\Delta})$. If $y\in Y$, we put
$$\nabla_Y(y)=\sum\limits_{\alpha\in\Delta}\nabla_{Y,\alpha}(y)\otimes\frac{\dd t_\alpha}{t_\alpha}\in Y\otimes_{\l_{\dR,\Delta}^+}\Omega^+.$$
\end{defi}

\begin{prop}
Let $Y\in\Rep_{\l_{\dR,\Delta}^+}(\Gamma_{K\Delta})$. The map $\nabla_Y$ is an integrable connection.
\end{prop}

\begin{proof}
The Leibniz rule follows from proposition \ref{propderivY}, and the integrability from the fact that $\Gamma_{K,\Delta}$ is abelian.
\end{proof}

We thus have a functor
$$\Rep_{\l_{\dR,\Delta}^+}^{\free}(\Gamma_{K,\Delta})\to\mathscr{R}_{\l_{\dR,\Delta}^+}$$
which induces functors
$$\Rep_{\l_{\dR,\Delta}}^{\reg}(\Gamma_{K,\Delta})\to\mathscr{R}_{\l_{\dR,\Delta}}^{\reg}\text{ and }\Rep_{\B_{\dR,\Delta}}^{\reg}(G_{K,\Delta})\to\mathscr{R}_{\l_{\dR,\Delta}}^{\reg}.$$

Let $Y\in\Rep_{\l_{\dR,\Delta}^+}^{\free}(\Gamma_{K,\Delta})$. If $y\in Y^{\Gamma_{K,\Delta}}$, we have $\nabla_{Y,\alpha}(y)=0$ for all $\alpha\in\Delta$, hence $\nabla_Y(y)=0$. By $K_{\Delta,\infty}$-linearity, the inclusion $Y^{\Gamma_{K,\Delta}}\subset Y^{\nabla_Y=0}$ induces a map
$$c_\nabla(Y)\colon K_{\Delta,\infty}\otimes_{K_\Delta}Y^{\Gamma_{K,\Delta}}\to Y^{\nabla_Y=0}.$$ 

\begin{prop}\label{propsecthorizY}
The $K_\Delta$-module $Y^{\Gamma_{K,\Delta}}$ is of finite type and the map $c_\nabla(Y)$ is an isomorphism.
\end{prop}

\begin{proof}
If $r\in\NN$, put $Y_r=(\l_{\dR,\Delta}^+/\Fil^{r+1}\l_{\dR,\Delta}^+)\otimes_{\l_{\dR,\Delta}^+}Y$: this is in particular a object in $\Rep_{K_{\Delta,\infty}}^{\free}(\Gamma_{K,\Delta})$. Put also
$$Y_r^{\nabla_Y=0}=\{y\in Y_r\,;\,(\forall\alpha\in\Delta)\,\nabla_{Y_r,\alpha}(y)=0\}$$
(this is a an abuse of notation since $\nabla_Y$ does not make sense on $Y_r$). By proposition \ref{propSenop2}, the natural map
$$K_{\Delta,\infty}\otimes_{K_\Delta}Y_r^{\Gamma_{K,\Delta}}\to Y_r^{\nabla_Y=0}$$
is an $K_{\Delta,\infty}$-linear isomorphism. We have  $Y^{\nabla_Y=0}=\varprojlim\limits_rY_r^{\nabla_Y=0}$, so in particular the inverse limit of the above isomorphisms is an isomorphism
$$\varprojlim\limits_r K_{\Delta,\infty}\otimes_{K_\Delta}Y_r^{\Gamma_{K,\Delta}}\isomto Y^{\nabla_Y=0}.$$
Similarly, we have $Y^{\Gamma_{K,\Delta}}=\varprojlim\limits_rY_r^{\Gamma_{K,\Delta}}$. The exact sequence
$$0\to\gr^rY\to Y_{r+1}\to Y_r\to0$$
induces the exact sequence
$$0\to\big(\gr^rY\big)^{\Gamma_{K,\Delta}}\to Y_{r+1}^{\Gamma_{K,\Delta}}\to Y_r^{\Gamma_{K,\Delta}}.$$
We have $\gr^rY\simeq\bigoplus\limits_{\substack{\underline{n}\in\NN^\Delta\\ \abs{\underline{n}}=r}}Y_1(\underline{n})$, so that $\big(\gr^rY\big)^{\Gamma_{K,\Delta}}\simeq\bigoplus\limits_{\substack{\underline{n}\in\NN^\Delta\\ \abs{\underline{n}}=r}}\big(Y_1(\underline{n})\big)^{\Gamma_{K,\Delta}}$. By proposition \ref{propSenop2} again, we know that $\big(Y_1(\underline{n})\big)^{\Gamma_{K,\Delta}}=\{0\}$ for all but finitely many values of $\underline{n}$. This implies that $\big(\gr^rY\big)^{\Gamma_{K,\Delta}}=\{0\}$, \ie that the natural map $Y_{r+1}^{\Gamma_{K,\Delta}}\to Y_r^{\Gamma_{K,\Delta}}$ is injective when $r\gg0$. As $Y_r^{\Gamma_{K,\Delta}}$ is a $K_\Delta$-module of finite type (hence a $F_0$-space of finite dimension) for all $r\in\NN$, this implies that $Y_{r+1}^{\Gamma_{K,\Delta}}\to Y_r^{\Gamma_{K,\Delta}}$ is in fact an isomorphism when $r\gg0$. In particular, the map $Y^{\Gamma_{K,\Delta}}\to Y_r^{\Gamma_{K,\Delta}}$ is an isomorphism for $r\gg0$ (this proves the first part of the proposition), and $\varprojlim\limits_r K_{\Delta,\infty}\otimes_{K_\Delta}Y_r^{\Gamma_{K,\Delta}}\simeq K_{\Delta,\infty}\otimes_{K_\Delta}Y^{\Gamma_{K,\Delta}}$.
\end{proof}


\subsection{Application to $p$-adic representations: link with multivariate $(\varphi,\Gamma)$-modules}

Let $V\in\Rep_{\QQ_p}(G_{K,\Delta})$. Then $\B_{\dR,\Delta}\otimes_{\QQ_p}V\in\Rep_{\B_{\dR,\Delta}}^{\reg}(G_{K,\Delta})$ (a $G_{K,\Delta}$-stable lattice being given by $\B_{\dR,\Delta}^+\otimes_{\QQ_p}V\in\Rep_{\B_{\dR,\Delta}}^{\reg}(G_{K,\Delta})$). Put
$$\D_{\dif}^+(V)=\big((\B_{\dR,\Delta}^+\otimes_{\QQ_p}V)^{H_{K,\Delta}}\big)_{\free}$$
and $\D_{\dif}(V)=\l_{\dR,\Delta}\otimes_{\l_{\dR,\Delta}^+}\D_{\dif}^+(V)$. By what precedes, this provides objects in $\mathscr{R}_{\l_{\dR,\Delta}^+}$ and $\mathscr{R}_{\l_{\dR,\Delta}}^{\reg}$ respectively.

We have $\D_{\dif}(V)=\D_{\dif}^+(V)\big[\frac{1}{t_\Delta}\big]$, $\D_{\dif}^+(V)^{\Gamma_{K,\Delta}}=(\B_{\dR,\Delta}^+\otimes_{\QQ_p}V)^{G_{K,\Delta}}$ and $\D_{\dif}(V)^{\Gamma_{K,\Delta}}=\D_{\dR}(V)$.

\begin{prop}\label{propdifdR}
(\cf \cite[Proposition 7.1]{AB2}) Let $V\in\Rep_{\QQ_p}(G_{K,\Delta})$. Then $V$ is de Rham if and only if $\D_{\dif}(V)$ is trivial (as a module with connection).
\end{prop}

\begin{proof}
Assume $V$ is de Rham: there exists $\underline{n}\in\ZZ^\Delta$ such that $V(\underline{n})$ has Hodge-Tate weight whose components are all non-positive, so that $\D_{\dR}(V(\underline{n}))=\D_{\dR}^+(V(\underline{n}))=:(\B_{\dR,\Delta}^+\otimes_{\QQ_p}V(\underline{n}))^{G_{K,\Delta}}$. The map $\alpha_{\dR}(V)$ is an isomorphism, hence
$$\alpha_{\dR}^+\colon\B_{\dR,\Delta}^+\otimes_{K_\Delta}\D_{\dR}^+(V(\underline{n}))\to\B_{\dR,\Delta}^+\otimes_{\QQ_p}V(\underline{n})$$
is injective with cokernel killed by $t_\Delta$. Taking invariants under $H_{K,\Delta}$, we get an injective map
$$g\colon\L_{\dR,\Delta}^+\otimes_{K_\Delta}\D_{\dR}^+(V(\underline{n}))\to(\B_{\dR,\Delta}^+\otimes_{\QQ_p}V(\underline{n}))^{H_{K,\Delta}}$$
whose cokernel is killed by some power of $t_\Delta$. The commutative diagram
$$\xymatrix{
\l_{\dR,\Delta}^+\otimes_{K_\Delta}\D_{\dR}^+(V(\underline{n}))\ar[r]^-f\ar@{^(->}[d] & \D_{\dif}^+(V(\underline{n}))\ar@{^(->}[d]\\
\L_{\dR,\Delta}^+\otimes_{K_\Delta}\D_{\dR}^+(V(\underline{n}))\ar[r]^-g & (\B_{\dR,\Delta}^+\otimes_{\QQ_p}V(\underline{n}))^{H_{K,\Delta}}}$$
shows that $f$ is injective. If $y\in\D_{\dif}^+(V(\underline{n}))$, there exists $N\in\NN$ such that $t_\Delta^Ny\in\im(g)\cap\D_{\dif}^+(V(\underline{n}))$. By definition of the functor $X\mapsto X_{\free}$, this shows that $t_\Delta^Ny\in\l_{\dR,\Delta}^+\otimes_{K_\Delta}\D_{\textcolor{red}{\dR}}^+(V(\underline{n}))$. The cokernel of $f$ is thus of $t_\Delta$-torsion, so that $f$ induces an isomorphism
$$\l_{\dR,\Delta}\otimes_{K_\Delta}\D_{\dR}^+(V(\underline{n}))\isomto\D_{\dif}(V(\underline{n}))=\D_{\dif}(V)$$
which implies that $\D_{\dif}(V)$ is trivial (as a module with connection).

Conversely, assume that $Y:=\D_{\dif}(V)$ is trivial as a module with connection over $\l_{\dR,\Delta}$: the natural map
$$\l_{\dR,\Delta}\otimes_{K_{\Delta,\infty}}Y^{\nabla_Y=0}\to Y$$
is an isomorphism. As $Y^{\nabla_Y=0}$ is a $K_{\Delta,\infty}$-module of finite type, there exists $n\in\NN$ such that $Y^{\nabla_Y=0}\subset t_\Delta^{-n}\D_{\dif}^+(V)$. Replacing $V$ with $V(\underline{n})$ for an appropriate $\underline{n}\in\ZZ^\Delta$, we may asume that $n=0$, so that $Y^{\nabla_Y=0}=\D_{\dif}^+(V)^{\nabla_Y=0}=K_{\Delta,\infty}\otimes_{K_\Delta}\D_{\dR}^+(V)$ (the last equality by proposition \ref{propsecthorizY}). Extending the scalars from $K_{\Delta,\infty}$ to $\B_{\dR,\Delta}$, we deduce that
$$\alpha_{\dR}(V)\colon\B_{\dR,\Delta}\otimes_{K_\Delta}\D_{\dR}^+(V)\to\B_{\dR,\Delta}\otimes_{\QQ_p}V$$
is an isomorphism, \ie that $V$ is de Rham.
\end{proof}

Now we relate our constructions to multivariate overconvergent $(\varphi,\Gamma)$-modules. These were constructed in \cite{PZ} and \cite{CKZ}, under the hypothesis that $K$ is a finite extension of $\QQ_p$, what we thus assume henceforth. We start by recalling the definitions and results we will need from \cite{PZ} and \cite{CKZ}. In the classical, univariate case, put
$$\AA_{F_0}=\Big\{\sum\limits_{n\in\ZZ}a_n\varpi^n\,;\,(\forall n\in\ZZ)\,a_n\in\W(k)\,\lim\limits_{n\to-\infty}a_n=0\Big\}$$
where $\varpi$ is seen as a variable. This is a Cohen ring for the field $\EE_{F_0}=k(\!(\overline{\varpi})\!)$. We equip $\AA_{F_0}$ with the commuting (semi-linear) Frobenius operator and $\Gamma_{F_0}$-action given by
$$\varphi(\varpi)=(1+\varpi)^p-1\text{ and }\gamma(\varpi)=(1+\varpi)^{\chi(\gamma)}-1.$$
Put $\BB_{F_0}=\Frac(\AA_{F_0})=\AA_{F_0}\big[\frac{1}{p}\big]$ and let $\BB_{F_0}^{\ur}$ the maximal unramified extension of $\BB_{F_0}$: this is a DVF with uniformizer $p$ and whose residue field $\EE$ is a separable closure of $\EE_{F_0}$. We have an injective morphism $\AA_{F_0}\to\widetilde{\AA}:=\W(C^\flat)$ sending $\varpi$ to $[\varepsilon]-1$. It extends into an injective map $\AA\to\widetilde{\AA}$, where $\AA$ is the $p$-adic completion of the ring of integers of $\BB_{F_0}^{\ur}$. There is a Frobenius operator $\varphi$ and a action of $G_{F_0}$ on $\BB=\AA\big[\frac{1}{p}\big]$ so that the previous map is $G_{F_0}$-equivariant and compatible with Frobenius. Moreover, there is an isomorphism $H_{F_0}\simeq\Gal(\BB/\BB_{F_0})$. As $H_K\leq H_{F_0}$, we put $\BB_K=\BB^{H_K}$: this is a finite Galois extension of $\BB_0$, and there exists an element $\varpi_K\in\BB_K$ such that its ring of integers
$$\AA_K=\Big\{\sum\limits_{n\in\ZZ}a_n\varpi_K^n\,;\,(\forall n\in\ZZ)\,a_n\in\W(k^\prime)\,\lim\limits_{n\to-\infty}a_n=0\Big\}$$
where $k^\prime$ is the residue field of $K_\infty$. This is a Cohen ring for the field $\EE^{H_K}=:\EE_K=k^\prime(\!(\overline{\varpi}_K)\!)$.

By \cite{Font90}, the functor $T\mapsto(\AA\otimes_{\ZZ_p}T)^{H_K}$ induces an equivalence between $\Rep_{\ZZ_p}(G_K)$ and the category $\Mod_{\AA_K}^{\et}(\varphi,\Gamma)$ of \'etale $(\varphi,\Gamma)$-modules over $\AA_K$, whose objects are $\AA_K$-modules of finite type endowed with commuting and semi-linear Frobenius operator and $\Gamma_K$-action, such that linearization of the Frobenius operator is an isomorphism. By inverting $p$, there is a similar equivalence between the isogeny categories.

The multivariate generalization of Fontaine result was proved in \cite{Z1}: let $\AA_{K,\Delta}$ be the $p$-adic completion of the tensor product $\AA_K\otimes_{\ZZ_p}\cdots\otimes_{\ZZ_p}\AA_K$ where the copies of $\AA_K$ are indexed by $\Delta$ (the copy of $\AA_K$ of index $\alpha\in\Delta$ will be denoted $\AA_{K,\alpha}$), and $\BB_{K,\Delta}=\AA_{K,\Delta}\big[\frac{1}{p}\big]$. We have $\AA_{K,\Delta}/(p)=:\EE_{K,\Delta}=\EE_K\otimes_{\FF_p}\cdots\otimes_{\FF_p}\EE_K$ (where the copies of $\EE_K$ are indexed by $\Delta$). For each $\alpha\in\Delta$, let $\varphi_\alpha$ and $\Gamma_{K,\alpha}=\iota_\alpha(\Gamma_K)$ denote the actions of $\varphi$ and $\Gamma_K$ on the factor of index $\alpha$ fixing the other factors. Denote by $\varphi_\Delta$ the monoid generated by the $\varphi_\alpha$ for $\alpha\in\Delta$. There exists a multivariate analogues of $\AA$ and $\BB$: let $\AA_\Delta$ (resp. $\BB_\Delta$) be the $p$-adic completion of $\varinjlim\limits_F\AA_{F,\Delta}$ where $F$ runs over the finite subextensions of $\Kbar/K$ (resp. $\BB_\Delta=\AA_\Delta\big[\frac{1}{p}\big]$). These rings are endowed with commuting actions of $\varphi_\Delta$ and $G_{K,\Delta}$, and $\AA_\Delta/(p)\simeq\varinjlim\limits_F\EE_{F,\Delta}$. There is an equivalence of categories
\begin{align*}
\Rep_{\ZZ_p}(G_{K,\Delta}) &\to \Mod_{\AA_{K,\Delta}}^{\et}(\varphi_\Delta,\Gamma_{K,\Delta})\\
T &\mapsto (\AA_\Delta\otimes_{\ZZ_p}T)^{H_{K,\Delta}}
\end{align*}
where $\Rep_{\ZZ_p}(G_{K,\Delta})$ is the category of $\ZZ_p$-modules of finite type endowed with a continuous linear action of $G_{K,\Delta}$, and $\Mod_{\AA_{K,\Delta}}^{\et}(\varphi_\Delta,\Gamma_{K,\Delta})$ the category of \'etale $(\varphi_\Delta,\Gamma_{K,\Delta})$-modules over $\AA_{K,\Delta}$, whose objects are finitely generated projective $\AA_{K,\Delta}$-modules with commuting semilinear actions of the $\varphi_\alpha$ for $\alpha\in\Delta$ and $\Gamma_{K,\Delta}$, and such that the linearization of $\varphi_\alpha$ is an isomorphism for all $\alpha\in\Delta$ (\cf \cite[\S2.3]{CKZ} and \cite[Theorem 4.1]{CKZ}). By inverting $p$, there is a similar equivalence $\Rep_{\QQ_p}(G_{K,\Delta})\to\Mod_{\BB_{K,\Delta}}^{\et}(\varphi_\Delta,\Gamma_{K,\Delta})$ between the corresponding isogeny categories.

On the other hand, Fontaine result was refined by Cherbonnier-Colmez as follows. Let $v^\flat$ be the valuation on $C^\flat$ normalized by $v^\flat(\widetilde{p})=1$. If $r\in\QQ_{>0}$, let $\widetilde{\AA}^{(0,r]}\subset\widetilde{\AA}$ be the subset made of those elements $z=\sum\limits_{m=0}^\infty p^m[z_m]$ (with $(z_m)_{m\in\NN}\in(C^\flat)^{\NN}$) such that $\lim\limits_{m\to\infty}rv^\flat(z_m)+m=+\infty$. Recall there are embeddings $\AA_K\subset\AA\hookrightarrow\widetilde{\AA}$: put $\AA_K^{(0,r]}=\AA_K\cap\widetilde{\AA}^{(0,r]}$ and $\AA^{(0,r]}=\AA\cap\widetilde{\AA}^{(0,r]}$. Inverting $p$, we define analogues $\BB_K^{(0,r]}\subset\BB^{(0,r]}$. The subring of overconvergent elements in $\AA_K$ (resp. $\AA$) is $\AA_K^\dagger:=\bigcup\limits_{r>0}\AA_K^{(0,r]}$ (resp. $\AA^\dagger:=\bigcup\limits_{r>0}\AA^{(0,r]}$). Inverting $p$, we define analogues $\BB_K^\dagger\subset\BB^\dagger$ (these are fields). Those rings are stable under the action of $G_K$, and we have $(\AA^{(0,r]})^{H_K}=\AA_K^{(0,r]}$ and $(\AA^\dagger)^{H_K}=\AA_K^\dagger$. Moreover, we have $\varphi(\AA^{(0,r]})\subset\AA^{(0,r/p]}$, so that $\AA^\dagger$ and $\AA_K^\dagger$ are stable under $\varphi$. Defining $\Mod_{\AA_K^\dagger}^{\et}(\varphi,\Gamma)$ similarly as $\Mod_{\AA_K}^{\et}(\varphi,\Gamma)$, we have:
\begin{theo}\label{theoCC}
(\cite[Proposition III.5.1 \& corollaire III.5.2]{CC}) The functor
\begin{align*}
\Rep_{\ZZ_p}(G_K) &\to \Mod_{\AA_K^\dagger}^{\et}(\varphi,\Gamma)\\
T &\mapsto (\AA^\dagger\otimes_{\ZZ_p}T)^{H_K}\\
\end{align*}
is an equivalence of categories.
\end{theo}
This result was extended to the multivariate case in \cite{PZ} and \cite{CKZ}. If $r\in\RR_{>0}$, let $\AA_{F_0,\Delta}^{(0,r]}\subset\AA_{F_0,\Delta}$ be the subring made of those elements $\sum\limits_{\underline{n}\in\ZZ^\Delta}(a_{n_1}\varpi_{\alpha_1}^{n_1})\otimes\cdots\otimes(a_{n_\delta}\varpi_{\alpha_\delta}^{n_\delta})$ (with $(a_{n_1},\ldots,a_{n_\delta})\in\W(k)^\Delta$) such that
$$\lim\limits_{\abs{\underline{n}}\to\infty}v_p(a_{n_1}\cdots a_{n_\delta})+\tfrac{rp}{p-1}\min\{n_1,\ldots,n_\delta\}=+\infty.$$
Put $\AA_{F_0,\Delta}^\dagger=\bigcup\limits_{r\in\RR_{>0}}\AA_{F_0,\Delta}^{(0,r]}\subset\AA_{F_0,\Delta}$ and $\BB_{F_0,\Delta}^\dagger=\AA_{F_0,\Delta}^\dagger\big[\frac{1}{p}\big]$. The subrings $\AA_{F_0,\Delta}^\dagger$ and $\BB_{F_0,\Delta}^\dagger$ of $\BB_{F_0,\Delta}$ are stable by $\varphi_\Delta$ and $\Gamma_{F_0,\Delta}$. In this context, the analogue of $\AA^\dagger$ is constructed as follows. For a finite subextension $F$ of $\Kbar/F_0$ and $r\in\QQ_{>0}$, we put $\AA_{F,\Delta,\circ}^{(0,r]}=\AA_F^{(0,r]}\otimes_{\ZZ_p}\cdots\otimes_{\ZZ_p}\AA_F^{(0,r]}$ (where the copies of $\AA_F^{(0,r]}$ are indexed by $\Delta$), and $\AA_{F,\Delta}^{(0,r]}=\AA_{F_0,\Delta}^{(0,r]}\otimes_{\AA_{F_0,\Delta,\circ}^{(0,r]}}\AA_{F,\Delta,\circ}^{(0,r]}$. We have $\AA_{\Delta,\circ}^{(0,r]}=\varinjlim\limits_F\AA_{F,\Delta,\circ}^{(0,r]}=\AA^{(0,r]}\otimes_{\ZZ_p}\cdots\otimes_{\ZZ_p}\AA^{(0,r]}$. Put $\AA_\Delta^{(0,r]}=\AA_{F_0,\Delta}^{(0,r]}\otimes_{\AA_{F_0,\Delta,\circ}^{(0,r]}}\AA_{\Delta,\circ}^{(0,r]}$ and $\AA_\Delta^\dagger=\bigcup\limits_{r>0}\AA_\Delta^{(0,r]}$. Inverting $p$ we get rings $\BB_\Delta^{(0,r]}\subset\BB_\Delta^\dagger$. These rings admit an action of $G_{K,\Delta}$. By \cite[Lemma 3.2.1]{PZ}, we have $(\AA_\Delta^\dagger)^{H_{K,\Delta}}=\AA_{K,\Delta}^\dagger=:\bigcup\limits_{r>0}\AA_{K,\Delta}^{(0,r]}$. Moreover, if $\alpha\in\Delta$, we have an operator $\varphi_\alpha\colon\AA_\Delta^{(0,r]}\to\AA_\Delta^{(0,r/p]}$, so that there is an action of $\varphi_\Delta$ on $\AA_\Delta^\dagger$ and $\BB_\Delta^\dagger$. The natural map $\AA_\Delta^\dagger\to\AA_\Delta$ (induced by the maps $\AA_{\Delta,\circ}^{(0,r]}\to\AA_\Delta$ extended by $\AA_K^{(0,r]}$-linearity) is $G_{K,\Delta}$ and $\varphi_\Delta$-equivariant. The generalization of theorem \ref{theoCC} is:
\begin{theo}\label{theoPZ}
(\cite[Corollary 3.4.4]{PZ}, \cite[Theorem 6.15]{CKZ}) The functor
\begin{align*}
\D^\dagger\colon\Rep_{\ZZ_p}(G_{K,\Delta}) &\to \Mod_{\AA_{K,\Delta}^\dagger}^{\et}(\varphi_\Delta,\Gamma_{K,\Delta})\\
T &\mapsto (\AA_\Delta^\dagger\otimes_{\ZZ_p}T)^{H_K}\\
\end{align*}
is an equivalence of categories, where the target category is defined similarly as $\Mod_{\AA_{K,\Delta}}^{\et}(\varphi_\Delta,\Gamma_{K,\Delta})$. 
\end{theo}
Moreover, by \cite[Theorem 3.4.2]{PZ} and \cite[Theorem 6.14]{CKZ}, base extension to $\AA_{K,\Delta}$ over $\AA_{K,\Delta}^\dagger$ induces an equivalence of categories
$$\Mod_{\AA_{K,\Delta}^\dagger}^{\et}(\varphi_\Delta,\Gamma_{K,\Delta})\isomto\Mod_{\AA_{K,\Delta}}^{\et}(\varphi_\Delta,\Gamma_{K,\Delta}).$$
\begin{nota}
If $T\in\Rep_{\ZZ_p}(G_{K,\Delta})$ and $r\in\QQ_{>0}$, we put $\D^{(0,r]}(T)=\big(\AA_\Delta^{(0,r]}\otimes_{\ZZ_p}T\big)^{H_{K,\Delta}}$. This is an $\AA_{K,\Delta}^{(0,r]}$-module endowed with an action of $\Gamma_{K,\Delta}$. Moreover, for each $\alpha\in\Delta$, the operator $\varphi_\alpha$ induces a semi-linear map $\D^{(0,r]}(T)\to\D^{(0,r/p]}(T)$.
\end{nota}

\begin{lemm}\label{lemmSurconvr}
Let $T\in\Rep_{\ZZ_p}(G_{K,\Delta})$ be free of rank $d$. There exists $r_T\in\QQ_{>0}$ such that for all $r\in\QQ\cap]0,r_T]$, the natural map
$$\alpha^{(0,r]}(T)\colon\AA_\Delta^{(0,r]}\otimes_{\AA_{K,\Delta}^{(0,r]}}\D^{(0,r]}(T)\to\AA_\Delta^{(0,r]}\otimes_{\ZZ_p}T$$
is an isomorphism and $\D^{(0,r]}(T)$ is projective of rank $d$ over $\AA_{K,\Delta}^{(0,r]}$.
\end{lemm}

\begin{proof}
The $\AA_{K,\Delta}^\dagger$-module $\D^\dagger(T)$ is projective of finite type: we can find a finite set $(a_i)_{1\leq i\leq s}$ generating the unit ideal in $\AA_{K,\Delta}^\dagger$ such that the localization $\D^\dagger(T)_{a_i}$ is free of rank $d=\rg_{\ZZ_p}(T)$ over $(\AA_{K,\Delta}^\dagger)_{a_i}$ for all $i\in\{1,\ldots,s\}$. We can choose $r_T\in\QQ_{>0}$ small enough such that $a_i\in\AA_{K,\Delta}^{(0,r_T]}$ and there exists a $(\AA_{K,\Delta}^\dagger)_{a_i}$-basis $(x_{i,j})_{1\leq j\leq d}$ of $\D^\dagger(T)_{a_i}$ made of elements in $\D^{(0,r_T]}(T)_{a_i}$ for all $i\in\{1,\ldots,s\}$. Put $D_i=\bigoplus\limits_{j=1}^d\big(\AA_{K,\Delta}^{(0,r]}(T)\big)_{a_i}x_{i,j}\subset\D^{(0,r]}(T)_{a_i}$. By theorem \ref{theoPZ}, the natural map
$$\alpha^\dagger(T)\colon\AA_\Delta^\dagger\otimes_{\AA_{K,\Delta}^\dagger}\D^\dagger(T)\to\AA_\Delta^\dagger\otimes_{\ZZ_p}T$$
is an isomorphism. Localizing, this provides an isomorphism $\big(\AA_\Delta^\dagger\big)_{a_i}\otimes_{\AA_{K,\Delta}^\dagger}\D^\dagger(T)\to\big(\AA_\Delta^\dagger\big)_{a_i}\otimes_{\ZZ_p}T$, which implies that the localization $(\alpha^{(0,r_T]}(T))_{a_i}\colon\big(\AA_\Delta^{(0,r]}\big)_{a_i}\otimes_{\AA_{K,\Delta}^{(0,r]}}\D^{(0,r]}(T)\to\big(\AA_\Delta^{(0,r]}\big)_{a_i}\otimes_{\ZZ_p}T$ is injective. In the basis $(1\otimes x_{i,j})_{1\leq j\leq d}$ and the basis induced by any basis of $T$ over $\ZZ_p$, this isomorphism is given by a matrix $M_i\in\GL_d\big(\big(\AA_\Delta^\dagger\big)_{a_i}\big)$. Shrinking $r_T$ further if necessary, we may assume that $M_i\in\GL_d\big(\big(\AA_\Delta^{(0,r_T]}\big)_{a_i}\big)$. This means that the composite map
$$\big(\AA_\Delta^{(0,r_T]}\big)_{a_i}\otimes_{(\AA_{K,\Delta}^{(0,r_T]})_{a_i}}D_i\to\big(\AA_\Delta^{(0,r_T]}\big)_{a_i}\otimes_{\AA_{K,\Delta}^{(0,r_T]}}\D^{(0,r_T]}(T)\xrightarrow{(\alpha^{(0,r_T]}(T))_{a_i}}\big(\AA_\Delta^{(0,r_T]}\big)_{a_i}\otimes_{\ZZ_p}T$$
is an isomorphism. This implies that $(\alpha^{(0,r_T]}(T))_{a_i}$ is surjective, hence an isomorphism, so that the map $\big(\AA_\Delta^{(0,r_T]}\big)_{a_i}\otimes_{\AA_{K,\Delta}^{(0,r_T]}}D_i\to\big(\AA_\Delta^{(0,r_T]}\big)_{a_i}\otimes_{\AA_{K,\Delta}^{(0,r_T]}}\D^{(0,r_T]}(T)$ is an isomorphism as well. Taking invariants under $H_{K,\Delta}$ implies that the map $D_i\to\big(\AA_{K,\Delta}^{(0,r_T]}\big)_{a_i}\otimes_{\AA_{K,\Delta}^{(0,r_T]}}\D^{(0,r_T]}(T)$ is an isomorphism. As this holds for all $i\in\{1,\ldots,s\}$, this shows that $\alpha^{(0,r_T]}(T)$ is an isomorphism and that $\D^{(0,r_T]}(T)$ is projective of rank $d$ over $\AA_{K,\Delta}^{(0,r_T]}$. This implies that the similar statement holds when $r_T$ is replaced by any $r\in\QQ\cap]0,r_T]$.
\end{proof}

If $n\in\NN_{>0}$ and $r\geq r_n:=\frac{1}{(p-1)p^{n-1}}$, there is an injective map $i_n\colon\widetilde{\AA}^{(0,r]}\to\B_{\dR}^+$ such that $i_n(\varpi)=[\varepsilon^{1/p^n}]-1=\varepsilon^{(n)}\exp(t/p^n)-1$ and $i_n(\AA_K^{(0,r]})\subset K_n[\![t]\!]$ (\cf \cite[Proposition III.2.1]{CC2}). The tensor product of these maps, indexed by $\Delta$, induces a ring homomorphism
$$i_{n,\Delta}\colon\AA_{\Delta,\circ}^{(0,r]}\to\bigotimes\limits_{\alpha\in\Delta}\B_{\dR}^+\to\B_{\dR,\Delta}^+$$
(note that it is not injective in general, since the tensor product in the LHS is taken over $\ZZ_p$ whereas it is  taken over $F_0$ in the RHS). It restricts into a ring homomorphism
$$i_{n,\Delta}\colon\AA_{K,\Delta,\circ}^{(0,r]}\to\bigotimes\limits_{\alpha\in\Delta}K_n[\![t_\alpha]\!]\to\l_{\dR,\Delta}^+.$$

\begin{lemm}\label{lemmextinAf0Delta}
If $r>\max\{r_n,\delta p^{-n}\}$, the restriction of $i_{n,\Delta}$ to $\AA_{F_0,\Delta,\circ}^{(0,r]}$ extends into a map $\AA_{F_0,\Delta}^{(0,r]}\to F_{0,\infty}[\![t_\alpha]\!]\to\l_{\dR,\Delta}^+$.
\end{lemm}

\begin{proof}
This is checked modulo $\Fil^{s+1}\l_{\dR,\Delta}^+$ for all $s\in\NN$. If $x=\sum\limits_{\underline{n}\in\ZZ^\Delta}(a_{n_1}\varpi^{n_1})\otimes\cdots\otimes(a_{n_\delta}\varpi^{n_\delta})\in\AA_{F_0,\Delta}^{(0,r]}$, we have to check that the series $i_{n,\Delta}(x)=\sum\limits_{\underline{n}\in\ZZ^\Delta}a_{n_1}\cdots a_{n_\delta}\prod\limits_{i=1}^\delta\big(\varepsilon^{(n)}\exp(t_{\alpha_i}/p^n)-1\big)^{n_i}$ converges in $F_{0,\infty}[t_\alpha]_{\alpha\in\Delta}/(t_\alpha)_{\alpha\in\Delta}^{s+1}$ for the $p$-adic topology. For each $\underline{n}\in\ZZ^\Delta$, we have
\begin{align*}
\prod\limits_{i=1}^\delta\Big(\varepsilon^{(n)}\exp(t_{\alpha_i}/p^n)-1\Big)^{n_i} &= \prod\limits_{i=1}^\delta\Big(\sum\limits_{k_i=0}^{n_i}(-1)^{n_i-k_i}\big(\varepsilon^{(n)}\big)^{k_i}\tbinom{n_i}{k_i}\exp(k_it_{\alpha_i}/p^n)\Big)\\
&= \prod\limits_{i=1}^\delta\Big((\varepsilon^{(n)}-1)^{n_i}+\sum\limits_{k_i=0}^{n_i}(-1)^{n_i-k_i}\big(\varepsilon^{(n)}\big)^{k_i}\tbinom{n_i}{k_i}(\exp(k_it_{\alpha_i}/p^n)-1)\Big)
\end{align*}
If $\underline{m}=(m_1,\ldots,m_\delta)\in\NN^\Delta$, the coefficient of $t_{\alpha_1}^{m_1}\cdots t_{\alpha_\delta}^{m_\delta}$ in the latter expression is
$$c_{n,\underline{m}}:=(\varepsilon^{(n)}-1)^{\sum\limits_{\substack{1\leq i\leq\delta\\ m_i=0}}n_i}\prod\limits_{\substack{1\leq i\leq\delta\\ m_i\neq0}}\Big(\sum\limits_{k_i=0}^{n_i}(-1)^{n_i-k_i}\big(\varepsilon^{(n)}\big)^{k_i}\tbinom{n_i}{k_i}\big(\tfrac{k_i}{p^n}\big)^{m_i}\Big)$$
whose valuation is larger that $\frac{1}{(p-1)p^{n-1}}\Big(\sum\limits_{\substack{1\leq i\leq\delta\\ m_i=0}}n_i\Big)-n\abs{\underline{m}}$. The coefficient of $t_{\alpha_1}^{m_1}\cdots t_{\alpha_\delta}^{m_\delta}$ in $i_{n,\Delta}(x)$ is the sum of the series $\sum\limits_{\underline{n}\in\ZZ^\Delta}a_{n_1}\cdots a_{n_\delta}c_{n,\underline{m}}$. If $\underline{n}\in\ZZ^\Delta$, we have $v_p(a_{n_1}\cdots a_{n_\delta}c_{n,\underline{m}})\geq\frac{1}{(p-1)p^{n-1}}\Big(\sum\limits_{\substack{1\leq i\leq\delta\\ m_i=0}}n_i\Big)-n\abs{\underline{m}}$ if $\mu_{\underline{n}}:=\min\{n_1,\ldots,n_\delta)\geq0$, and this goes to $+\infty$ when $\abs{\underline{n}}\to\infty$. Assume $\mu_{\underline{n}}<0$: we have
\begin{align*}
v_p(a_{n_1}\cdots a_{n_\delta}c_{n,\underline{m}}) &= v_p(a_{n_1}\cdots a_{n_\delta})+\tfrac{rp}{p-1}\mu_{\underline{n}}+v_p(c_{n,\underline{m}})-\tfrac{rp}{p-1}\mu_{\underline{n}}\\
&\geq v_p(a_{n_1}\cdots a_{n_\delta})+\tfrac{rp}{p-1}\mu_{\underline{n}}+\tfrac{1}{(p-1)p^{n-1}}\sum\limits_{\substack{1\leq i\leq\delta\\ m_i=0}}n_i-n\abs{\underline{m}}-\tfrac{rp}{p-1}\mu_{\underline{n}}\\
&\geq v_p(a_{n_1}\cdots a_{n_\delta})+\tfrac{rp}{p-1}\mu_{\underline{n}}+\tfrac{1}{(p-1)p^{n-1}}\delta\mu_{\underline{n}}-n\abs{\underline{m}}-\tfrac{rp}{p-1}\mu_{\underline{n}}\\
&\geq v_p(a_{n_1}\cdots a_{n_\delta})+\tfrac{rp}{p-1}\mu_{\underline{n}}-n\abs{\underline{m}}+\tfrac{p}{p-1}\big(\tfrac{\delta}{p^n}-r\big)\mu_{\underline{n}}\geq v_p(a_{n_1}\cdots a_{n_\delta})+\tfrac{rp}{p-1}\mu_{\underline{n}}-n\abs{\underline{m}}.
\end{align*}
As $\lim\limits_{\abs{\underline{n}}\to\infty}v_p(a_{n_1}\cdots a_{n_\delta})+\tfrac{rp}{p-1}\min\{n_1,\ldots,n_\delta\}=+\infty$ by hypothesis, this shows that the series indeed converges.
\end{proof}

If $r>\max\{r_n,\delta p^{-n}\}$, lemma \ref{lemmextinAf0Delta} implies that $\B_{\dR,\Delta}^+$ is equipped with a $\AA_{F_0,\Delta}^{(0,r]}$-algebra structure: the map $i_{n,\Delta}\colon\AA_{\Delta,\circ}^{(0,r]}\to\B_{\dR,\Delta}^+$ extends into a ring homomorphism
$$i_{n,\Delta}\colon\AA_\Delta^{(0,r]}\to\B_{\dR,\Delta}^+.$$

\begin{theo}\label{theocompphiGammaDdif}
Let $T\in\Rep_{\ZZ_p}(G_{K,\Delta})$. If $r\in\QQ_{>0}$ is small enough, there is a $\Gamma_{K,\Delta}$-equivariant isomorphism of $\l_{\dR,\Delta}^+$-modules
$$\l_{\dR,\Delta}\otimes_{\AA_{K,\Delta}^{(0,r]}}\D^{(0,r]}(T)\isomto\D_{\dif}\big(T\big[\tfrac{1}{p}\big]\big).$$
\end{theo}

\begin{proof}
By lemma \ref{lemmSurconvr}, there exists $r_T\in\QQ_{>0}$ such that for all $r\in\QQ\cap]0,r_T]$, the natural map
$$\AA_\Delta^{(0,r]}\otimes_{\AA_{K,\Delta}^{(0,r]}}\D^{(0,r]}(T)\to\AA_\Delta^{(0,r]}\otimes_{\ZZ_p}T$$
is a $G_{K,\Delta}$-equivariant isomorphism. Take $n\in\NN_{>0}$ large enough such that $\max\{r_n,\delta p^{-n}\}\leq r_T$, and assume that $r\in\QQ_{>0}$ is such that $\max\{r_n,\delta p^{-n}\}\leq r\leq r_T$. Extending the scalars to $\B_{\dR,\Delta}^+$ via $i_{n,\Delta}$ provides a $G_{K,\Delta}$-equivariant isomorphism
$$\B_{\dR,\Delta}^+\otimes_{\AA_{K,\Delta}^{(0,r]}}\D^{(0,r]}(T)\to\B_{\dR,\Delta}^+\otimes_{\ZZ_p}T\simeq\B_{\dR,\Delta}^+\otimes_{\l_{\dR,\Delta}^+}\D_{\dif}^+(V)$$
where $V=T\big[\frac{1}{p}\big]\in\Rep_{\QQ_p}(G_{K,\Delta})$. Taking invariants under $H_{K,\Delta}$ gives a $\Gamma_{K,\Delta}$-equivariant isomorphism
$$\L_{\dR,\Delta}^+\otimes_{\AA_{K,\Delta}^{(0,r]}}\D^{(0,r]}(T)\isomto\L_{\dR,\Delta}^+\otimes_{\l_{\dR,\Delta}^+}\D_{\dif}^+(V).$$
Applying the functor $X\mapsto X_{\free}$ provides an isomorphism
$$\big(\L_{\dR,\Delta}^+\otimes_{\AA_{K,\Delta}^{(0,r]}}\D^{(0,r]}(T)\big)_{\free}\isomto\D_{\dif}^+(V)$$
and it remains to show that $\big(\L_{\dR,\Delta}^+\otimes_{\AA_{K,\Delta}^{(0,r]}}\D^{(0,r]}(T)\big)_{\free}=\l_{\dR,\Delta}^+\otimes_{\AA_{K,\Delta}^{(0,r]}}\D^{(0,r]}(T)$. As $\D^{(0,r]}(T)$ is projective of finite rank over $\AA_{K,\Delta}^{(0,r]}$, it is enough to show that $\D^{(0,r]}(T)$ is mapped to $\D_{\dif}^+(V)$ by $i_{n,\Delta}$. By \cite[Lemma 3.2.4]{PZ}, we have $\D^\dagger(T)=\AA_{F_0,\Delta}^\dagger\otimes_{\AA_{F_0,\Delta,\circ}^\dagger}\big(\AA_{\Delta,\circ}^\dagger\otimes_{\ZZ_p}T\big)^{H_{K,\Delta}}$: similarly, we have $\D^{(0,r]}(T)=\AA_{F_0,\Delta}^{(0,r]}\otimes_{\AA_{F_0,\Delta,\circ}^{(0,r]}}\big(\AA_{\Delta,\circ}^{(0,r]}\otimes_{\ZZ_p}T\big)^{H_{K,\Delta}}$, so we are reduced to check that $\big(\AA_{\Delta,\circ}^{(0,r]}\otimes_{\ZZ_p}T\big)^{H_{K,\Delta}}$ maps to $\D_{\dif}^+(V)$ by $i_{n,\Delta}$. Working componentwise, this follows from \cite[Proposition 5.7]{BE}.
\end{proof}


\begin{thebibliography}{10}

\bibitem{And06}
F.~{Andreatta}.
\newblock {Generalized ring of norms and generalized \((\varphi ,\Gamma
  )\)-modules}.
\newblock {\em {Annales Scientifiques de l'\'Ecole Normale Sup\'erieure.
  Quatri\`eme S\'erie}}, 39(4):599--647, 2006.

\bibitem{AB}
F.~{Andreatta} and O.~{Brinon}.
\newblock {Surconvergence des repr\'esentations \(p\)-adiques: le cas relatif}.
\newblock In {\em Repr\'esentation \(p\)-adiques de groupes \(p\)-adiques I.
  Repr\'esentations galoisiennes et \((\varphi, \Gamma)\)-modules}, volume 319
  of {\em Ast\'erisque}, pages 39--116. Soci\'et\'e Math\'ematique de France,
  2008.

\bibitem{AB2}
F.~{Andreatta} and O.~{Brinon}.
\newblock {B\(_{\text{dR}}\)-repr\'esentations dans le cas relatif}.
\newblock {\em {Annales Scientifiques de l'\'Ecole Normale Sup\'erieure.
  Quatri\`eme S\'erie}}, 43(2):279--339, 2010.

\bibitem{Ax}
J.~{Ax}.
\newblock {Zeros of polynomials over local fields. The Galois action}.
\newblock {\em {J. Algebra}}, 15:417--428, 1970.

\bibitem{BE}
L.~{Berger}.
\newblock Repr\'esentations $p$-adiques et \'equations diff\'erentielles.
\newblock {\em Invent. Math.}, 148:219--284, 2002.

\bibitem{BC}
L.~{Berger} and P.~{Colmez}.
\newblock {Familles de repr\'esentations de de Rham et monodromie
  \(p\)-adique}.
\newblock In {\em Repr\'esentation \(p\)-adiques de groupes \(p\)-adiques I.
  Repr\'esentations galoisiennes et \((\varphi, \Gamma)\)-modules}, pages
  303--337. Soci\'et\'e Math\'ematique de France, 2008.

\bibitem{Algebre10}
N.~{Bourbaki}.
\newblock {\em {\'El\'ements de math\'ematiques. Alg\`ebre. Chapitre 10 -
  Alg\`ebre homologique}}.
\newblock Masson, 1980.

\bibitem{Bri2006}
O.~{Brinon}.
\newblock {Repr\'esentations cristallines dans le cas d'un corps~r\'esiduel
  imparfait}.
\newblock {\em Annales de l'Institut Fourier}, 56(4):919--999, 2006.

\bibitem{B1}
O.~{Brinon}.
\newblock {\em {Repr\'esentations $p$-adiques cristallines et de de Rham dans
  le cas relatif}}, volume 112 of {\em M\'emoires de la Soci\'et\'e
  Math\'ematique de France (N.S.)}.
\newblock {Soci\'et\'e Math\'ematique de France}, 2008.

\bibitem{perfectoid}
O.~{Brinon}, F.~{Andreatta}, R.~{Brasca}, B.~{Chiarellotto}, N.~{Mazzari},
  S.~{Panozzo}, and M.~{Seveso}.
\newblock {An introduction to perfectoid spaces}.
\newblock In {\em An excursion into \(p\)-adic Hodge theory: from foundations
  to recent trends}, pages 207--265. Paris: Soci\'et\'e Math\'ematique de
  France, 2019.

\bibitem{Hawaii}
O.~{Brinon} and B.~{Conrad}.
\newblock {CMI summer school notes on $p$-adic Hodge theory}.
\newblock Lecture notes of courses given at the University of Hawaii at Manoa,
  Honolulu, Hawaii, June 15 - July 10, 2009.

\bibitem{CKZ}
A.~{Carter}, K.~{Kedlaya}, and G.~{Z\'abr\'adi}.
\newblock {Drinfeld's lemma for perfectoid spaces and overconvergence of
  multivariate $(\varphi,\Gamma)$ -modules}.
\newblock {\em Doc. Math.}, 26:1329--1393, 2021.

\bibitem{CC}
F.~{Cherbonnier} and P.~{Colmez}.
\newblock {Repr\'esentations $p$-adiques surconvergentes}.
\newblock {\em {Inventiones Mathematicae}}, 133(3):581--611, 1998.

\bibitem{CC2}
F.~{Cherbonnier} and P.~{Colmez}.
\newblock {Th\'eorie d'Iwasawa des repr\'esentations $p$-adiques d'un corps
  local}.
\newblock {\em {Journal of the American Mathematical Society}}, 12:241--268,
  1999.

\bibitem{Font90}
J.-M. {Fontaine}.
\newblock {Repr\'esentations $p$-adiques des corps locaux. I}.
\newblock In {\em The Grothendieck Festschrift, Vol. II}, volume~87 of {\em
  Progress in Mathematics}, pages 249--309. {Birkh\"auser}, 1990.

\bibitem{Font04}
J.-M. {Fontaine}.
\newblock {Arithm\'etiques des repr\'esentations $p$-adiques}.
\newblock In {\em Cohomologies $p$-adiques et applications arithm\'etiques
  III}, volume 295 of {\em Ast\'erisque}, pages 1--115. Soci\'et\'e
  Math\'ematique de France, 2004.

\bibitem{Matlis}
E.~{Matlis}.
\newblock {The higher properties of $R$-sequences}.
\newblock {\em {Journal of algebra}}, 50:77--112, 1978.

\bibitem{NSW}
J.~{Neukirch}, A.~{Schmidt}, and K.~{Wingberg}.
\newblock {\em {Cohomology of number fields}}, volume 323 of {\em {Grundlehren
  der Mathematischen Wissenschaften}}.
\newblock Springer, 2008.

\bibitem{Olsson}
M.~{Olsson}.
\newblock {On Faltings' method of almost \'etale extensions}.
\newblock In {\em Algebraic geometry, Seattle 2005. Proceedings of the 2005
  Summer Research Institute, Seattle, WA, USA, July 25--August 12, 2005}, pages
  811--936. American Mathematical Society, 2009.

\bibitem{PZ}
A.~{Pal} and G.~{Z\'abr\'adi}.
\newblock {Cohomology and overconvergence for representations of powers of
  Galois groups}.
\newblock {\em {Journal of the Institute of Mathematics of Jussieu}},
  20(2):361--421, 2021.

\bibitem{SW}
P.~{Scholze} and J.~{Weinstein}.
\newblock {\em {Berkeley Lectures on $p$-adic Geometry}}, volume 207 of {\em
  Annals of Mathematics Studies}.
\newblock Princeton University Press, 2020.

\bibitem{Sen}
S.~{Sen}.
\newblock {Continuous cohomology and $p$-adic Galois representations}.
\newblock {\em {Invent. Math.}}, 62:89--116, 1980.

\bibitem{Serre97}
J.-P. {Serre}.
\newblock {\em {Galois cohomology}}.
\newblock Springer Monographs in Mathematics. Springer, 1997.

\bibitem{Tate}
J.~T. {Tate}.
\newblock {\(p\)-divisible groups}.
\newblock In {\em {Proceedings of a conference on local fields, NUFFIC Summer
  School, Driebergen, 1966}}, pages 158--183. Springer, 1967.

\bibitem{W}
J.~{Weinstein}.
\newblock {$\Gal(\overline{\QQ}_p/\QQ_p)$ as a geometric fundamental group}.
\newblock {\em IMRN}, 2017:2964--2997, 2017.

\bibitem{Z1}
G.~{Z\'abr\'adi}.
\newblock {Multivariable $(\varphi,\Gamma)$-modules and products of Galois
  groups}.
\newblock {\em Math. Res. letters}, 25:687--721, 2018.

\bibitem{Z2}
G.~{Z\'abr\'adi}.
\newblock {Multivariable $(\varphi,\Gamma)$ modules and smooth $o$-torsion
  representations}.
\newblock {\em Selecta Math.}, 24:935--995, 2018.

\end{thebibliography}

\end{document}